\documentclass[11pt]{amsart}

\usepackage{amsmath, amssymb, mathtools}

\usepackage{hyperref}

\newcounter{my_enumerate_counter}
\newcommand{\pushcounter}{\setcounter{my_enumerate_counter}{\value{enumi}}}
\newcommand{\popcounter}{\setcounter{enumi}{\value{my_enumerate_counter}}}

\DeclareMathOperator{\eq}{eq}
\DeclareMathOperator{\meager}{{\tt meager}}
\newcommand{\cstar}{$\mathrm{C}^*$}
\DeclareMathOperator{\Plm}{P_{\text{$\leq m$}}}
\DeclareMathOperator{\Plmm}{P_{\text{$\leq m-1$}}}
\DeclareMathOperator{\Pgm}{P_{\text{$>m$}}}
\DeclareMathOperator{\Pgmm}{P_{\text{$>m+1$}}}
\DeclareMathOperator{\Succ}{Succ}
\DeclareMathOperator{\Th}{Th}
\DeclareMathOperator{\dom}{dom}

\DeclareMathOperator{\ee}{Elm}
\DeclareMathOperator{\range}{range}
\DeclareMathOperator{\cov}{cov}

\DeclareMathOperator{\tp}{tp}

\newcommand{\udt}{uniformly definable by a sequence of types}

\newcommand{\cN}{\mathcal N}
\newcommand{\cM}{\mathcal M}

\newcommand{\cV}{\mathcal V}

\newcommand{\forces}{\Vdash}

\newcommand{\bt}{\mathbf t} 
\newcommand{\bs}{\mathbf s}

\newcommand{\bbF}{{\mathbb F}}

\newcommand{\bfr}{\mathbf r}
\newcommand{\bfs}{\mathbf s}
\newcommand{\bfC}{\mathbf C}
\newcommand{\bfD}{\mathbf D}
\newcommand{\bfE}{\mathbf E}
\newcommand{\bfT}{\mathbf T}
\DeclareMathOperator{\bfTT}{\mathbf T_{trees}}
\newcommand{\bfF}{\mathbf F}

\newcommand{\bbQ}{\mathbb Q}
\newcommand{\bbR}{\mathbb R}
\newcommand{\bbZ}{\mathbb Z}

\newcommand{\calL}{L}

\newcommand{\cZ}{{\mathcal Z}}

\newcommand{\fM}{\mathfrak M}

\newcommand{\rs}{\restriction}

\newcommand{\cT}{\mathcal T}

\newcommand{\frW}{\mathfrak W}

\newcommand{\hM}{\hat{\mathcal M}}
\newcommand{\fp}{\mathfrak p} 
\newcommand{\bSigma}{\mathbf\Sigma}
\newcommand{\bPi}{\mathbf\Pi}
\newcommand{\bDelta}{\mathbf\Delta}

\newcommand{\md}{\mathbf d}
\newcommand{\calD}{\mathcal D}
\newcommand{\bbP}{\mathbb P}

\newcommand{\bbPT}{\mathbb P_{\bfT}}
\newcommand{\bbPTS}{\mathbb P_{\bfT,\Sigma}}
\newcommand{\bbPTSM}{\mathbb P_{\bfT,\Sigma,\fM}}
\newcommand{\bbPTM}{\mathbb P_{\bfT,\fM}}
\newcommand{\bbFTS}{\mathbb F_{\bfT,\Sigma}}

\newcommand{\e}{\varepsilon}
\newtheorem{thm}{Theorem}[section]
\newtheorem{theorem}{Theorem}
\newtheorem{coro}[thm]{Corollary}

\newtheorem{question}[thm]{Question}
\newtheorem{claim}[thm]{Claim}

\newtheorem{lemma}[thm]{Lemma}

\newtheorem{prop}[thm]{Proposition}
\DeclareMathOperator{\dist}{dist}

\theoremstyle{definition}
\newtheorem{remark}[thm]{Remark}

\newtheorem{definition}[thm]{Definition}
\newtheorem{problem}[thm]{Problem}

\newcommand{\cP}{\mathcal P} 
\newcommand{\Too}{T_2}
\newcommand{\olo}{\omega^{<\omega}}

\newcommand{\oo}{\omega^{\omega}}

\DeclareMathOperator{\rank}{rank}

\newcommand{\dminus}{\dot -}
\title{Omitting types in  logic of metric structures}

\author{Ilijas Farah}

\address{Department of Mathematics and Statistics, York University, 4700 Keele Street, North York, Ontario, Canada, M3J 1P3} 
\email{ifarah@mathstat.yorku.ca}
\urladdr{http://www.math.yorku.ca/~ifarah}

\author{Menachem Magidor}

\address{The Hebrew University of Jerusalem\\
Einstein Institute of Mathematics\\
Edmond J. Safra Campus, Givat Ram\\
Jerusalem 91904, Israel}

\email{mensara@savion.huji.ac.il}

\date{\today}

\begin{document} 

\begin{abstract} 
This paper is about
omitting types in logic of metric structures introduced by Ben Yaacov, Berenstein, Henson and Usvyatsov. 
While a complete type is omissible in some model of a countable  complete theory if and only if it is not principal, this is not true for the incomplete types
by a result of Ben Yaacov.  We prove that there is no simple test for determining whether a type is omissible in a model of a
theory $\bfT$ in a countable language. More precisely, 
 we find a theory in a countable language such that the set of types omissible in some of its models  
is a complete $\bSigma^1_2$ set and a complete theory 
in a countable language such that the set of types omissible in some of its models is a complete $\bPi^1_1$ set. 
Two more unexpected examples are given:  (i) a complete theory $\bfT$ and a countable set of  types such that each of its finite sets is jointly omissible in a model of $\bfT$,  
but the whole set is not and  (ii) a complete theory and two types that are separately omissible, 
but not jointly omissible, in its models.   
%We also introduce the notion of a uniform sequence of types and prove that the above pathologies do not occur for such types. 
\end{abstract} 

\maketitle

The Omitting Types Theorem is one of the most useful methods for constructing models of 
first-order theories with prescribed properties (see \cite{keisler1973forcing}, \cite{hodges2006building}, 
or any general text in model theory). It implies, among other facts,  the following
(here $S_n(\bfT)$ denotes the space of complete $n$-types in theory~$\bfT$). 
\begin{enumerate}
\item\label{I.1.1}   If $\bfT$ is a  theory in  a
countable language, then the set of all $n$-types realized in every model of  $\bfT$ is Borel
in the logic topology on $S_n(\bfT)$. 
\item\label{I.1.2}  If $\bfT$ is in addition complete, then any sequence $\bt_n$, for $n\in \omega$, 
of types each of which can be omitted in a model of $\bfT$ can be simultaneously omitted in 
a model of $\bfT$. 
\pushcounter
\end{enumerate}
Types $\bt_n$ appearing in \eqref{I.1.2}  are not required to be complete, but the theory~$\bfT$~is. 

While the standard omissibility  criterion for a given type  in some model of a given theory 
in classical logic applies regardless of whether the type is complete or not, 
situation in  logic of metric structures is a bit more subtle. 

The omitting types theorem in logic of metric structures (\cite[\S 12]{BYBHU} or \cite[Lecture 4]{Ha:Continuous})
has  following straightforward 
consequences (see Proposition~\ref{P.Borel}  for a proof of \eqref{I3} and
Corollary~\ref{C.Omitting} 
 for a proof of \eqref{I4}).  
 \begin{enumerate}
\popcounter
\item  \label{I3} 
If $\bfT$ is a  theory in  a
countable language of logic of metric structures, 
then the set of all \emph{complete} $n$-types realized in every model of  $\bfT$ is Borel
in the logic topology on $S_n(\bfT)$. 
\item \label{I4} If $\bfT$ is moreover complete, then any sequence $\bt_n$, for $n\in \omega$, 
of \emph{complete} types each of which can be omitted in a model of $\bfT$ can be simultaneously omitted in 
a model of $\bfT$. 
\pushcounter
\end{enumerate}
Examples constructed by 
I. Ben-Yaacov (\cite{BY:Definability}) 
and T. Bice  (\cite{Bice:Brief}) 
demonstrate that omitting partial 
types in logic of metric structures is inherently more complicated than omitting complete types.  
  Our results, expressed using descriptive set theory,  show  
 that the problem of omitting types in logic of metric structures  is essentially intractable.

\begin{theorem} \label{T0} 
\begin{enumerate}
\popcounter
\item \label{T0.1} There is a complete theory $\bfT$ in a countable language 
such that the set of all types omissible over a model of $\bfT$ is 
$\bPi^1_1$-complete. 
\item \label{T0.2} There is a theory $\bfT$ in a countable language such that the set of 
 all types omissible in a model of $\bfT$ is 
$\bSigma^1_2$-complete. 
\item \label{T0.3} There is a separable structure $M$ in a countable language
such that the set of all partial  types
omitted in $M$ is a complete~$\bPi^1_1$ set. 
\pushcounter
\end{enumerate}
\end{theorem} 

\begin{proof} 
\eqref{T0.1} is proved in Theorem~\ref{P1}, \eqref{T0.2} is proved in Theorem~\ref{P2}, 
and \eqref{T0.3} is Corollary~\ref{C-1-1}. 
\end{proof}

%%% THM3
We also show that \eqref{I.1.2} fails in logic of metric structures (as customary in logic, $\omega$ denotes the least infinite ordinal
identified with the set of natural numbers). 

%
%If $\bfT$ is a complete theory and $\cF$ is a set of types each one of which is omissible in a model of $\bfT$, 
%are all types in $\cF$ omitted in a single model of $\bfT$? The case when $\cF$ is finite easily reduces to the case when $\cF$ has two elements by 
%an easy argument  as in \S\ref{S.pairing}, and the latter case is Question~\ref{Q.Simultaneous}. 
%We have some information on the case when $\cF$ is countable. 
%
%%% THM3

\begin{theorem}\label{T1}
There are a  complete  theory $\bfT$ in a countable language 
and  types $\bs_n$, for $n\in \omega$,
such that for every $k$ there exists a model of $\bfT$ that omits
all $\bs_n$ for $n\leq k$ but 
no model of $\bfT$ simultaneously omits all~$\bs_n$. 
\end{theorem}

\begin{proof} This is proved in \S\ref{S.T1}. 
\end{proof} 

Theorem~\ref{T1} should be compared to a consequence of 
 \cite[Corollary~4.7]{ben2009model}:
 under certain  additional conditions 
any countable set of types that are not omissible in a 
model of a complete separable theory $\bfT$  
has a finite subset consisting of types that are  
not omissible in a 
model of  $\bfT$. 

The following gives an another striking example of a peculiar behaviour of types in the logic of metric structures. 

\begin{theorem}\label{T2}
There are a  complete  theory $\bfT$ in a countable language $L$ 
and  types $\bs$ and $\bt$ omissible in models of $\bfT$  
such that no model of $\bfT$ simultaneously omits both 
$\bs$ and $\bt$. \end{theorem}

\begin{proof} This is proved in \S\ref{S.ProofT2}. 
\end{proof} 

Our proof of Theorem~\ref{T2} uses the following  quotable curiosity. 
 
\begin{theorem} \label{C.generic} 
There exists a complete theory $\bfT$ in a countable language 
and a type $\bt(x)$  omissible in a model of $\bfT$, yet  
 forced by   $\bbP_{\bfT}$ to be realized in~$M_G$.\qed
\end{theorem} 

\begin{proof} This is proved at the end of  \S\ref{S.Extension}. 
\end{proof} 

In the course of proving our results, we also prove a somewhat surprising fact that the class trees of height $\omega$ is axiomatizable 
in logic of metric structures (Theorem~\ref{T.T.Ax}). 
Following~\cite{BYBHU} we write 
\[
r\dminus s=\max(0,r-s).    
\]

\subsection*{Acknowledgments}
This work  was initiated  during MM's visit to Toronto in May 2013  
and continued during  the workshop on Descriptive Set Theory at the 
Erwin Schr\"odinger Institute in October 2013, 
during  IF's visit to Jerusalem in November 2013,  
during our visit to Oberwolfach in January 2014, during the Workshop on Model Theory and \cstar-algebras in M\"unster in July 2014
and following IF's visit to Beer Sheva and Jerusalem in May 2015. 

 We would like to thank Chris Eagle and Robin Tucker-Drob for very useful remarks, and to 
 Bradd Hart for a large number of very useful remarks. 
 We are indebted to the anonymous referee for a most  useful seven-page report that
 has greatly helped us improve the exposition and simplify some of the proofs. 

The first author and the second author's visits to Toronto  were
 partially supported by  NSERC.

\section{Preliminaries}\label{S.preliminaries} 

We assume that the reader is acquainted with the logic of metric structures  (\cite{BYBHU}, 
\cite{Ha:Continuous}). We strictly follow the outline of this logic given in \cite{BYBHU}. 
In particular, all metric structures are required to have diameter $1$ and 
all formulas are $[0,1]$-valued. All function and predicate symbols are equipped 
with a fixed modulus of uniform continuity. Every structure is a complete metric space in which 
interpretations of functional and relational symbols respect this modulus.  
It is a straightforward exercise to see that our results apply both to the modification of 
this logic adapted to operator algebras (\cite{FaHaSh:Model2}) and  the  unbounded variations of the logic of metric 
structures. 

For a formula $\varphi(\bar x)$, a structure of the same language $M$ and a tuple $\bar a$ in $M$,  
the interpretation of $\varphi(\bar x)$ at $\bar a$ in $M$ is denoted 
 $\varphi(\bar a)^M$. 
We define the theory of a structure to be the set of sentences
 \[
 \Th(M)=\{\varphi: \varphi^M=0\}. 
 \]
Since every constant scalar function is a formula, 
 the evaluation  functional $\varphi\mapsto \varphi^M$ is uniquely determined by its kernel.  It is therefore possible, and often convenient, 
 to  consider this functional as the theory of~$M$. This is similar to  the treatment in \cite[\S 2.2]{Muenster}, 
with one (inconsequential) difference.    
Formulas in \cite{Muenster}  are $\bbR$-valued and therefore 
form a vector space. In our presentation  the formulas are $[0,1]$-valued and  form  
an affine subspace of the said vector space. 

We shall tacitly use  completeness theorem for the logic of metric structures whenever convenient~(\cite{yaacov2010proof}).

\subsection{Conditions and types}\label{S.general.conditions} \label{S.Conditions}
A \emph{closed condition} is an expression of the form $\varphi(\bar x)=0$ for a formula $\varphi(\bar x)$ and 
 \emph{type} is a set of closed conditions. 
 (Open conditions will be defined after Lemma~\ref{L.CFC} below.)  
 Type $\bt$ is \emph{realized} in a structure~$M$ if there is a tuple $\bar a$ in $M$ of the 
 appropriate sort 
 such that $\varphi(\bar a)^M=0$ for all conditions $\varphi(\bar x)=0$ in $\bt$. 
 %Compactness theorem for logic of metric structures implies that types realized
% in a model of theory~$\bfT$ are exactly  types that are consistent with $\bfT$. 
%A type $\bt$ is \emph{consistent} if it is consistent with $\bfT=\emptyset$. 
 Type $\bt$ is \emph{consistent with theory $\bfT$} if it is realized in a model of $\bfT$. 
By a compactness argument, this 
 is equivalent to every finite subset of $\bt$ being approximately realizable in a model of $\bfT$. 
   If free variables of every  formula appearing in $\bt$ are included in $\{x_0, \dots, x_{n-1}\}$   and $n$
 is minimal with this property,  
 we say that $\bt$ is an \emph{$n$-type}. 
 
 A type is \emph{complete} if it is equal to the set of all closed conditions satisfied 
by a tuple in some metric structure. A fragment of a complete type is an 
\emph{incomplete} type.

Throughout the following discussion $L$, $\bfT$, and $M$  denote  a language, an $L$-theory, and an $L$-structure.  
For simplicity of notation we shall assume that the language $L$ is single-sorted. Generalizations of our results 
 to multi-sorted language are straightforward. 

\subsubsection{Continuous functional calculus} \label{S.CFC} 
Every $\calL$-formula $\varphi(\bar x)$ has a modulus of uniform continuity and 
the set $R_\varphi$ of all possible values of~$\varphi$ in all $\calL$-structures is a compact subset of $[0,1]$ (\cite{BYBHU}). 
If  $\varphi(\bar x)$ is a formula and $f\colon R_\varphi\to [0,1]$  is  continuous then $f(\varphi(\bar x))$ is a formula. 

We shall consider  generalized conditions of the form $\varphi\in K$, where 
$K\subseteq [0,1]$.
Two generalized conditions $\varphi(\bar x)\in K$ and $\psi(\bar x)\in M$ 
are \emph{equivalent} if in every $\calL$-structure $A$ for every $\bar a$ of the appropriate sort
one has $\varphi(\bar a)^A\in K$ if and only if $\psi(\bar a)^A\in M$.

\begin{lemma} \label{L.CFC} Consider $K\subseteq [0,1]$ and a generalized condition $\varphi(\bar x)\in K$.  
\begin{enumerate}
\item If $K$ is closed then   $\varphi(\bar x)\in K$
 is equivalent to a closed condition. 
\item If $K$ is open then  $\varphi(\bar x)\in K$  
is equivalent  to a condition of the form $\psi(\bar x)<1$. 
 \end{enumerate}
\end{lemma} 

\begin{proof} 
 Since $[0,1]$ is compact and metric,  every closed subset $K$ is a zero set of a continuous real-valued function $d(\cdot, K)$. 
  Therefore if $K$ is closed then  we have 
a continuous $f\colon R_\varphi\to [0,1]$ such that 
$f^{-1}(\{0\})=K$ and 
$\varphi(\bar x)\in K$ is equivalent to $\psi(\bar x)=0$ with $\psi(\bar x)=f(\varphi(\bar x))$. 
 If~$K$ is open then we can choose $f$ so that 
  $f^{-1}(\{1\})=[0,1]\setminus K$, and  
$\varphi(\bar x)\in K$ is equivalent to $\psi(\bar x)<1$ with $\psi(\bar x)=f(\varphi(\bar x))$. 
\end{proof} 

Lema~\ref{L.CFC} implies that every set of closed conditions is equivalent to a type, 
and we shall  slightly abuse the language and refer to sets of closed conditions as types.

\subsubsection{Pairing types} 
\label{S.pairing} 

The reader will excuse us for making some easy observations for future reference. 

\begin{lemma} \label{L.pairing} 
If $\bt$ and $\bs$ are types over a consistent and complete theory
then there are types $\bt\wedge \bs$ and $\bt \vee \bs$ such that 
for every  $M\models \bfT$ we have that  
\begin{enumerate}
\item $M$ omits $\bt \vee \bs$ if and only if it omits both $\bt$ and $\bs$, 
\item $M$ omits $\bt \wedge \bs$ if and only if  it omits  at least one of  $\bt$ or $\bs$. 
\end{enumerate}
\end{lemma} 

\begin{proof} By renaming the free variables of $\bs$ if necessary, we may assume that for some $m\geq 1$ and $k\geq 1$ 
 types $\bt$ and $\bs$ are an $m$-type and a $k$-type, respectively,   
in disjoint sets of variables. 

%%%%%%%%%
By adding dummy conditions to $\bt$ and/or $\bs$ we may assume that the types are of the same cardinality $\kappa$. 
Enumerate  conditions in $\bt$ as $\varphi_\xi(\bar x)=0$ for $\xi<\kappa$
and  conditions in $\bs$ as $\psi_\xi(\bar y)=0$ for $\xi<\kappa$. 
For a finite  $F\subseteq \kappa$  let 
\begin{align*}
\theta_F(\bar x, \bar y)&=\max_{\xi\in F} (\varphi_\xi(\bar x), \psi_\xi(\bar y)), \\
\zeta_F(\bar x, \bar y)&=\min\{\max_{\xi\in F} \varphi_\xi(\bar x), \max_{\xi\in F} \psi_\xi(\bar y)\}. 
\end{align*} 
Let  
\begin{align*}
\bt\wedge \bs
&=\{\theta_F: F\subseteq \kappa, F\text{ finite}\}, \\
\bt\vee \bs
&=\{\zeta_F: F\subseteq \kappa, F\text{ finite}\}.
\end{align*}  
Fix $M\models \bfT$ and an $m+k$-tuple $\bar a, \bar b$ in $M$. 
This tuple  realizes $\bt\wedge\bs$ if and only if 
for all $\xi$ we have $\varphi_\xi(\bar a)=0=\psi_\xi(\bar b)$. This is equivalent to 
$\bar a$ realizing~$\bt$ and $\bar b$ realizing $\bs$. 
Since the $m+k$-tuple $\bar a,\bar b$ was arbitrary, $M$ realizes $\bt\wedge\bs$ if and only 
if it realizes both $\bt$ and $\bs$. 

On the other hand, tuple $\bar a, \bar b$ realizes $\bt\vee \bs$ if 
and only if for every finite~$F\subseteq \kappa$ at least one of 
(i)  $\max_{\xi\in F} \varphi_\xi(\bar a)=0$ 
or (ii) $\max_{\xi\in F} \psi_\xi(\bar b)=0$ holds. 
If for a cofinal set of finite subsets 
$F$ of $\kappa$  (i) applies then  $\bar a$ realizes~$\bt$. 
Otherwise, there is a finite $F_0\subseteq \kappa$ such that for all $F\supseteq F_0$ 
(ii) applies and  $\bar b$ realizes~$\bs$.  
Since this tuple was arbitrary, 
 $M$ realizes $\bt\vee\bs$ if and only if it realizes at least one of the types $\bt$ or $\bs$. 
\end{proof}

\subsubsection{Type $\bt_\omega$}\label{S.pomega}
Assume $\bfT$ is a theory and  $\bt=\{\varphi_j(\bar x)=0: j\in \omega\}$ is an $n$-type omissible in a model of  $\bfT$. 
We shall assume $n=1$ for simplicity. 
 To $\bt$ we associate the type $\bt_\omega$ in infinitely many variables $x_j$, for $j\in \omega$, 
consisting of the following generalized conditions. 
\begin{enumerate}
\item [$(\bt_\omega 1)$] $\varphi_j(x_n)\leq \frac 1n$ for all $j<n$, and 
\item [$(\bt_\omega 2)$]  $d(x_j,x_{j+1})\leq 2^{-j}$  for all $j\in \omega$. 
\end{enumerate}
We can think of $\bt_\omega$ as an $\omega$-type---an  increasing union of $n$-types $\bt_n$, 
where~$\bt_n$ is the restriction of $\bt_\omega$ to the formulas involving only 
  $x_j$, for $j<n$.  
The following is clear. 

\begin{lemma} \label{L1} A model $M$ realizes type 
$\bt$ if and only if  every (equivalently, some) dense subset~$D$ of its universe  includes a sequence that realizes~$\bt_\omega$. \qed
\end{lemma} 

A slightly finer fact is true. If $D$ is an arbitrary dense subset of the universe of $M$,   
 consider $D^{<\omega}$ as a tree with respect to the end-extension. 
Let $T_{D,\bt}$ be the family of all $\bar d\in D^{<\omega}$ such that, with $n$ denoting
the length of $\bar d$ (using von Neumann's convention that $n=\{0,1,\dots ,n-1\}$, 
this is the  $n\in\omega$ such that $\dom(d)=n$), we have   
$M\models \bt_n(\bar d)$. 
We use  common terminology from   descriptive set theory and say that 
 a tree is   \emph{well-founded} if it has no infinite branches. 

\begin{lemma}\label{L1+} 
A model $M$ omits a type $\bt$ if and only if the tree $T_{D,\bt}$  
is well-founded for some (equivalently, for every) dense $D\subseteq M$. 
\end{lemma} 

\begin{proof} A proof virtually identical to the proof of Lemma~\ref{L1} 
shows that a tuple $\bar a$ realizes $\bt$ if and only if it is a limit of nodes 
on an infinite branch of~$T_{D,\bt}$. 
\end{proof} 

\subsection{Spaces} For a fixed countable language $L$ (see \S\ref{S.Formulas} below)
we shall now review definitions of standard Borel spaces of 
formulas, structures, complete types, and incomplete types.

\subsubsection{The Borel space $\hM(L)$ of models} 
\label{S.Borel.Models} Fix a countable language $L$. 
There are several ways to encode  separable $L$-structures. 
%by their countable substructures. 
%Every such structure is isomorphic with a   
%structures whose underlying set is  $\omega$. Thus 
%the space of countable $L$-structures is equipped  with the Cantor-set topology and  a
%  natural continuous action of the permutation group $S_\infty$. This observation  
%is a rich source of results on the interface between (classical first-order) 
%model theory and   descriptive  set theory 
%(see e.g.  \cite{Gao:Invariant}). 
%The space of separable metric structures  of a fixed countable language $L$ can be construed as a standard Borel
%space in more than one way. In \cite{EFPRTT} it was shown that every metric $L$-structure can be 
%canonically extended to one whose universe is the Urysohn metric space. 
%The Borel space $\cM(L)$ of all $L$-structures obtained in this way is not convenient for our purpose. 
We shall consider the space  essentially introduced in \cite[\S 4]{yaacov2017metric}.  
Although this space was denoted $\cM(L)$ in \cite{yaacov2017metric}, we use 
the notation $\hM(L)$ to avoid  conflict with \cite{EFPRTT}. 
 The space $\hM(L)$ is  defined as follows. 

For simplicity we consider the case when $L$ has no predicate symbols. 
Let $d_j$, for $j\in \omega$, be a sequence of new constant symbols and let $L^+=L\cup \{d_j: j\in \omega\}$. 
Let $\fp_j$, for $j\in \omega$, be an enumeration of all  $L^+$-terms with no free varlables
closed under the application of function symbols from $L$. 
Recall that in  logic of metric structures every function symbol $f$  is 
equipped with a fixed modulus of uniform continuity $\Delta_f\colon (0,1]\to (0,1]$
and that in every $L$-structure $M$ it is required that the representation of $f$ satisfies
(using the $\max$-metric on tuples in $M$)
\[
d(\bar x,\bar y)<\Delta_f(\e)\text{ implies } d(f(\bar x), f( \bar y))\leq \e. 
\] 
All relational symbols are also equipped with moduli of uniform continuity. 
By  $\hM(L)$ we denote the space of all functions 
\[
\gamma\colon \omega^2\to [0,1]
\]
such that
\begin{enumerate}
\item [(i)] $\gamma$ is a metric on $\omega$, 
\item [(ii)] $\gamma$ respects the moduli of uniform continuity of all functions in $L$. 
In particular \cite[(UC), p. 8]{BYBHU} holds: if $f$ is a function symbol with modulus of uniform continuity $\Delta_f$ 
and $i,j,i'$ and $j'$ are such that $\fp_{j'}=f(\fp_j)$ and $\fp_{i'}=f(\fp_i)$ 
then 
\[
\gamma(i, j)<\Delta_f(\e)\text{ implies } 
\gamma(i',j')\leq \e. 
\]
(An analogous condition holds for $n$-ary function symbols for $n\geq 2$.)
\end{enumerate}
The set of $\gamma\in [0,1]^{\omega^2}$ satisfying (i) and (ii) is a closed subspace of the Hilbert cube, 
and  $\hM(L)$ is equipped with the induced compact metric topology. 
To  $\gamma\in \hM(L)$ we associate  the structure $M_0(\gamma)$ 
with the universe $\omega$ 
and the  metric defined by $d(i,j)=\gamma(i,j)$. 
Terms $\fp_j$, for $j\in \omega$ (and therefore function symbols) 
are interpreted in $M_0(\gamma)$ using (ii). 
This structure is not an  $L$-structure  because it is incomplete, and 
its metric completion $M(\gamma)$ is a separable $L$-structure. 
Every complete separable metric $L$-structure $M$ is isometric to 
 a completion of $M_0(\gamma)$ corresponding to some $\gamma\in \hM(L)$
 and every such $M$ has many nonisomorphic representations in $\hM(L)$. 
 (Necessarily so, because the isomorphism relation of $L$-structures is frequently 
 not classifiable by countable structures; see \cite{Hj:Book}.)

One can modify $\hM(L)$ to accommodate the case when   $L$ has countably many 
predicate symbols.
  For an $n$-ary predicate symbol $R$ let  $\gamma_R\colon \omega^n\to [0,1]$
correspond to the interpretation of $R$ in $M(\gamma)$. 
If $R(j)$ and $k(j)\in \omega$, for $j\in \omega$, enumerate the relational symbols in $L$ 
and their arities then $M(\gamma)$ is identified with a closed subset of 
the Hilbert cube $[0,1]^{\omega^2}\times \prod_{j<\omega} [0,1]^{\omega^k(j)}$ in the natural fashion.  
%The straightforward details are omitted. 

\begin{remark} 
The  space $\hM(L)$ is similar to the space of separable \cstar-algebras $\hat\Gamma$ 
  introduced in  \cite{FaToTo:Descriptive}. Although $\hM(L)$ is different from the Borel space 
  of $L$-structures $\cM(L)$ defined in \cite{EFPRTT}, these two spaces are equivalent
  in the sense of \cite[Definition~2.1]{FaToTo:Turbulence} by \cite[Proposition~2.6 and Proposition~2.7]{FaToTo:Turbulence}. 
  The proof of this fact is analogous to the proof 
  given in  \cite[\S 3]{EFPRTT} for the case of \cstar-algebras. 
\end{remark}

A special case of the following lemma in the case of 
 \cstar-algebras was proved in \cite[Proposition~5.1]{FaToTo:Descriptive}. 
The proof of the general case is virtually identical. 

\begin{lemma} \label{L.Borel.Theory} 
%Assume $L$ is a countable language.  
 The function from $\hM(L)$ to the space of $L$-theories that associates 
the theory of $M$ to $M$ is Borel. \qed
\end{lemma} 

Lemma~\ref{L.1.5} below  is a straightforward consequence
of Lemma~\ref{L.Borel.Theory}  (similar to the one in the appendix of \cite{FaToTo:Descriptive}, 
where it was proved that separable \cstar-algebras that tensorially absorb the Jiang--Su algebra $\cZ$ form a Borel set.  

\begin{lemma} \label{L.1.5} Suppose $L$ is  a countable language. 
If $\bfT$ is an $L$-theory then the set of all $\gamma\in \hM(L)$ 
such that $M(\gamma)\models \bfT$ is Borel. \qed
\end{lemma}

\subsubsection{The linear space of formulas} \label{S.Formulas} 
For $n\in\omega$ let $\bbF_n(L)$ denote the set of all formulas whose free variables are among 
$\{x_0,\dots, x_{n-1}\}$. If $L$ is clear from the context we shall  write $\bbF_n$ instead of 
$\bbF_n(L)$.  
On the affine space~$\bbF_n(L)$ consider the seminorm
\[
\|\varphi\|_\infty=\sup_{M,\vec a} |\varphi(\vec a)^M|
\]
and the pseudometric 
\[
d_\infty(\varphi,\psi)=\|\varphi-\psi\|_\infty. 
\]
Here $M$ ranges over all $L$-structures, $\vec a$ ranges over all $n$-tuples of elements in~$M$, 
and $\varphi(\vec a)^M$ is the interpretation of $\varphi$  
 at $\vec a$ in $M$. 

Let $\frW_n(L)$ denote the completion of the quotient space $\bbF_n(L)/\|\cdot\|_\infty$. 
While in  \cite[\S 6]{Muenster} formulas are $\bbR$-valued and  the analogous space is a real Banach space, 
our $\frW_n(L)$ is an affine space. Given an $L$-theory $\bfT$ consider
the seminorm 
\[
\|\varphi\|_\bfT=\sup_{M,\vec a} |\varphi(\vec a)^M|, 
\]
where $M$ ranges over the models of $\bfT$ and $\vec a$ ranges over $n$-tuples of elements of $M$. 
To $\|\cdot\|_{\bfT}$ we associate  a pseudometric $d_{\bfT}$ and an  affine space $\frW_n(\bfT)$
defined analogously to $d_\infty$ and $\frW_n(L)$ above. 
  
As pointed out in \cite{BYBHU}, if $L$ is countable then $\frW_n(L)$ is separable with respect to this 
pseudometric. 
In some situations one may  implicitly consider theory as inseparable (no pun intended) 
 from the language $L$. 
For example, if $L$ has a function symbol associated to every complex number (as e.g. 
in the axiomatization of \cstar-algebras given in \cite{FaHaSh:Model2}) then $L$ is uncountable. However,~$\bfT$
may still imply that $\frW_0(L)$ is separable. 
In this case we say that  \emph{$L$ is separable}.  
The   separability of $L$ is assumed throughout the present section (and most of the paper).

\subsubsection{Compact metrizable spaces of theories and  types} \label{S.Borel.types}
Let $L$ be a countable language. Consider  the affine spaces of formulas 
$\bbF_n(L)$ and  $\frW_n(L)$  as defined in \S\ref{S.Formulas}. 
Every $L$-structure  $M$ defines the evaluation  linear  functional on~$\frW_0$ by
\[
\varphi\mapsto \varphi^M. 
\]
Since $\frW_0(L)$ is normed by 
$\|\varphi\|_\infty=\sup_{M} |\varphi^M|$   (see \S\ref{S.Formulas}),  
every evaluation functional  has norm $\leq 1$. 
Such functionals are \emph{complete $L$-theories} (see~\S\ref{S.preliminaries}). The space of all complete $L$-theories is denoted   $S_0^L$  
and   equipped with   the topology of pointwise convergence (also known as the     \emph{logic topology} or the    \emph{weak$^*$-topology}).  
 By the compactness theorem for the logic of metric structures,~$S_0^L$ is compact 
  (see~\cite[Theorem~5.8]{BYBHU}). Since $L$ is separable $_0^L$ is metrizable, although 
  it does not have a canonical metric. 

Suppose an  incomplete 
 $L$-theory $\bfT$ is fixed (possibly $\bfT=\emptyset$). 
 Every complete extension $\bfT'$ of $\bfT$ defines a continuous 
 affine functional on $\frW_0(\bfT)$ and hence an element of $S_0^L$. 
 The set of functionals associated to complete extensions of $\bfT$ in this manner
 is, by the compactness theorem for the logic of metric structures, 
  a closed subset of $S_0^L$  (see also \cite[Lemma~6.1.1 and Lemma~6.1.2]{Muenster}).   
Since each functional in this set uniquely determines a complete extension of $\bfT$, 
we identify this compact metrizable space with the space of complete extensions of $\bfT$
and denote it $S_0^L(\bfT)$. We write~$S_0(\bfT)$ whenever $L$ is clear from the context.

%Via the interpretation map  $(M,\varphi)\mapsto \varphi^M$ 
%the spaces $\hM(L)$ and $\bbF_0(L)$ are in  duality (although note that $\hM(L)$ is not a linear space). 

%%%%%%
An $L$-model $M$ and an $n$-tuple $a_i$, for $i<n$, 
in $M$ define
by interpretation a linear functional $\tp_M(\bar a)$ on $\frW_n(L)$  by 
\[
\tp_M(\bar a)(\varphi(\bar x))=\varphi(\bar a)^M. 
\]
 The 
space $S_n(\bfT)$ of all consistent complete $n$-types is compact in the logic topology, 
and we identify it with the space of (consistent and complete) $n$-types  over $\bfT$. 

The \emph{set of realizations} of $n$-ary type $\bt$ in model $M$ 
is  
\[
\bt(M)=\{\bar b\in M^n: M\models \bt(\bar b)\}.
\]
 Completeness of $\bfT$ implies that for every $M\models \bfT$ 
the 
set 
\[
\{\bt\in S_n(\bfT): \bt\text{ is realized in $M$}\}
\]
is dense in the logic topology on $S_n(\bfT)$.

\subsubsection{A metric on the space of complete $n$-types over a 
complete theory~$\bfT$} 
\label{S.Metric} 
Let $L$ be a countable language, let~$\bfT$ be a complete $L$-theory, and fix~$n\geq 1$. 
Following (\cite[p. 44]{BYBHU}) on the space $S_n(\bfT)$ of complete $n$-types over~$\bfT$  (\S\ref{S.Borel.types})
we define  metric $d=d_{\bfT,n}$ by ($M^n$ is considered with the $\max$-distance)  
\begin{multline*}
d(\bt,\bs)=\inf\{d(\bar a,\bar b): \text{ there exist $M\models T$ and $\bar a$ and $\bar b$ in $M$}\\
\text{ such that 
$M\models \bt(\bar a)$ and $M\models \bs(\bar b)$}\}. 
\end{multline*}
Since both types and $\bfT$ are complete, the triangle inequality is satisfied.
Also, $d(\bt,\bs)=0$ if and only if $\bs(M)=\bt(M)$ for every $M\models \bfT$, in which case $\bs$ and $\bt$ correspond to the same functional on~$\frW_n(\bfT)$. 
This topology is stronger than the compact topology  and 
these topologies roughly correspond
to  the norm and weak topologies in functional  analysis (see however
Lemma~\ref{L.d.Borel}).

Given an $n$-type $\bt$ over a theory $\bfT$ and a new $n$-tuple of constants $\bar c$, 
we let~$\bfT_{\bt/\bar c}$ denote the theory in the language $L\cup \{\bar c\}$ 
obtained by extending~$\bfT$ with axioms asserting that $\bar c$ realizes $\bt$. 
More precisely, one adds all conditions of the form $\varphi(\bar c)=0$ to $\bfT$, 
 where $\varphi(\bar x)=0$ 
is a condition in $\bt$.  If  $\bt$ is complete then 
so is~$\bfT_{\bt/\bar c}$.  One can iterate this definition and name realizations of 
more than one type, as in the following proof. 

\begin{lemma} \label{L.d.Borel}
For every $n$ and  $\e\geq 0$ the  set $\{(\bfr,\bfs): d(\bfr,\bfs)>\e\}\subseteq S_n(\bfT)^2$ is open in the logic topology.  
\end{lemma} 

\begin{proof} 
Fix types $\bt$ and $\bs$ such that $d(\bt,\bs)>\e$. 
This is equivalent to stating that 
\[
\bfT_{\bt/\bar c,\bs/\bar d}\models d(\bar c, \bar d)>\e. 
\]
 Then by compactness  there exists a  finite set of open conditions  $\bfT_0$ in~$\bfT_{\bt/\bar c,\bs/\bar d}$
 such that $\bfT_0\models d(\bar c, \bar d)>\e$. 
 This defines a logic open neighbourhood~$U$ of $(\bt,\bs)$ in $S_n(\bfT)$  
such that $d(\bfr,\bfs)>\e$ for all $(\bfr,\bfs)\in U$. 
\end{proof}

\subsubsection{The compact metrizable space $S_n^-(\bfT)$ 
 of incomplete $n$-types over $\bfT$}\label{S.Incomplete}
Fix a countable language $L$, $n\geq 1$ and a
 (not necessarily complete) 
 $L$-theory~$\bfT$.  An $n$-type~$\bt$ in $\bfT$ is a countable set of conditions (\S\ref{S.Conditions}), 
but we can also identify it with the set 
(considering a type as a  set of conditions)
\[
K_{\bt}=\{\bs\in S_n(\bfT): \bt\subseteq \bs\}. 
\]
This set is closed
(and therefore compact) in the logic topology. We claim that every nonempty compact $K\subseteq S_n(\bfT)$ is 
equal to $K_{\bt}$ for some $\bt$.  Fix~$K$; then  
\[
\bt_K=\bigcap_{\bs\in K} \bs 
\]
 is a type that includes $\bfT$. 
If $\bs\in S_n(\bfT)\setminus  K$, then since $K$ is closed 
there exists an open  condition $\varphi(\bar x)<\e$ such that the closed condition 
 $\varphi(\bar x)\leq \e/2$ belongs to $\bs$
but not to any type in $K$. Therefore no type $\bs\notin K$  extends $\bt_K$ and we have
\[
K=\{\bs\in S_n(\bfT): \bs\supseteq \bt_K\}. 
\]
We can therefore  identify the space $S_n^-(\bfT)$ of not necessarily complete types over 
$\bfT$ with  the space of all closed subsets of $S_n(\bfT)$ (the so-called exponential space) 
equipped with its compact metrizable topology
given by the Hausdorff metric. (Note that the empty set corresponds to all inconsistent types extending $\bt$.)

We identify $S_n(\bfT)$ with a subset of $S_n^-(\bfT)$ consisting of all singletons. 
This is a closed subset of $S_n^-(\bfT)$ and the subspace topology agrees with the 
compact metrizable topology on $S_n(\bfT)$. When discussing  complexity of a set of 
types over $\bfT$ we consider the set of all types,  complete and incomplete. 

A straightforward proof of the following lemma is omitted
(see Lemma~\ref{L.pairing} for the definitions of types $\bs\vee \bt$ and $\bs\wedge\bt$). 

\begin{lemma}\label{L.vee}
Suppose $\bfT$ is a theory in a countable language. For all $k$ and $m$
the type pairing maps 
\begin{align*} 
\vee&\colon S_m^-(\bfT)\times S_k^-(\bfT)\to S_{m+k}^-(\bfT)\\
\wedge&\colon S_m^-(\bfT)\times S_k^-(\bfT)\to S_{m+k}^-(\bfT)
\end{align*}  
are continuous. \qed
\end{lemma}

\subsubsection{The compact metrizable space  of pairs $(\bfT,\bt)$}\label{S.Pairs}
Fix a countable language $L$. 
For a fixed $n$ the set of pairs 
 $(\bfT,\bt)$ where $\bfT$ is a (not necessarily complete) $L$-theory and~$\bt$ is a complete type over $\bfT$ 
is endowed with a compact metrizable 
topology as follows. 
We  identify each pair $(\bfT,\bt)$ with the complete theory  $\bfT_{\bt/\bar c}$ (see 
\S\ref{S.Metric}).  Every complete theory in the language obtained by extending the 
language of $\bfT$ by adding constants~$\bar c$ is equal to~$\bfT_{\bt/\bar c}$ for some pair $(\bfT,\bt)$. 
Therefore the set 
\[
\{(\bfT, \bt): \bfT\text{ is a not necessarily complete $L$-theory and } \bt\in S_n(\bfT)\}
\]
 is identified with the compact metrizable space   
 $S_0^{L\cup \{\bar c\}} (\bfT)$ as in 
  \S\ref{S.Borel.types}.

\subsection{Omitting complete types} 

%\subsubsection{Principal types} 
\label{S.Principal} 
Assume $\bfT$ is a complete theory in a countable language
and $\bt$ is a (not necessarily complete) 
 $n$-type in the signature of $\bfT$.  
As in \S\ref{S.Incomplete}, we  identify~$\bt$ with the set $K_{\bt}$ of all complete types extending $\bt$. 
An $n$-type (complete or not) $\bt$ is \emph{principal} (or \emph{isolated}) if 
for every  $\e>0$  the set (using the metric on $S_n(\bfT)$ defined in  \S\ref{S.Metric})
\[
B_\e(\bt)=\{\bs\in S_n(\bfT): \dist(\bs,K_{\bt})<\e\}
\]
is somewhere dense in the logic topology with respect to $\bfT$. 
This is equivalent to the definition of a principal type given  in  \cite[Definition~12.2]{BYBHU}
by   \cite[Proposition~12.5]{BYBHU}. 
 % 
% 
%We state a result from \cite{BYBHU} for future reference 
%(\cite[Definition~9.16]{BYBHU}, 
% \cite[Proposition~12.4]{BYBHU},  and 
%see also \S\ref{S.Definable}).
%
% 
%
%
%\begin{prop} \label{L.NP} Suppose $\bt(\bar x)$ is a type over a theory $\bfT$. 
%The following are equivalent. 
%\begin{enumerate}
%\item $\bt(\bar x)$ is  principal. 
%\item  For every $\e>0$ there exists an open condition $\psi(\bar x)<\delta$
%such that for every  stronger condition $\varphi(\bar x)<\delta'$,   
%and every   $M\models \bfT$ and tuple $\bar a$ in~$M$
%$M\models \varphi(\bar a)<\delta'$ implies 
%$\dist(\bar a, \bt(M))\leq \e$. 
%\item The distance predicate   $\dist(\bar x, \bt(M))$ is  definable in every $M\models \bfT$. 
% \item The metric and logic topologies on the space of types agree at $\bt$, 
% \item The set $\bt(M)$ of all realizations of $\bt$ in  $M$ is always definable.\qed
% \end{enumerate}
%\end{prop} 
The following is  the  omitting
 types theorem given in~\cite[Theorem~12.6]{BYBHU}  (note that in \cite{BYBHU} all types were assumed to be complete).

\begin{thm} \label{L.triv} Suppose  $\bt$ is a type over a theory $\bfT$ in a countable language. 
\begin{enumerate}
\item If $\bt$  is not principal, then~$\bfT$ has a separable model omitting~$\bfT$. 
\item If $\bt$ is principal and both $\bt$ and $\bfT$ are complete then  $\bt$ is realized in every  $M\models\bfT$. \qed
\end{enumerate}
\end{thm} 

%(2) Suppose $\bt$ is a principal and complete $k$-ary type. Since the language of $\bfT$ is countable 
%we can choose a countable dense set of conditions in $\bt$, enumerated  as   $\varphi_n(\bar x)=0$, for $n\in \omega$.   
% For  $n\geq 1$ the set 
%$B_{1/n}(\bt)=\{\bs\in S_k(\bfT): d(\bt,\bs)<1/n\}$
%is somewhere dense in the logic topology. Let $U_n$ be a logic open set included in the closure of $B_{1/n}(\bt)$. 
%By shrinking $U_n$ if needed we may assume that for all $l\leq n$ and all $\bs\in U_n$ the condition 
%$\varphi_l(\bar x)\leq 1/n$ belongs to $\bs$.  
%%%%%%
% Since the set of 
%types realized in $M$ is dense in the logic topology of  $S_k(\bfT)$, there exists 
% $\bar a_n\in M^k$ such that the type  $\bs_n$ of $\bar a_n$ belongs to $U_n$.  
%By the triangle inequality we have  $d(\bs_n, \bs_{n+1})<\frac 1n +\frac 1{n+1}$, 
%and we can therefore choose $\bar a_n$, for $n\in \omega$, to be a Cauchy sequence in $M^k$. 
%By the choice of $U_n$, the limit of this sequence is a realization of $\bt$ in $M$. 
%\end{proof} 

This paper is  about the case not covered by Theorem~\ref{L.triv}: principal, 
but not complete, types over a (possibly complete) theory in a countable language.

\subsubsection{The set of omissible complete types is Borel}

 The set of $n$-types $\bt$ omissible in a model of a theory $\bfT$
  is $\bSigma^1_2$ (see Proposition~\ref{P.Sigma-1-2}).  
We consider the logic topology on the space of all complete theories 
in a fixed countable language $L$ (\S\ref{S.Borel.types}). 
For $n\in \omega$ consider the space of pairs $(\bfT,\bt)$ where $\bfT$ is 
an $L$-theory and $\bt$ is a complete $n$-type over $\bfT$ with respect to the logic topology defined in \S\ref{S.Pairs}.

\begin{prop} \label{P.Borel} 
Fix  $n\in \omega$. 
\begin{enumerate} 
\item The set of all pairs $(\bfT,\bt)$ such that $\bfT$ is a complete theory and  
$\bt$ is a complete $n$-type realized in every model of~$\bfT$ is Borel.  
\item The set of all pairs $(\bfT,\bt)$ such that $\bfT$ is a  theory and  
$\bt$ is a not necessarily complete  $n$-type realized in some  model of~$\bfT$ is closed. 
\end{enumerate}
\end{prop} 

\begin{proof} (1) By the Omitting Types Theorem (\cite[\S 12]{BYBHU},  \cite{Ha:Continuous}, or Corollary~\ref{C.Omitting})
type $\bt$ has to be realized in every model of $\bfT$ if and only if it is principal
 (principal types were defined in \S\ref{S.Principal}).

Since the logic topology is second countable, expressing the fact that~$B_\e(\bt)$ 
is nowhere dense requires only quantification over a countable set of open conditions. 
It therefore suffices to show that the set $\{\bs: d(\bt,\bs)\geq \e\}$ is Borel, 
and this follows from Lemma~\ref{L.d.Borel}.  

(2) 
This is a consequence of the  compactness theorem, \cite[Theorem~5.8]{BYBHU}. 
\end{proof} 

The  completeness assumption on types in Proposition~\ref{P.Borel} (1) 
is necessary by Theorem~\ref{T0}, but see also Proposition~\ref{P.UDT-Borel}.

\subsection{A test for elementary equivalence}
We include a general test for elementary equivalence  used only in~\S\ref{S.T1} below. 
A subset $Y$ of a metric space is \emph{$\e$-dense} if for every point $x\in X$ 
there exists $y\in Y$ such that $d(x,y)<\e$. 

\begin{lemma} \label{L.isomorphic}
Assume $A$ and $B$ are $L$-structures such that 
for every finite $L_0\subseteq L$ and every 
$\e>0$ there are $\e$-dense substructures $A_0$ and $B_0$ of $L_0$-reducts 
of $A$ and $B$, 
respectively, which are isomorphic. 
Then $A$ and $B$ are elementarily equivalent. 
\end{lemma} 

\begin{proof} 
By  \cite[Proposition~6.9]{BYBHU} 
every formula can be uniformly approximated by formulas in 
 prenex  form. It will therefore suffice to show that every formula in   prenex form has the same value in $A$ and $B$. 
Let $\varphi$ be an $L$-sentence in  prenex  
form,  
\[
\varphi=\adjustlimits \sup_{x_0} \inf_{x_1} \dots \adjustlimits \sup_{x_{2n-2}} \inf_{x_{2n-1}} 
\psi(\bar x),
\]
where $\psi(\bar x)$ is quantifier-free 
and let $L_0\subseteq L$ be a finite subset
consisting only of symbols that appear in $\psi$. 
For $\delta>0$ fix $\e>0$ small enough so that 
perturbing  variables in $\bar x$ by $\leq \e$ 
does not change the value of $\psi(\bar x)$ by more than $\delta/2$. 
Let $A_0$ and $B_0$ be isomorphic $\e$-dense substructures of reducts of $A$ and $B$, respectively.   
Then 
\[
|\varphi^A-\varphi^{A_0}|\leq \delta/2
\qquad \text{and}
\qquad
|\varphi^B-\varphi^{B_0}|\leq \delta/2. 
\]
Since $A_0\cong B_0$ we have $\varphi^{A_0}=\varphi^{A_0}$ and $|\varphi^A-\varphi^B|\leq \delta$. 
Since $\delta>0$ was arbitrary, we conclude 
that $\varphi^A=\varphi^B$. Since $\varphi$ was arbitrary  $A$ and $B$ are 
elementarily equivalent. 
\end{proof} 

%
%Lemma~\ref{L.isomorphic} can also be proved by using the EF-games (see \cite{Ha:Continuous}
%or \cite{goldbring2015games}) 
%and it admits a number of  yet unexplored possibilities for generalizations. 
%For example, one could define a variant of the Gromov--Hausdorff distance on the space of 
%$L$-structures  and show that for any fixed sentence $\varphi$ the computation of $\varphi$ 
%is continuous with respect to this metric. 
%See also \cite[Corollary~2.1]{FaGoHaSh:Existentially}. 
%A related convenient fact is that the evaluation of $L$-sentences on the  
%space of substructures of a `monster' model, considered with the Hausdorff metric, 
%is continuous. This is 
%applicable to the perturbation theory of \cstar-algebras (see \cite[\S 5.15]{Muenster}). 

\subsection{A simple fact about definable sets} 
\label{S.Definable} 
We include a brief discussion of definable sets 
as a response to the anonymous referee's request. 
Let $L$ be a language and let $M$ be an $L$-structure. 
Following 
\cite[Definition~9.1]{BYBHU} (see also \cite[\S 3.1]{Muenster}) 
we say that a predicate $R\colon M\to [0,1]$ is \emph{definable} 
if it  
can be uniformly approximated by $L$-formulas. 
More precisely, for every $\e>0$ there is an $L$-formula $\varphi(x)$ 
such that 
$\sup_{x\in M} |R(x)-\varphi(x)^M|\leq \e$. 
 A closed subset~$Z$ of  $M$  is \emph{definable} 
(\cite[Definition~9.16]{BYBHU}, also \cite[\S 3.2]{Muenster}) 
 if the associated 
distance predicate   
\[
\dist(x,Z)=\inf_{z\in Z} d(x,z)
\]
is definable. 
In the case when $M$ is a multi-sorted structure and $Z$ is a subset of one of its sorts, $S$, 
one requires only $\sup_{x\in S^M}|\dist(x,Z)-\varphi(x)|\leq \e$. 
By \cite[Theorem~9.17]{BYBHU} (see also 
\cite[Theorem~3.2.2]{Muenster} for the treatment of
 sets definable in a not necessarily complete theory) a set is definable if and only if 
for every $L$-formula $\psi(\bar y,x)$ the predicates 
\[
\sup_{x\in Z} \psi(\bar y,x)\text{ and } 
\inf_{x\in Z} \psi(\bar y,x)
\]
are definable. In other words, expanding the definition of an $L$-formula by allowing 
quantification over definable sets results in a conservative extension of the language $L$. 
(An extensive discussion of definability in  multi-sorted and unbounded case,  imaginaries, $\bfT^{\eq}$, and $A^{\eq}$ in the logic of metric structures 
can be found in \cite[\S 3.1--3.3]{Muenster}.)
We state and prove the single-sorted, single variable version of an (undoubtedly 
 well-known) lemma that will be used in \S\ref{S.T1}.

\begin{lemma} \label{L.Definable}
Suppose that $M$ is a metric $L$-structure with a sort $S$ and that $Z\subseteq S^M$
is the zero set of some formula $\varphi(x)$ such that 0 is an isolated point of the set $Y=\{\varphi(a): a\in S^M\}$. 
Then both $Z$ and $S\setminus Z$ are definable. 
\end{lemma} 

\begin{proof} Let $f\colon [0,1]\to [0,1]$ be any  continuous function such that $f(0)=0$ and $f(t)=1$
for $t\in Y\setminus\{0\}$. Then $f(\varphi(a))^M$ is equal to 
$0$ if $a\in Z$
and it is equal to $1$ if $a\in S^M\setminus Z$. 
The formulas
\begin{align*} 
\psi_0(x)&=\inf_z \min(1,f(\varphi(z))+d(x,z)),\\
\psi_1(x)&=\inf_z \min(1,(1-f(\varphi(z)))+d(x,z))
\end{align*} 
 satisfy $\psi_0(a)^M=\dist(a,Z)$ and $\psi_1(a)^M=\dist(a, M\setminus Z)$ for all $a\in S^M$, 
 showing that both $Z$ and $S\setminus Z$ are definable. 
\end{proof} 

\section{Trees} \label{S.Trees} 

The category of trees of height $\omega$ with order-preserving and level-preserving maps 
is not axiomatizable in classical logic. We will see that it is axiomatizable in logic of metric structure. 
The examples in \S\ref{S.Difficult} of theories with sets of types of high complexity 
and the examples in \S\ref{S.T1} of complete theories with omissible, but not simultaneously omissible, 
types with be expansions of the theory of trees. The axiomatizability result (Theorem~\ref{T.T.Ax}) will not be used directly 
in other proofs.

\subsection{The Theory of the Baire space} \label{S.Baire} 
Let $L_{\cN}$ be a language with a single sort~$D_1$. 
 The intended interpretation of $D_1$  is $\olo\sqcup \oo$. 
  Language $L_{\cN}$  is  equipped with the following. 
\begin{enumerate}
\item  \label{I.L.1} Constant symbols for all elements of $\olo$ (we shall identify 
$t\in \olo$ with the corresponding constant). 
\item \label{I.L.2} Unary function symbols $f_k$ for $k\in \omega$. 
\pushcounter
\end{enumerate}
Interpretations of  each $f_k$  is required to be  1-Lipshitz. 
Theory $\bfT_{\cN}$ is the theory of the $L_{\cN}$-model $\cN$ whose  universe is  
$\olo\sqcup \oo$, 
described as follows. 

The length of  $s\in \olo$ will be denoted by $|s|$, hence $\dom(s)=|s|$. 
The metric on $\olo\sqcup \oo$ is the standard Baire space metric
(with $\Delta(s,t)=\min(|s|,|t|, \{n: s(n)\neq t(n)\}$ for $s\neq t$)
 \[
 d(s,t)=1/(\Delta(s,t)+1). 
 \]
  If $s\sqsubseteq t$ then $d(s,t)=1/(|s|+1)$. 
  For $s\in \olo$  and $k\leq |s|$ we denote 
  the initial segment of $s$ with length $k$ by $s\rs k$. 
 For $k\in \omega$ the function $f_k$ is interpreted as 
 \[
  f_k(s)=\begin{cases} 
  s\rs k& \text{ if }k\leq |s|\\
   s&\text{  if }k>|s|. 
\end{cases} 
\]
 Clearly $\olo$ is a dense subset of $D_1^M$ which is closed under all $f_k$. 

 Since $\cN$ has a  subset consisting of  elements that are interpretations of constant symbols, 
 it is an atomic model of $\bfT_{\cN}$ and  
  every~$N\models\bfT_{\cN}$ has an elementary submodel isometrically isomorphic to $\cN$.  

\subsection{Axiomatizability and $L_{\cN}^-$}  \label{S.WF} 
\label{L.CN-}
A class of metric structures is axiomatizable if it is the class of all models of a first-order theory $\bfT$, and 
a category is axiomatizable if it is equivalent to an axiomatizable class of metric structures. 

\begin{thm} \label{T.T.Ax} The category of  trees of height $\leq \omega$ with respect to the order-preserving and level-preserving morphisms
is axiomatizable. 
\end{thm} 
  
Let $L_{\cN}^-$ be the reduct of~$L_{\cN}$  %(\S\ref{S.Baire}) 
with sort $D_1$ and functions $f_k$ for $k\in \omega$, but without constant 
symbols for the elements of $\olo$. Recall that the height of a tree $T$ is the least ordinal 
$\alpha$ such that the $\alpha$th level of $T$ is empty. 
Suppose $T$ is a tree of height $\omega$. It is isomorphic to a tree of finite sequences of 
elements of a large enough set with respect to the end-extension ordering. 
An $L_{\cN}^-$ structure is $\cM_T$ associated to $T$ as in  \S\ref{S.Baire}.  
One defines distance $d(s,t)=1/(\Delta(s,t)+1)$ for distinct elements $s$ and $t$ in~$T$, and 
the interpretation of $f_k$ is defined so that $f_k(s)=s$ if $|s|\leq k$ 
and $f_k(s)=s\rs k$ if  $|s|>k$. 
The $d$-completion of~$T$  is an 
$L_{\cN}^-$-structure  whose domain is naturally identified with the union of~$T$ and the set of its infinite branches.

\begin{proof}[Proof of Theorem~\ref{T.T.Ax}]
We have already associated metric structure $\cM_T$ to a tree $T$. 
To prove that the class $\{\cM_T: T$ is a tree of height $\leq\omega\}$ is axiomatizable, we use a well-known result  (see e.g. \cite[Theorem~2.4.1]{Muenster}):  
 A category of metric structures is axiomatizable if it is  
 closed under the isomorphisms, ultraproducts, and elementary submodels. 
Only the closure under ultraproducts requires some argument.

Suppose $T(j)$, for $j\in J$, is a family of trees and $\cV$ is an ultrafilter on~$J$. 
With $M_j=\cM_{T(j)}$ we need to find a tree $T$ such that $M=\prod_{\cV} M_j$ is naturally isomorphic to $\cM_T$.  
For $m\in \omega$ let $T_m$ be the $\cV$-ultraproduct of $m$-th levels $T(j)_m$ of $T(j)$.  
The elements of $T_m$ are the equivalence classes of functions $g\in \prod_j  T(j)_m$
modulo the $\cV$-almost everywhere equality. Let $T$ be the tree whose $m$-th level is $T_m$ with the natural ordering. 
It is a routine to check that $\cM_T\cong \prod_{\cV} M_j$.  
 
By \cite[Theorem~2.4.1]{Muenster},   there exists a theory $\bfTT$ such that $M\models \bfTT$ if and only if $M\cong \cM_T$ for some tree $T$.

It remains to prove that the category of models of $\bfTT$ is equivalent to the category of trees of height $\leq\omega$. 
 Since $T$ is identified with a dense subset of $\cM_T$, 
   an order-preserving and  level-preserving $f\colon T\to S$ 
defines a morphism from $\cM_T$ to $\cM_S$. We therefore have a functor that associates metric structure $\cM_T$ 
to a tree $T$. 

In order to find the inverse functor, fix $M\models \bfTT$. 
Then $M$ 
 has an  $F_\sigma$ subset
\[
M=\{a\in N: f_k(a)=a\text{ for some }k\}. 
\]
This is a dense subset of   $M$ since $d(a, f_n(a))\leq 1/(n+1)$ for all $a$. 
With the ordering defined by $a\sqsubseteq b$ if and only if $a=f_k(b)$ for some $k$, $T_M$ is a tree of height $\omega$. 
 Since $T_M$ is a dense subset of $M$ and $\cM_T$ has $T$ as a dense subset,~$T_{\cM_T}$ is naturally isomorphic to $T$ and  $\cM_{T_M}$ is naturally isomorphic to $M$. 
This completes the proof. 
\end{proof}

This subsection concludes with some straightforward observations on models of $\bfTT$ and 
an introduction of terminology. If $N\models \bfTT$, the
  elements of $N\setminus T_N$ are in a natural bijective correspondence to 
   the branches of tree~$T_N$, 
 because $f_k(x)=f_k(y)$ for all $k$ implies $d(x,y)=0$ and therefore
 $x=y$.
Thus any  $N\models\bfTT$ has at most 
two kinds of elements, 
\emph{nodes} (the elements of~$T_N$) and \emph{branches} (the elements of $N\setminus T_N$). 
The nodes of $N$ are the ranges of  interpretations of functions
 $f_m$, for $m\in \omega$. 
If $f_m(a)=a$ then  $d(a,b)<1/(m+1)$ implies $a=b$, and therefore nodes form a discrete subset of $N$. 
The \emph{height} of a node $a\in N$,  denoted $|a|$, 
  is the least $m$ such that $f_{m}(a)=a$.

All non-nodes of $N$ satisfy  $d(f_m(a), a)=1/(m+1)$ for all $m$. 
These are the \emph{branches} of $N$. 
If  $N\models \bfTT$ then $T_N$ %as in \S\ref{S.WF}  
is a tree and 
we naturally extend  the notation $|t|$ and~$\sqsubseteq$
to elements of $T_N$. This applies to expansions of models of $\bfT_{\cN}$ and $\bfTT$  used in 
\S\ref{S.Difficult} and 
\S\ref{S.T1}.

\subsection{Some definable sets} \label{S.NodesEtc}

The results of the present subsection,  that some natural sets are definable in 
 models of $\bfTT$, will be used in \S\ref{S.M4} and \S\ref{S.ProofT2}. 
The following lemma was added because of a request of the 
 anonymous referee. %,  will be used in \S\ref{S.T1} (see~\S\ref{S.Definable}). 

\begin{lemma} \label{L.Def.Nodes} Suppose %$T$ is a tree of height $\omega$,  
$N\models \bfTT$
%is a model of the theory of $T$, $\bfT_T$ (\S\ref{L.CN-}), 
and $m\geq 1$. 
\begin{enumerate}
%\popcounter
\item \label{I.Def.2} Both the set of all nodes of height $>m$ 
and the set of all nodes of height $\leq m$ are definable subsets of $N$.  
\item \label{I.Def.1} The set of all nodes of height $m$ (i.e. the $m$th \emph{level} of $N$) 
is a definable subset of $N$. 

\item \label{I.Def.3} For a node $a$ in $N$ and $k>m=|a|$ the set 
\[
\Succ_k(a)=\{b\in N: f_m(b)=a\text{ and } |b|=k\}
\]
is  a definable (from parameter $a$) subset of $N$. 

\item \label{I.Def.4} For $m$ and $n$ in $\omega$ the set 
\[
N_{m,n}=\{a\in N: |a|=m, \rho_N(a)\geq n+1\}
\]
 is a definable subset of $N$.   %%%%%%%%%%%%%!!!!!!!!!!
\pushcounter
\end{enumerate}
\end{lemma} 

\begin{proof} If $T$ is any tree of height $\omega$, for every  node $s$ of $T$ we have $f_m(s)=s$ if $|s|\leq m$
and $d(s,f_m(s))=1/(m+1)$ if  $|s|>m$.  
Therefore  $d(s,f_m(s))<1/(m+1)$ is equivalent to   
 $d(s,f_m(s))=0$, and  
\eqref{I.Def.2}  follows by Lemma~\ref{L.Definable}.

\eqref{I.Def.1} Let 
\[
\varphi_m(x)=d(f_m(x), x)+\left|\frac 1{m} \dminus d(f_{m-1}(x),x)\right|.  
\] 
Then $\varphi_m(a)^N=0$ if and only if $f_m(a)=a$ and $d(f_{m-1}(a),a)\geq 1/m$. 
But $f_m(a)=a$ if and only if $|a|\leq m$ and $d(f_{m-1}(a),a)\geq 1/m$ if and only if $|a|\geq m$. 
%If $b\in N$ is a node and $|b|>m$ or it is a branch 
% then $d(f_m(b),b)\geq 1/(m+1)$.  If $b\in N$ is a node with  $|b|<m$ then 
%  $d(f_{m-1}(b),b)<1/m$ since
%$f_{m-1}(b)=b$. 
  Therefore  the $m$th level of $N$ is the zero-set of $\varphi_m$. Since 
   $\varphi_m(b)^N<1/(m+1)$
 implies $\varphi_m(b)^N=0$,  by Lemma~\ref{L.Definable}
 the $m$th level of $N$ is definable.

\eqref{I.Def.3}  For $m<k$ let  $\psi_{m,k}(x,y)=d(f_m(x),y)+\varphi_{k}(x)$. 
Then $\Succ_k(a)$ is the zero-set of $\psi_{m,k}(x,a)$ and 
$\psi_{m,k}(x,a)<1/(k+1)$ implies $\psi_{m,k}(x,a)=0$. 
Therefore $\Succ_k(a)$ is definable from parameter $a$ by Lemma~\ref{L.Definable}. 

\eqref{I.Def.4} Let   $\varphi_m$ and $\psi_{m,n}$  be as in the earlier part of this proof. Then  
$N_{m,n}$ is the zero-set of $\varphi_m(x)+\inf_y \psi_{m,n}(y, x)$, 
and it  is definable by Lemma~\ref{L.Definable}. 
 \end{proof}

\begin{remark} To the best of our knowledge, it is not known whether the intersection of two definable sets in a metric structure is necessarily definable.
However, the definable sets $X$
that appear in Lemma~\ref{L.Def.Nodes} are `strongly definable' 
in the sense that both $X$ and its complement $X^\complement$ are closed, and 
$\dist(X,X^\complement)>0$. The intersection of `strongly definable' sets is `strongly definable.'   
\end{remark} 

Lemma~\ref{L.Def.Nodes} shows that, to some extent, the  models of $\bfTT$  behave as
discrete first-order structures. The following lemma is an example of this.

\begin{lemma} \label{L.counting} 
The following property, 
 depending on $m,n,k,l$ in $\omega$ such that $k<n$, 
  of trees   is elementary: 
 \begin{quote} 
 Every node  $s$ of height $m$  and rank at least $n+1$ 
 has at least~$l$ distinct immediate successors  of rank  $k$. 
  \end{quote} 
\end{lemma} 

\begin{proof} Fix $N\models\bfTT$. 
Lemma~\ref{L.Def.Nodes} implies that   
   both the $m$-th level of~$T_N$ and the set of nodes of rank $k$ are definable.
   Also, the distance of any two distinct nodes on the $m$th level is at least $1/(m+1)$
   and the assertion that such $1/(m+1)$-discrete set $X$ has at least $l$ elements can be expressed  as follows: 
\[
   \inf_{y_0\in X,y_1\in X,\dots, y_{l-1}\in X} \frac 1{m+1}\dminus \min_{j<k<l}\d(y_j,y_k). 
\]
Using our assumption on $X$, the value of 
this sentence is 0 if and only if $X$ contains at least $l$ elements. 
\end{proof} 

\section{Complexity of spaces of types} 
\label{S.Difficult} 

 From now on, \emph{all types are assumed to be partial and consistent}. 
%Following~\cite{BYBHU} we write $r\dminus s$ for $\max(0,r-s)$.    
Let us briefly recall  definitions of pointclasses in  
  projective hierarchy. 
A subset of a Polish space~$X$ is~$\bDelta^1_1$ if it is Borel (this is Souslin's Theorem, 
\cite[14.11]{Ke:Classical}). 
The images
of $\bDelta^1_1$ sets under Borel-measurable functions are $\bSigma^1_1$ (or analytic) sets, and their complements
are~$\bPi^1_1$ sets. 
For $n\geq 2$ one defines $\bSigma^1_n$ and $\bPi^1_n$ 
subsets of $X$ by recursion. Continuous images of $\bPi^1_n$  sets 
are $\bSigma^1_{n+1}$ sets and their complements are $\bPi^1_{n+1}$ sets. 
This is a proper hierarchy and for every $n$ the  pointclasses
$\bSigma^1_{n+1}\setminus \bSigma^1_n$, $\bSigma^1_{n+1}\setminus \bPi^1_{n+1}$, 
$\bPi^1_{n+1}\setminus \bSigma^1_{n+1}$, 
and $\bPi^1_{n+1}\setminus \bPi^1_n$ are nonempty. 
If $\Gamma$ is a pointclass then a subset $Z$ of a Polish space is \emph{$\Gamma$-hard} 
if every set in $\Gamma$ is a preimage of  $Z$ by a continuous function. 
If a $\Gamma$-hard set $Z$ itself belongs to $\Gamma$ then it is said to be \emph{$\Gamma$-complete}. 
See \cite{Ke:Classical} for more information.

 Whenever we say that a type $\bt$ is omissible in a model of  $\bfT$ it is assumed 
 to be  consistent with $\bfT$.  

\subsection{Simple complexity results} 

\label{S.Basic.Complexity}

Recall that the space $S_n(\bfT)$ of complete $n$-types over a complete theory $\bfT$ is a compact 
metric space (\S\ref{S.Borel.types}) 
and that the space $S_n^-(\bfT)$ of not necessarily complete $n$-types over $\bfT$ is identified with the
compact metrizable space of its closed subsets (\S\ref{S.Incomplete}). 

\begin{lemma} \label{L.Omissible} \label{L.Sigma11} 
Suppose $L$ is a countable language,  $M$ is a separable $L$-structure, and $n\in \omega$. 
 \begin{enumerate}
 \item The set of all 
complete $n$-types realized in $M$ is $\bSigma^1_1$. 
\item The set of all not necessarily complete $n$-types realized in $M$ is~$\bSigma^1_1$ 
\end{enumerate} 
\end{lemma}

\begin{proof} (1) 
Let $a(M,j)$, for $j\in \omega$, be an enumeration of a countable subset of 
 the  universe of $M$. 
 Define $f\colon   (\oo)^n\to S_n(\bfT)\cup \{*\}$ by 
 setting 
 $f(x(0), \dots, x(n-1))$ to be  
  the  type of the limit of $x(0)(j), \dotsm x(n-1)(j)$, for $j\in \omega$, in $M$
  if this limit exists in $M$ and to  $*$
  if the sequence $\{a(M,x(i)(j)): j\in \omega\}$ is not Cauchy for some $i<n$.  
   This function is Borel and the image of $(\oo)^n$ is the set of $n$-types realized in $M$. 

(2) By (1), the set $A\subseteq S_n(\bfT)$ of all complete 
$n$-types omitted in $M$ is~$\bPi^1_1$. 
We need to show that the set $\{K\in S_n^-(\bfT): K\subseteq A\}$ is also $\bPi^1_1$. 
This is standard but we include an argument for the convenience of the reader. The set
\[
Z=\{(x,K)\in S_n(\bfT)\times S_n^-(\bfT): x\in K\}
\]
is closed and  $K\subseteq A$ if and only if $(\forall x)((x,K)\in Z\rightarrow x\in A)$, 
giving the required $\bPi^1_1$ definition.  
%
%Here is an alternative proof of (2). With the natural definition of the space of incomplete 
% $\omega$-types
%$S_\omega^-(\bfT)$ one notes that the map $\bt\mapsto \bt_\omega$ 
%(as in \S\ref{S.pomega}) is a continuous map from $S_n^-(\bfT)$ to $S_\omega^-(\bfT)$. 
%By Lemma~\ref{L1+} type $\bt$ is omitted in $M$ iff the tree $T_{D,\bt}$ is 
%well-founded. 
 \end{proof}

 The set of complete 
 $n$-types realized in a  
  model $M$ of $\bfT$  is always dense in the logic topology on $S_n(\bfT)$ 
   and therefore if~$M$ does not realize all types then 
 the set of complete types realized in $M$ is not closed in logic topology.  
 A model~$M$ of a complete theory $\bfT$ is \emph{atomic} if the set of realizations
of principal types is dense in~$M^n$ for every $n\geq 1$
 (see \cite[p. 79]{BYBHU}). This is equivalent 
to every tuple of~$M^n$ having a principal type. 
Compare the following with Proposition~\ref{P.Borel}.  
    
  \begin{prop} \label{P.Sigma-1-2}  \label{C.Prime} %\label{L.Sigma-1-2} 
Fix $n\geq 1$ and a countable language $L$.  
\begin{enumerate}
  \item 
  If $\bfT$ is a complete $L$-theory with an atomic
   model then the set of 
   (not necessarily complete) 
  $n$-types omissible in a model of $\bfT$ is $\bPi^1_1$. 
\item If $\bfT$ is a (not necessarily complete) $L$-theory then the set of all 
(not necessarily complete) 
$n$-types omissible in a model of $\bfT$ is $\mathbf \bSigma^1_2$. 
\end{enumerate}
\end{prop}
  
  \begin{proof} 
  Suppose that   $\bfT$  
  has an atomic model.  Then it has a separable atomic model and this model
  is an elementary submodel of every other model of $\bfT$ (i.e. it is a \emph{prime} model of $\bfT$). 
  This implies that a type is omissible in a model of $\bfT$ if and only if it is omitted in its atomic model.   Together with   (2) of Lemma~\ref{L.Omissible} this implies (1). 
  
  (2) By the proof  of Lemma~\ref{L.Omissible}
the set 
   \[
   \{(M,\bt): M\models \bfT, \bt\in S_n^-(\bfT),\text{  and $M$ omits }\bt\}
   \]
    is $\bPi^1_1$
   and the set of all types omitted in some model of $\bfT$ is clearly its continuous 
   image. 
   \end{proof}  

Proposition~\ref{P.Sigma-1-2} (1) is sharp by
Corollary~\ref{C-1-1} below. 
  The following related result  
was inspired by \cite{Bice:Brief}. 

\begin{prop} \label{P2+} There are a countable language $L$ and a separable $L$-structure~$M$ 
such that the set of quantifier-free unary types realized in $M$ is 
a complete $\bSigma^1_1$ set. 
\end{prop} 

\begin{proof} Language $L$ has only one unary predicate symbol $f$, interpreted as a 1-Lipshitz function.  
Consider the Baire space  $\oo$ with a complete separable metric.  
We fix a homeomorphism of $\oo$ with $[0,1]\setminus \bbQ$
(e.g. send an  irrational to  the associated continuous fraction). 
 Fix a closed $X\subseteq \oo\times\oo$. 
Consider~$X$ with the $\max$-metric induced from~$\oo$ and interpret $f$ as the projection 
to the $x$-axis. Then  $f$ is 1-Lipshitz. 

The only atomic formulas in $L$ are $f(x)$ and $d(x,y)$. If $\varphi(x)$ is a quantifier-free formula with only one 
free variable, then the only atomic subformulas of $\varphi$ are of the form $f(x)$ and $d(x,x)$. 
The latter is identically equal to~0, and therefore the quantifier-free type of an element of $M$ 
is completely determined by its projection to the $x$-axis (with $\oo$ identified with a subset of $[0,1]$).  
Choosing $X$ so that its projection is a complete analytic set 
(combine \cite[\S 25.A]{Ke:Classical} and \cite[\S 27.B]{Ke:Classical}) 
 completes the proof. 
\end{proof}

\subsection{Type $\bs_0$ and the standard model} \label{S.q}
 We shall describe a theory $\bfT_{\cN,h}$ and a unary type $\bs_0$ in
the expanded language   $L_{\cN,h}=L_{\cN}\cup\{h\}$ (where $h$ is a new unary function symbol) 
such that the only model of $\bfT_{\cN,h}$
omitting $\bs_0$ is the unique expansion $\cN_h$ of $\cN$ to a model of $\bfT_{\cN,h}$.

Fix an enumeration $s_n$, for $n\in \omega$, of $\olo$. 
Let $\cN_h$ be the expansion of $\cN$ to $L_{\cN,h}$ obtained by 
 interpreting $h$ as follows (here $\langle n\rangle$ denotes the element of $\olo$
of length $1$ whose only digit is $n$). 
\[
h(x)=
\begin{cases} 
s_n, & \text{ if $x=\langle n \rangle$ for some $n$}\\
f_1(x), & \text{ otherwise.}
\end{cases} 
\]
Then $h$ is 2-Lipshitz because $d(x,y)\leq 1/2$ and $x\neq y$ 
 implies $h(x)=h(y)$. Let $\bfT_{\cN,h}=\Th(\cN_h)$ and
\[
\psi=\sup_x \inf_y d(x,h(y))+d(f_1(y),y).
\]
The range of $h\rs \{y: f_1(y)=y\}$ is $\olo$.  Since 
$\olo$ is a dense subset of $\cN$ we have $\psi^{\cN_h}=0$.  
By elementarity, 
 in every model $M$ of $\bfT_{\cN,h}$ the set $\{h^M(x): f_1(x)=x\}$ is dense in $M$.

\begin{lemma}\label{L.s0} There exists a partial $1$-type $\bs_0$ 
over $\bfT_{\cN,h}$ such that 
 $M$ omits~$\bs_0$ if and only if $M\cong \cN_h$, 
for every $M\models \bfT_{\cN,h}$.  
\end{lemma} 

\begin{proof} 
Let $\bs_0(x)$ be the type consisting of conditions 
\[
d(f_1(x), \langle n\rangle)=1
\]
for all $n\in \omega$. Then every finite subset of $\bs_0$ is realized in $\cN_h$ by a 
large enough $\langle m\rangle$. 
On the other hand, $\cN_h$ clearly omits $\bs_0$. Fix  $M\models \bfT_{\cN,h}$ which omits $\bs_0$.  Then 
 $\{y\in M: f_1(y)=y\}=\{\langle n \rangle : n\in \omega\}$. It follows that~$\cN_h$ %, the atomic model of $\bfT_{\cN,h}$,
 is dense in $M$ and $M$ and $\cN_h$ are isometrically isomorphic. 
  \end{proof} 

\subsection{$\bPi^1_1$-completeness}  Fix a complete theory $\bfT$ in a countable language. 
The  set of (not necessarily complete) $n$-types omissible in a model of $\bfT$ 
is~$\bSigma^1_2$,  by Proposition~\ref{P.Sigma-1-2}. 
  
 \begin{thm} \label{P1} There is 
 a complete theory $\bfT_2$ in a countable language $L_2$ 
 such that 
 the space of all 2-types $\bt$  omissible in a model of   $\bfT_2$ 
 is $\bPi^1_1$-complete. 
 \end{thm} 
 
 \begin{proof} We define an expansion $L_2$ of $L_{\cN}$  (\S\ref{S.Baire}) 
 to be a two-sorted language with the sorts $D_1$ and $D_2$. If $M$ is an $L_2$-structure
 then $D_1^M$ is an $L_{\cN}$-structure. 

A \emph{subtree} of $\olo$ is a (possibly empty) subset  $a$ of $\olo$ such that  $s\in a$ 
implies $t\in a$ for all $t\sqsubseteq s$ in $\olo$. 

 The intended interpretation of  $D_2$ is  the space $\cT$ of all subtrees of~$\olo$.   
The  language~$L_2$  is  equipped with the following

%  ($\ee$ stands for `element'; we  need a better notation)
\begin{enumerate}
\popcounter
\item\label{I.L.3}  Constant symbols $S_n$, for $n\in \omega$, 
for all finite-width subtrees of $\olo$ all of whose branches have eventually zero value.
\item A binary predicate symbol $\ee$ of sort $D_1\times D_2$. 
%\item predicate symbol $R$. 
\pushcounter
\end{enumerate}
The interpretation of $\ee$   is required to be  1-Lipshitz. 
Theory $\bfT_2$ is the theory of the $L_2$-model $\cN_2$ described as follows. 
The universe of $\cN_2$ is the set $\olo\sqcup \oo \sqcup \cT$ and $\cN_2$ is an 
extension of the $L_{\cN}$-model $\cN$ as described in~\S\ref{S.Baire}. 

 The metric on $\cT$ is defined via 
(let  
 $\delta(a,b)=\min\{k: a\cap k^{\leq k}\neq b\cap k^{\leq k}\}$)
  \[
 d(a,b)=1/\delta(a,b).
 \]
Since the set of all subtrees of $\olo$ is a closed 
set in the Cantor-set topology of $\cP(\olo)$, 
  $D_2^{\cN_2}$ is a compact metric space and metric $d$ is easily seen to be compatible with this topology. 
Interpretations of constant symbols from \eqref{I.L.3}  form a countable dense set. 
  
We introduce an auxiliary function $\ell\colon \olo\to \omega$ via
\[
\ell(t)=\max(\{|t|\}\cup \range(t))
\]
and define $\ee$  on $\olo\times \cT$ via
 \[
 \ee(t,S)=\begin{cases} 
 0, & \text{ if } t\in  S,\\
 1/(\ell(t)+1), & \text{ if } t\notin S. 
 \end{cases} 
 \]
The predicate $\ee$ is Lipshitz on $\olo \times D_2^{\cN_2}$
because $d(S,T)\leq 1/k$ and $d(s,t)\leq 1/k$ 
implies that $\ee(s,S)\leq 1/m$ iff $\ee(t,T)\leq 1/m$ for all $m\leq k$.  
We can therefore continuously extend $\ee$ to $\oo\times \cT$. 
Then we have $\ee(x,S)=0$ for all  $S\in \cT$ and all branches $x$ of $S$.

  Let $\bfT_2=\Th(\cN_2)$. 
  Since the interpretations of  constants in $L$  
  form a dense subset of $\cN_2$, it is an atomic model 
of~$\bfT_2$. Therefore a type is omissible in a model of $\bfT_2$ if and only if it is omitted
in $\cN_2$.

Let $S\in \cT$. 
 We let $\bt^S$ be the partial  type in variables
 $x,y$ of the sort $D_1\times D_2$ consisting of the following conditions: 
\begin{enumerate}
\popcounter
\item\label{I.T2..1}  $\ee(f_k(x),y)=0$ for all $k$. 
\item\label{I.T2..2} $\ee(t,y)=0$ if $t\in S$ and $\ee(t,y)=1/(\ell(t)+1)$ if $t\notin S$, for all constants 
$t\in \olo$. 
\item\label{I.T2..3} $d(S_n,y)=\e_n$, where $\e_n=d(S_n,S)$, for all $n$. 
\item  \label{I.T2..4}
$|d(f_k(x), x)-1/(k+1)|=0$.
\pushcounter
\end{enumerate}
Suppose $b,c$ is a realization of $\bt^S$ in a model $M$ of $\bfT_2$. Then $D_1^M\models \bfT_{\cN}$ 
and~$T_M$ is a tree of height $\omega$ (see \S\ref{S.WF}).  
By \eqref{I.T2..2}
and  \eqref{I.T2..3} $c$ is a subtree of~$T_M$ such that $c\cap \olo=S$. 
By  \eqref{I.T2..4} $b$ is a branch and by 
 \eqref{I.T2..1} it is a branch of the tree~$c$.

The map $\cT\ni S \mapsto \bt^S\in S_2^-(\bfT_2)$ (see \S\ref{S.Incomplete}) is clearly continuous
(each of the spaces is considered with respect to its compact metrizable topology). 
Since the set of well-founded trees in $\cT$ is $\bPi^1_1$-complete (\cite[32.B]{Ke:Classical}), 
it only remains to check 
that $\bt^S$ is omissible in a model of  $\bfT_2$ if and only if $S$ is well-founded. 
Since the standard model $\cN_2$ of $\bfT_2$
is the atomic model of~$\bfT_2$, this is equivalent to $\bt^S$ being 
omissible in $\cN_2$. 

If $S$ is well-founded then $\cN_2$ omits~$\bt^S$. This is because 
if $(b,a)$ realizes~$\bt^S$ then $a=S$, and therefore $b\in \oo$ has to be a `true' branch of $S$.  
If $S$ is ill-founded and $b$ is its branch, then 
$\cN_2$ realizes $\bt^S$ by $(b,S)$. 

Since by (1) of  Proposition~\ref{C.Prime} the set of types omissible in a model of $\bfT$ is $\bPi^1_1$, 
this completes the proof. 
\end{proof} 

\begin{coro} \label{C-1-1} 
There exists a separable model $\cN_2$ such that the 
set of types omitted in $\cN_2$ is $\bPi^1_1$-complete. 
\end{coro} 

\begin{proof} Model $\cN_2$ used in Theorem~\ref{P1} is an atomic model of its theory, 
and therefore a type is omissible in a model of $\bfT_2$ if and only if $\cN_2$ omits it. 
\end{proof} 

\subsection{$\bSigma^1_2$-completeness}
Denote  the  compact metrizable space of all
subtrees of $\olo\times \olo$ by~$\cT^2$. For $R\in \cT^2$ and $x\in \oo$ let 
\[
R_x=\{s\in \olo: (s,x\rs |s|)\in R\}. 
\]
The following fact is well-known but we could not find a reference in the literature. 

\begin{lemma} \label{L.Sigma12-complete}
The subspace $Z$ of all $R\in \cT^2$ such that for some~$x$ the tree~$R_x$ is well-founded is a
complete $\bSigma^1_2$ set. 
\end{lemma} 

\begin{proof} This set is clearly $\bSigma^1_2$. Since every uncountable Polish space is Borel-isomorphic to $\oo$, it suffices
to show that every $\bSigma^1_2$ subset of $\oo$ is a continuous preimage of~$Z$. 
Suppose $A$ is a $\bSigma^1_2$ subset of $\oo$. 
Fix a closed subset $F$ of $(\oo)^3$
such that $A=\{x: (\exists y\in \oo)(\forall z\in \oo) (x,y,z)\notin F\}$ and  
let  $T=\{(x\rs m, y\rs m, z\rs m): (x,y,z)\in F, m\in \omega\}$. 
Define  $R\colon \oo\to \cT^2$ by  
\[
R(x)=\{(s,t): |s|=|t|\text{ and }(x\rs |s|, s,t)\in T\}.
\]
 This 
is a continuous map, and  $x\in A$ if and only if the tree 
$R(x)_y$ is well-founded for some $y\in \oo$. Therefore $A=R^{-1}(Z)$. 
Since $A$ was an arbitrary~$\bSigma^1_2$ subset of $\oo$  
this completes the proof. 
\end{proof} 

The following theorem is logically incomparable with Theorem~\ref{P1} since 
 the theory $\bfT_3$  is not complete.

 \begin{thm} \label{P2} There is a theory $\bfT_3$ in a countable language $L_3$ such that 
 the space of all $2$-types  omissible in a model of   $\bfT_3$ is $\bSigma^1_2$-complete. 
 \end{thm}

\begin{proof} 
We define a joint expansion $L_3$ of $L_2$ as in the proof of Theorem~\ref{P1}
and $L_{\cN,h}$ as in~\S\ref{S.q}.  
It is  a three-sorted language with sorts $D_1$, $D_2$, and~$D_3$, 
with $D_1$ and $D_2$ interpreted as before, with the exception of predicate $\ee$. 
The intended interpretation of~$D_3$ is $\cT^2$.

In addition to  symbols imported from $L_2$ and $h$, in~$L_3$ we have the following. 
\begin{enumerate}
\popcounter
\item  A constant symbol $c$ of  sort $D_1$. 
\item  Constant symbols $R_n$, for $n\in \omega$, of sort $D_3$  
for all finite-width subtrees of $\olo\times \olo$ 
all of whose branches have eventually zero value. 
%\item A constant symbol $R$ for a distinguished element of $\cT^2$. 
\item Predicate $\ee$ of the  
sort $D_1\times D_1\times D_3$.%\footnote{Formally, this is a new predicate but we use the same notation for the sake of simplicity.}
\pushcounter
\end{enumerate}
Let $L_3^-$ be the language $L_3$ without the constant symbol $c$. 
Theory $\bfT_3$ is the theory of the $L_3^-$-model $\cN_3$ described as follows. 
Its universe is equal to $\olo\sqcup \oo\sqcup \cT\sqcup\cT^2$
and it includes model $\cN_2$ as defined in the proof of Theorem~\ref{P1}
and has predicate $h$ as defined in~\ref{S.q}. 

The metric $d$ on $\cT^2$ defined as
\[
d(R,S)=\inf\{1/k: R\cap (k^{\leq k})^2\neq S\cap (k^{\leq k})^2\}
\]
turns $\cT^2$ into a compact metric space. 

Define $\ee$  on $\olo\times \olo\times \cT^2$ via
 \[
 \ee(s,t,R)=\begin{cases} 
 0, & \text{ if } (s,t)\in  R,\\
 1/(\max(\ell(s),\ell(t))+1), & \text{ if } (s,t)\notin R. 
 \end{cases} 
 \]
This  predicate is Lipshitz since $d(S,T)\leq 1/k$, 
$d(s_1,t_1)\leq 1/k$ and $d(s_2,t_2)\leq 1/k$ 
together imply that $\ee(s_1,s_2,S)=1/m$ iff $\ee(t_1,t_2,T)=1/m$ for all $m\leq k$.  
We can therefore continuously extend $\ee$ to $\oo\times \cT^2$. 
By  the continuity  we have $\ee(x,y,S)=0$ for all $x,y\in \oo$ and all $S\in \cT$. 

 Theory   $\bfT_3=\Th(\cN_3)$ is not a complete $L_3$-theory. 
   since 
 it provides no information on the interpretation of 
the constant symbol  $c$. For  $R\in \cT^2$
let  $\bt^R(x,y)$ be a type in the sort $D_1\times D_3$  consisting of the following conditions. 
\begin{enumerate}
\popcounter
\item\label{I.pR.0}  $d(R_n,y)=r$, where $r=d(R_n,R)$, for all $n\in \omega$. 
\item \label{I.pR.1} $\ee(f_k(x), f_k(c), y)=0$ for all $k\in \omega$. 
\item \label{I.pR.2}  $|d(f_k(x), x)-1/(k+1)|=0$.
\pushcounter
\end{enumerate}
If $(b,S)$ realizes $\bt^R$ in $M\models\bfT_3$, then $R=S$ by \eqref{I.pR.0}, $f_k(b)\neq b$ for all $k$ 
by \eqref{I.pR.2}, and $b$ is a branch of $R_c$ by \eqref{I.pR.1}.  
Let $\bs=\bs_0$ defined in \S\ref{S.q}

We claim that $\bt^R$ and $\bs$ are simultaneously 
 omissible if and only if there exists a real $a\in \oo$ such that $R_a$ is well-founded. 
If there is such a real, then the model of~$\bfT_3$ obtained by interpreting $c$ as $a$ omits both 
$\bt^R$ and~$\bs$. 
Now assume there is no such real and let $N$ be a model of $\bfT_3$ in which~$\bs$ is omitted. 
By Lemma~\ref{L.s0}, 
 the reduct of  $N$  to $L_{\cN}$ is isometrically isomorphic to $\cN_h$. 
If $a$ is the interpretation of $c$ in $N$, 
then $c\in \oo$. Therefore the tree~$R_c$ is ill-founded, and  $\bt^R$ is realized in $N$ 
by $(b,R)$ where $b$ is any branch of $R_c$. 

The function $\cT^2\ni R\mapsto \bt^R\in S_1^-(\bfT_3)$  (see \S\ref{S.Incomplete}) 
is clearly continuous. Therefore  Lemma~\ref{L.vee} implies that the function 
$\cT^2\ni R\mapsto \bt^R\vee \bs_0\in S_2^-(\bfT_3)$ is continuous. 
We therefore have a continuous map from $\cT^2$ into $S_2^-(\bfT_3)$ such that the 
preimage of the set of omissible types is, by 
Lemma~\ref{L.Sigma12-complete},  a  $\bSigma^1_2$-complete set, 
and this concludes the proof.   
\end{proof}

\section{Forcing and  omitting types}
\label{S.forcing}

Our study of generic models is motivated by potential applications to operator 
algebras (see \cite{Fa:Logic},  \cite{Muenster}, also \S\ref{S.Uniform} and \S\ref{S.CR}). 
Related results   
were obtained in  \cite{ben2009model} and \cite{eagle2013omitting}, 
similarly inspired by Keisler's classic~\cite{keisler1973forcing}. 
Both of these papers study a version of Keisler's forcing adapted to the 
infinitary version of the logic of metric structures. 
In the first-order logic a type (complete or partial) is omitted in a generic 
model if and only if it is omissible. 
In the  logic of metric structures this  remains true for 
complete types (by \cite[Theorem~12.6]{BYBHU} or 
 Theorem~\ref{P3} below) but not for partial types (Theorem~\ref{C.generic}). 
There are several good sources for 
the metric model-theoretic forcing  
(\cite{caicedo2014omitting}, \cite{eagle2013omitting}, \cite{ben2009model}, 
\cite[Appendix A]{goldbring2014kirchberg}, \cite[\S 6]{Muenster},  and \cite{FaGoHaSh:Existentially}). Since 
the present paper is a companion to \cite{Muenster} meant to be self-contained
and accessible to non-logicians,  we include some of the basics for the reader's convenience. 
The forcing construction described below is also known as the \emph{Henkin construction}.

\subsection{The forcing notion $\bbPTS$} \label{S.Forcing.Sigma}
Fix  a (not necessarily complete) theory $\bfT$ in a (not necessarily separable) 
 language~$L$ and a set of $L$-formulas $\Sigma$
 with  the following closure 
properties. 
\begin{enumerate}
\item [($\Sigma1$)] $\Sigma$  includes all quantifier-free formulas,
\item  [$(\Sigma 2)$] $\Sigma$ is closed under taking subformulas and  change of variables, 
\item [$(\Sigma 3)$] if $k\in \omega$,  $\varphi_i(\bar x)$, for $0\leq i<k$, 
 are in $\Sigma$,  and $f\colon [0,1]^k\to [0,1]$ is a continuous 
function, then  $f(\varphi_0(\bar x),\dots, \varphi_{k-1}(\bar x))$ is in $\Sigma$. 
\end{enumerate}
Two most interesting cases are when $\Sigma$ is the set of all quantifier-free formulas and when $\Sigma$ is the set of all formulas.

We postulate a simplifying assumption 
that~$L$ has a single sort with a single domain of quantification. If this is not the case, the 
 forcing can be modified by adding an infinite supply of constants (like $d_j$, for $j\in \omega$ below) 
 for every domain of quantification. For example, in the 
 case of \cstar-algebras, tracial von Neumann algebras, 
  or other Banach algebras 
 one adds constants~$d^n_j$, for $j\in \omega$, 
 for elements of the $n$-ball for every $n\geq 1$; 
 we omit the   straightforward details (see \cite[\S 6]{Muenster} for the case of \cstar-algebras). 

Let $d_j$, for $j\in \omega$, be a sequence of new constant symbols and let 
\[
L^+=L\cup \{d_j: j\in \omega\}.
\] 
If $F=(f(0), \dots, f(n-1))$ is an $n$-tuple of natural numbers, then we write 
 \[
 \bar d_F=(d_{f(i)}: i<n).  
\]
Let
\[
\Sigma^+=\{\varphi(\bar d_F): \varphi(\bar x)\in \Sigma, \bar d_F\text{ 
is of the same length as }\bar x \}. 
\]
For every $L$-formula $\varphi(\bar x)$ and every tuple $F$ of the appropriate 
length,~$\varphi(\bar d_F)$ is an $L^+$-sentence.  Conversely, 
 every $L^+$-sentence is of this form.

An open  condition   (see \S\ref{S.CFC}) $\varphi(\bar d)<\e$ is \emph{satisfied} in model $M$ 
if there exists $\bar a$ in $M$ of the appropriate length such that $\varphi(\bar a)^M<\e$. 
A condition is \emph{consistent with $\bfT$} if it is satisfied  in some model of $\bfT$. 
Suppose  $\varphi$ and $\psi$ are  $L^+$-sentences such that every $d_i$ that occurs in $\psi$ occurs in $\varphi$. 
We write 
\[
\bfT\cup\{\varphi<\e\}\models \psi<\delta
\]
 if  in every $M\models \bfT$ and every interpretation $\bar a$ of constants $d_i$ occurring in $\varphi$ in  $M$
 one has that $\varphi(\bar a)^M<\e$ implies $\psi(\bar a)^M<\delta$.

A \emph{condition} in $\bbPTS$ is   a triplet
 \[
 p=(\psi^p, F^p, \e^p)
 \]
(we shall write $(\psi,F,\e)$ whenever $p$ is clear from the context)
  where $\psi$ is an $n$-ary formula in $\Sigma$, $F$ is an $n$-tuple of natural numbers, 
    $\e>0$, 
and $\psi(\bar d_F)< \e$ 
is a condition consistent with $\bfT$. 
We shall write $\bar d^p$ instead of~$\bar d_{F^p}$.  
The poset~$\bbPTS$ is ordered  by 
\begin{align*}
 p\geq q& \qquad\text{if}\qquad\text{$F^p\subseteq F^q$ and }
\bfT\cup\{\psi^q(\bar d^q)<\e^q\}\models  \psi^p(\bar d^p)<\e^p. 
 \end{align*}
If $p\geq q$ then we say that $q$ \emph{extends} $p$ or that $q$ is \emph{stronger than} $p$. 
By Lemma~\ref{L.CFC} every condition is equivalent to some $p$ such that $\e^p=1$. 
Conditions $p$ and $q$ are \emph{incompatible}, $p\perp q$, if no condition extends both $p$~and~$q$. 
Conditions $p$ and $q$ are \emph{compatible}, $p\not\perp q$, 
 if some condition extends both $p$ and $q$. 

In the terminology of \cite{hodges2006building} and \cite{FaGoHaSh:Existentially}, 
if $\Sigma$ consists of all quantifier-free formulas then $\bbPTS$ is the \emph{Robinson forcing}, 
or the \emph{finite forcing}. If~$\Sigma$ consists of all formulas, then $\bbPTS$ is the \emph{infinite forcing}. 
In the latter case, we shall write $\bbPT$ for $\bbPTS$. 

We  identify condition $p=(\psi^p,F^p,\e^p)$ in $\bbPTS$ with 
the  open condition $\psi^p(\bar d^p)<\e^p$
and  use notations $\bfT+p$ and $\bfT\cup \{\psi(\bar d^p)<\e\}$ interchangeably. 

A recap of 
 standard forcing terminology (\cite{Ku:Set}, \cite{schindler2014set}) is in order. 
 Subset~$G$ of $\bbPTS$ is a \emph{filter} if every two elements of~$G$ have a common extension in $G$ and $q\in G$ and $p\geq q$ implies $p\in G$. 
A subset $\bfD$ of $\bbPTS$ is \emph{dense} if every $q\in \bbPTS$ has an extension in  $\bfD$. 
It is \emph{dense below} some $p\in \bbPTS$ if every $q\leq p$ has an extension in $\bfD$. 
If $\bfF$ is a family of dense subsets of $\bbPTS$ then a filter $G$ is \emph{$\bfF$-generic} 
if $G\cap \bfD\neq \emptyset$ for all $\bfD\in \bfF$.   

\begin{lemma} \label{L.CtbleGeneric} If $\bfF$ is a countable family of dense subsets of $\bbPTS$ 
then there exists a $\bfF$-generic filter. 
\end{lemma} 

\begin{proof} Enumerate sets in $\bfF$ by $\omega$ 
and choose a decreasing sequence $p_n$, for $n\in \omega$, so that $p_n$ belongs to the $n$th set in $\bfF$. 
Then 
\[
G=\{q\in \bbPTS: (\exists n) p_n\leq q\}
\]
 is an $\bbF$-generic filter. 
\end{proof} 

\begin{lemma} \label{L.Dphi}
For  $\varphi(\bar d_F)\in \Sigma^+$ and $k\geq 1$ the set
\[
\bfD_{\varphi(\bar d_F), k}=\{p\in \bbPTS: (\exists r\in \bbR) \bfT+p\models |\varphi(\bar d_F)-r|<1/k\}. 
\]
 is  dense in  $\bbPTS$. 
\end{lemma}  

 \begin{proof}  Fix a condition  $p\in \bbPTS$ and let $n=n^p$. We may assume that $\e^p<1/k$. 
 If $M\models \bfT$ is such that $\bar a\in M^n$ satisfies $p$ then  
  $M\models \psi^p(\bar a)<\e^p$. With $r=\varphi(\bar a)^M$, the condition
 \[
q= (\max(\psi^p, |\varphi(\bar d_F)-r|), F^p, \e^p)
 \]
is satisfied in $M$ by $\bar a$ and it extends $p$. 
\end{proof} 

% The fact that $\bfD_{\varphi(\bar d_F),\e}$ is dense in $\bbPTS$ 
%for all    $\varphi$ and $\e>0$ 
%even if $\varphi\notin \Sigma$ is also true (it follows from Cohen's forcing lemma, see \cite{Ku:Set} or \cite{schindler2014set}), 
%but we shall not need it. 

If $L$ is separable then  $G$ meets  all dense sets of the form $\bfD_{\varphi(\bar d_F), \e}$ 
if and only if it meets all dense sets of the form $\bfD_{\varphi_j(\bar d_F), 1/k}$ 
where $\varphi_j(\bar x)$, for $j\in \omega$, is a set of formulas dense 
in the pseudometric~$d_{\bfT}$  defined in  \S\ref{S.Formulas} and $k\in \omega$. 

 A formula is an \emph{$\forall\exists$-formula} if it is of the form 
 \[
\textstyle \sup_{\bar x}\inf_{\bar y}\psi(\bar x, \bar y, \bar z)
 \]
 where $\psi$ is quantifier-free. A theory $\bfT$ is \emph{$\forall\exists$-axiomatizable}
 if it is axiomatizable by a set of $\forall\exists$-sentences. 
We emphasize that in the following proposition theory $\bfT$ 
 is not assumed to be  complete and that the language $L$ is not assumed to be separable. 
 
\begin{thm} \label{P3} 
Suppose  $\bfT$ is a theory  in a 
language $L$  and $\Sigma$ is a set of $L$-formulas 
satisfying $(\Sigma 1)$--$(\Sigma 3)$. Then there is a  family $\bfF$ of dense subsets of $\bbPTS$ with the following property. 
If $G$ is an   $\bfF$-generic filter, then there exists a unique
 $L^+$-structure $M_G$ satisfying the following.  
\begin{enumerate}
\item \label{I.P3.1} 
The interpretations of  $\{d_j: j\in \omega\}$ form 
a dense subset of $M_G$. 
\item \label{I.P3.1.1} Every condition $p\in G$ is satisfied in $M_G$. 
\item \label{I.P3.2} If $\Sigma$ is the set of all  $L$-formulas 
then $M_G\models \bfT$. 
\item \label{I.P3.3} If $\bfT$ is $\forall\exists$-axi\-omati\-za\-ble
then $M_G\models \bfT$. 
\end{enumerate}
If $L$ is separable, then  an 
 $\bbF$-generic filter $G$ and model $M_G$ with the above properties  exist. 
\end{thm} 

The proof of Theorem~\ref{P3} is broken up into a few lemmas. 
Family $\bfF$ shall include all sets   $\bfD_{\varphi(\bar d_F), \e}$ defined above 
as well as five other families of dense subsets of $\bbPTS$ similarly 
indexed by $L^+$-formulas $\varphi(\bar s)$ in $\Sigma$, natural numbers, and tuples of 
constants in 
$L^+$. 
A proof that each of these sets is dense in $\bbPTS$ is, being very similar to the proof 
of the density of $\bfD_{\varphi(\bar d_F), \e}$, omitted. 

The first family, indexed by 
an $L^+$-formula $\varphi(\bar d_F, x)$ such that $\varphi$ belongs to  $\Sigma$
with a single free variable 
$x$, $r\in \bbQ$, and $m\geq 1$ is defined by 
\begin{align*}
\bfC_{\varphi(\bar d_F, x),r,m}=\{p\in \bbPTS: &\bfT+p\models \inf_x\varphi(\bar d, x)>r-1/m\\
&\text{ or } (\exists j) \bfT+p\models \varphi(\bar d_F, d_j)<r\}. 
\end{align*}
The second family is indexed by  
$i,j $ and $k>0$ in $\omega$ 
(`$d$' stands for `distance') 
\[
\bfD_{d,i,j,k}=\{p\in \bbPTS: (\exists r) \bfT+p\models |d(d_i,d_j)-r|<1/k\}. 
\]
The third and fourth families are indexed by 
predicate symbols 
$P$  or functional symbols~$g$ in $L$,  sets $F\subseteq \omega$ whose cardinality is equal to the arity 
of~$P$ (or of $g$, respectively), 
 and $k>0$ in $\omega$
 \begin{align*}
\bfD_{P,F,k}&=\{p\in \bbPTS: (\exists r\in [0,1]) \bfT+p\models |P(\bar d_F)-r|<1/k\}\\
\bfD_{g,F,k}&=\{p\in \bbPTS: (\exists l\in \omega) \bfT+p\models |f(\bar d_F)-d_l|<1/k\}. 
\end{align*} 
The next family is indexed by 
 $p\in \bbPTS$  such that  $\psi^p=\inf_x \varphi(\bar d_F, x)$ for 
 some $\varphi$ in $\Sigma$: 
\[
\bfD_{\inf,p}=\{q: q\perp p\text{ or } (\exists j\in \omega) \bfT+
q\models \varphi(\bar d_F, d_j)<\e^p\}. 
\]
If $\Sigma$ consists of quantifier-free formulas and $\bfT$ is $\forall\exists$-axiomatizable then the fifth family of dense sets is 
indexed by 
 quantifier-free formulas $\varphi(\bar x,\bar y)$  (recall that all quantifier-free formulas belong to $\Sigma$) 
such that $\sup_{\bar x}\inf_{\bar y}\varphi(\bar x, \bar y)=0$ is an axiom of  $\bfT$, 
finite $F\subseteq \omega$ of the appropriate cardinality, and $k>0$: 
\[
\bfE_{\varphi(\bar d_F, x), k}=\{q\in \bbPTS: (\exists F') \bfT
+q\models \varphi(\bar d_F, \bar d_{F'})<1/k\}. 
\]

\begin{lemma} \label{C.Dense} 
Each of the sets 
$\bfC_{\varphi(\bar d_F, x),r,m}$, 
$\bfD_{d,i,j,k}$, 
$\bfD_{P,F,k}$, 
$\bfD_{g,F,k}$, 
$\bfD_{\inf,p}$, and
$\bfE_{\varphi(\bar d_F, x), k}$
is dense in $\bbPTS$ for every choice of parameters, as long as  $\varphi$ belongs to $\Sigma$. 
\end{lemma} 

\begin{proof} 
We prove that every set of the form $\bfE_{\varphi(\bar d_F, x), k}$ is dense. 
Proofs of the other cases are very similar,  and therefore omitted. 

Suppose $p\in \bbPTS$ and with $n=n^p$ find $M\models \bfT$ and $\bar a\in M^n$
that satisfies $p$. Since 
$\sup_{\bar x}\inf_{\bar y}\varphi(\bar x, \bar y)^M=0$, 
there exists  $\bar b\in M^m$ (where $m$ is the length of tuple $\bar y$) 
such that 
\[
\varphi(\bar a, \bar b)^M<1/k. 
\]
Let $F'$ be an $m$-tuple in $\omega$ such that for all $i<j<m$ and all $l<n$ 
we have $f'(j)\neq f(l)$ and $f'(i)\neq f'(j)$. 
The open  condition 
\[
q=(\max(\psi^p, \varphi(\bar d_F, \bar d_{F'})), \min(\e^p, 1/k))
\]
is satisfied in $M$ by $\bar a,\bar b$ and it extends $p$. 
Since $p$ was arbitrary, this proves that 
$\bfE_{\varphi(\bar d_F, x), k}$ is dense in $\bbPTS$. 
\end{proof} 

All parameters of  dense sets in Lemma~\ref{C.Dense} range over countable sets, 
with the exception of formulas $\varphi$ and conditions $p$. 
Suppose  $L$ is countable.  Then for every $n\in \omega$ 
we can choose a countable dense subset  $\bfD_n$ of  $\bbF_n(L)$ defined in \S\ref{S.Formulas}. 
Then   
\begin{multline*}
\bbFTS=\{
\bfD_{\varphi(\bar d_F), 1/k}, 
\bfC_{\varphi(\bar d_F, x),r,m}, 
\bfD_{d,i,j,k}, 
\bfD_{P,F,k}, 
\bfD_{g,F,k}, 
\bfD_{\inf,p}, 
\bfE_{\varphi(\bar d_F, x), k}:\\
\{ \varphi, \psi\}\subseteq \bigcup_n \bfD_n,\{i,j,k\}\subseteq \omega, F\in \omega^{<\omega}, g\in L, P\in L\}
\end{multline*}
is a countable family of dense subsets of $\bbPTS$. 

\begin{lemma} \label{L.P3.1}  
Suppose  $\bfT$ is a theory  in a 
language $L$  and $\Sigma$ is a set of $L$-formulas 
satisfying $(\Sigma 1)$--$(\Sigma 3)$.
If  $G\subseteq \bbPTS$ is an $\bbFTS$-generic filter then 
there is 
a unique $L^+$-structure $M_G$ with 
the interpretations of  $\{d_j: j\in \omega\}$ as  
a dense subset  and such that 
for all $p\in \bbPTS$ 
 the following holds: 
\begin{enumerate}
\item [$(*_p)$]  If $r\geq 0$ then 
 $(|\psi^p-r|, F^p, \e^p)\in G$ if and only if  
 \[
 M_G\models |\psi^p(\md^p)-r|<\e^p.
 \] 
\end{enumerate}
\end{lemma} 

\begin{proof} Let 
 $M_G^0=\{\md_j:j\in \omega\}$. 
%   such that  for every $p\in \bbPTS$ we have  
%\[
% (\psi^p(\md^p))^{M_G^0}<\e^p. 
%\]
Since $G$ intersects 
$\bfD_{d,i,j,k}$ for all $i,j$, and $k$, 
\[
d(\md_i,\md_j)=r  
\Leftrightarrow
(\forall k\geq 1)(\exists p\in G)  
\bfT+p\models |d(d_i,d_j)-r|<1/k
\]
 defines a metric on $M_G^0$. 
 Since $G$ intersects   
 $\bfD_{P,F,k}$ 
 for every $n$-ary predicate~$P$ in~$L$ and  every $n\in \omega$, 
for $F\in \omega^n$ 
 we can define
 \[
 P(\md_F)=r
\Leftrightarrow
(\forall k\geq 1)(\exists p\in G)  
\bfT+p\models |P(\md_F)-r|<1/k. 
\]
 Let the universe of~$M_G$ be  the metric completion of~$M_G^0$. 
Since  interpretation of every relational symbol
$P$ on $M_G^0$  respects the uniform continuity modulus associated with $P$, 
it has a unique extension to a predicate on $M_G$. 

For every function symbol $g$ in $L$ filter 
$G$ intersects all $\bfD_{g,F,k}$.  Therefore for every 
$\bar d_F$ there exists a  sequence $d_{n(j)}$, for $j\in \omega$, 
in $M_G^0$ such that 
\[
g(\md_F)^{M_G}=\lim_j \md_{n(j)}
\]
 is well-defined. 
This interpretation of $g$ respects the modulus of uniform continuity associated with $g$ and can 
therefore be extended to a function on~$M_G$. 
This completes the definition of $M_G$. 

We shall prove  $(*_p)$ 
  by induction on complexity of the formula~$\psi^p$.  

The case when $\psi^p$ is an atomic formula is immediate from the definition. 
Assume that $(*_q)$  holds for all $q$ such that $\psi^q$ is a proper 
 subformulas of $\psi^p$. 

Suppose $\psi^p=f(\psi_0(\bar d_F), \dots, \psi_{n-1}(\bar d_F))$ for $n\geq 1$, $F\in \omega^{<\omega}$,    continuous function $f$, and 
 formulas $\psi_j$, for $j<n$.  
  We shall prove the direct implication  in $(*_p)$ in the case when $r=0$. Suppose $p\in G$ and 
  let 
\[
r_j=\psi_j(\md_F)^{M_G}
\]
for $j<n$.  Since $f$ is uniformly continuous on bounded sets, we can 
fix a large enough $k\in \omega$ so that $|s_j-r_j|<1/k$ for all $j<n$ implies 
$|f(\bar s)-f(\bar r)|<\e^p$. 
Since $G$ is $\bbFTS$-generic, there are $q_j\in G$ 
  such that $q_j\in \bfD_{\psi_j(\bar d_F), 1/k}$,
for all $j<n$. 
By the inductive hypothesis we have $M\models |\psi_j(\md_F)-r_j|<1/k$ for all $j<n$. 
Therefore $M\models \psi^p(\md_F)<\e^p$, thus proving the direct implication in    $(*_p)$ in case when $r=0$. 

If $r\neq 0$ then the above argument applied when $f$ is replaced with $f-r$ proves the direct implication in  
$(*_p)$. The converse is  automatic since $M_G\models \psi^p(\md^p)=r$ for exactly one $r\geq 0$. 
 
 Suppose  
 $\psi^p$ is of the form 
  $\inf_x \varphi(\bar d_F, x)$ for some $\varphi$ for 
which the claim has been proven.   As before, we first prove the direct implication in $(*_p)$ in case 
when $r=0$. 

Suppose $p\in G$ and fix $q\in G\cap \bfD_{\inf,p}$. 
Then~$q$ and $p$ are compatible because~$G$ is a filter,  
and therefore $\bfT+q\models \varphi(\bar d_F,d_j)<\e^p$ for some  $j\in \omega$. 
Again since $G$ is a filter, the condition  $\varphi(\bar d_F,d_j)<\e^p$ belongs to $G$.  
By applying the inductive hypothesis to this condition and $\varphi$ we have
 $M_G\models \varphi(\md_F,d_j)<\e^p$ and therefore 
 $M_G\models \inf_x\varphi(\md_F,x)<\e^p$. Since $G$ was an arbitrary $\bbF$-generic 
 filter containing $p$, the direct implication in $(*_p)$ holds with $r=0$. 
 If $r\neq 0$ then the above argument applied when 
$\varphi(d_F, x)$ is replaced with $|\varphi(d_F, x)-r|$ 
 proves the direct implication in  
$(*_p)$. As before, the converse is automatic.

%Suppose $p$ is 
%a condition such that  $\psi^p$ is of the form $\inf_x \varphi(\bar d_F, x)$ for some $\varphi$ for 
%which the claim has been proven.   
Since $\sup_x \varphi=1-\inf_x (1-\varphi)$, 
this covers all cases and concludes the inductive proof of $(*_p)$ for all $p$. 

It only remains to prove the uniqueness.  
But if $M$ is an $L^+$-structure that satisfies all conditions in $G$ and 
the set of interpretations $\{d_i^M: i\in \omega\}$ is dense in $M$ then 
$M$ is clearly isometrically isomorphic to $M_G$. 
\end{proof}

\begin{proof}[Proof of Theorem~\ref{P3}]
Given $L$,   $\bfT$,  and $\Sigma$, 
set $\bbF=\bbFTS$  is a  family of dense subsets of $\bbPTS$ 
which by Lemma~\ref{L.P3.1} has the property that every $\bbF$-generic filter  defines a
unique $L^+$-structure $M_G$ 
such that \eqref{I.P3.1} and 
\eqref{I.P3.1.1} hold.

\eqref{I.P3.2} 
If $\Sigma$ consists of all $L$-formulas, then 
\eqref{I.P3.1.1} implies  that $M_G\models \bfT$. 

\eqref{I.P3.3} Assume $\bfT$ is $\forall\exists$-axiomatizable and recall that  $(\Sigma 1)$ implies 
$\Sigma$ includes  all quantifier-free formulas. 
For every quantifier-free formula $\varphi(\bar x,\bar y)$ such that $\sup_{\bar x}\inf_{\bar y}\varphi(\bar x, \bar y)=0$ is an axiom of $\bfT$, 
for every $F$ of the appropriate cardinality,  and every $k>0$ the set 
$\bfE_{\varphi(\bar d_F, x), k}$ is met by the generic filter $G$. 
A simple argument now shows that
  $M_G$ satisfies all  $\forall\exists$-axioms of $\bfT$ 
and  \eqref{I.P3.3} follows.

Now suppose $L$ is countable. Then $\bbFTS$ is also countable and an $\bbF$-generic filter 
$G$  exists by Lemma~\ref{L.CtbleGeneric}. 
\end{proof}

\begin{remark} The proof of Theorem~\ref{P3} used the fact that 
for every sentence $\varphi(\bar d_F)$ in $\Sigma^+$ and $\e>0$  the set $\bfD_{\varphi(\bar d_F), \e}$ 
of conditions that decide 
the $\varphi^{M_G}$ up to $\e$
 is dense in $\bbPTS$. A deeper fact is  
worth mentioning. Even if $\varphi(\bar d_F)$ is 
an $L^+$-sentence that  does not belong 
to $\Sigma^+$ then for every $\e>0$  the set of conditions $p$ in $\bbPTS$ that decide the value of 
$\varphi(\bar d_F)^{M_G}$ for every sufficiently generic $G$ (possibly 
more than merely $\bbF$-generic, but a countable family of dense sets still suffices) 
 up to $\e$ is dense in $\bbPTS$. 
 This is a consequence of Cohen's Truth Lemma, 
 \cite[Lemma~IV.2.24]{Ku:Set}. 
\end{remark}

Principal types were defined in \S\ref{S.Principal}. 

\begin{definition} 
Suppose $\Sigma$ is a set of $L$-formulas 
which satisfies $(\Sigma 1)$--$(\Sigma 3)$ and $\bt(\bar x)$ is an $n$-ary type for some $n\geq 1$. 
  An $n$-ary 
 type~$\bt(\bar x)$ is \emph{$\Sigma$-non-principal}
if there exists $k\geq 1$ 
such that 
for every  $F\in \omega^n$ 
 the set
 (using the $\max$-distance on $M^n$)
\begin{align*} 
\bfD_{\bt(\bar x), F,k}=\{q\in \bbPTS: &M\models \bfT\text{ implies} \\
&
\inf\{\dist(\bar a, \bt(M)): M\models \psi^q(\bar a)<\e^q\} \geq 1/k\}. 
\end{align*} 
 is dense in $\bbPTS$. 
\end{definition}

%A proof of the following well-known fact is analogous to the proof of Theorem~\ref{P3}. 

%%%%%%%%%!!!!!!!!!
\begin{thm} \label{P3.4}
Suppose $\bfT$ is a complete $L$-theory and $\Sigma$ is a set of $L$-formulas 
satisfying $(\Sigma 1)$--$(\Sigma 3)$. 
If 
$\bt$ is a $\Sigma$-non-principal type then there is a family $\bfF_{\bt}$
of dense subsets of $\bbPTS$ such that  if $G$ is $\bfF_{\bt}$-generic 
then  $M_G$  omits~$\bt$. If $L$ is separable then $\bbF_{\bt}$ can be chosen countable. 
\end{thm}

\begin{proof} Suppose $\bt(\bar x)$ is a 
$\Sigma$-non-principal $n$-type. Fix 
 $k\geq 1$ such that $\bfD_{\bt, F, k}$ is dense in $\bbPTS$ for all 
 $F\in\omega^n$. 
With $\bbFTS$ as defined before Lemma~\ref{L.P3.1}, let 
\[
\bbF_{\bt}=\bbFTS\cup \{\bfD_{\bt, F, k}: F\in \omega^n\}. 
\]
 If $G$ is $\bbF_{\bt}$-generic, 
 then (with  $\md_F$ and $M_G^0$ as in the proof of Lemma~\ref{L.P3.1}) 
 we have  $\dist(\md_F, \bt(M_G))\geq 1/k$ for all $F\in \omega^n$. 
Since $\{\md_F: F\in \omega^n\}=(M_G^0)^n$ is dense in~$M_G$, the latter model   omits~$\bt$. 

The last claim follows from the fact that $\bbFTS$ is countable if $L$ is. 
\end{proof}

The following is well-known  (\cite[\S 12]{BYBHU} or \cite[Lecture~4]{Ha:Continuous}; see
also Corollary~\ref{C.Omitting}). 

\begin{coro} \label{C3.4} Suppose $\bfT$ is a complete theory in a countable language 
and~$\bt_n(\bar x)$ is a complete type over $\bfT$ for every $n\in \omega$. 
Then the following are equivalent. 
\begin{enumerate}
\item \label{C.P3.1.1} Each $\bt_n$ is not principal. 
\item \label{C.P3.1.2} Each $\bt_n$ is omissible in a model of $\bfT$. 
\item \label{C.P3.1.3} There exists a family of dense subsets $\bbF$ of 
 $\bbPT$ such that for every $\bbF$-generic filter $G$ model~$M_G$ 
 is a model of $\bfT$ which omits all $\bt_n$. 
 \end{enumerate}
 \end{coro} 
 
 \begin{proof} The equivalence of \eqref{C.P3.1.1} and \eqref{C.P3.1.2} is 
 Lemma~\ref{L.triv},  \eqref{C.P3.1.3} clearly implies \eqref{C.P3.1.2}, 
 and the converse is Theorem~\ref{P3.4}. 
 \end{proof}

\subsection{Forcing with `certifying structures'} 
 Let $\bfT$ be a not necessarily complete theory,  let $\Sigma$ be a set of 
 formulas in the language of $\bfT$ satisfying  closure properties ($\Sigma$1)--($\Sigma$3) as in \S\ref{S.Forcing.Sigma}
 and let     $\fM$ be a nonempty 
set of models of $\bfT$. Forcing $\bbPTSM$ is defined as follows. 
Its conditions  are  triples
 \[
 p=(\psi^p, F^p, \e^p)
 \]
(we shall write $(\psi,F,\e)$ whenever $p$ is clear from the context)
  where $\psi$ is an $n$-ary formula, $F$ is an $n$-tuple of natural numbers, 
    $\e>0$, 
and $\psi(\bar d_F)< \e$
is a condition satisfied in some model in $\fM$. We shall write $\bar d^p$ instead of~$\bar d_{F^p}$.  
We let $p\geq q$ if the following holds. 
\begin{enumerate}
\item [] 
 $F^p\subseteq F^q$ and 
for every $M\in \fM$ and $\bar a$ in $M$ of the appropriate length, 
if $\psi^q(\bar a)^M<\e^q$ then $\psi^p(\bar a)<\e^p$.
\end{enumerate}
If $p\geq q$ we say that $q$ \emph{extends} $p$ or that $q$ is \emph{stronger than} $p$. 
 If $\bfT$ is a complete theory, 
then every condition consistent with $\bfT$ is realized in every model of $\bfT$ and $\bbPTSM$ is 
isomorphic to $\bbPTS$ if $\fM$ is any nonempty set of models of $\bfT$. 

A proof of Theorem~\ref{P3.1} below 
is analogous to the proof of 
Theorem~\ref{P3} and is therefore omitted. 
As before, theory $\bfT$ 
 is not assumed to be  complete and that the language $L$ is not assumed to be separable.

\begin{thm} \label{P3.1} 
Suppose  $\bfT$ is a theory  in a 
language $L$  and 
 that either $\Sigma$ consists of all $L$-formulas or that $\bfT$ is $\forall\exists$-axiomatizable and 
$\Sigma$ includes all quantifier-free formulas.   
 Then there is a  family $\bfF$ of dense subsets of $\bbPTSM$ with the following property. 
If $G$ is an   $\bfF$-generic filter, then there exists a unique
 $L^+$-structure $M_G$ which is a model of $\bfT$ and 
 has the interpretations of  $\{d_j: j\in \omega\}$ as
a dense subset.  
 \qed
\end{thm}

\section{Forcing} 
\label{S.Set-theoretic}
In the present section we discuss the relation with  the set-theoretic forcing, 
in which one constructs a generic extension of a model of ZFC. 
Some acquaintance with the method of 
forcing is required (e.g. \cite{Ku:Set} or \cite{schindler2014set}).  
Readers not interested in forcing may  want to  skip ahead to \S\ref{S.Uniform}.

\subsection{Cohen forcing} 

Recall that two forcing notions $\bbP_0$ and $\bbP_1$ are \emph{forcing equivalent} 
if there is a poset $\bbP_3$ and order-preserving maps $f_j \colon \bbP_j\to \bbP_3$ 
such that $f_j[\bbP_j]$ is dense in $\bbP_3$ for $j<2$ (\cite[Definition~IV.4.25]{Ku:Set}). 
Forcing equivalence of $\bbP_0$ and $\bbP_1$ is 
 equivalent to asserting that for every  $V$-generic filter  $G_j$ in $\bbP_j$ 
there exists a filter $G_{1-j}\subseteq \bbP_{1-j}$ in $V[G_j]$ such that $V[G_j]=V[G_{1-j}]$, for $j<2$
(\cite[Lemma~IV.4.6]{Ku:Set}). 

 The poset for adding a Cohen real has   nonempty rational intervals in $[0,1]$  
as conditions,  ordered by the inclusion. 
  It is forcing-equivalent to every forcing with a countable dense set and no minimal elements 
  (this is 
immediate from e.g. \cite[Lemma~IV.4.26]{Ku:Set}).

\begin{lemma} \label{L.Cohen} \label{L.Cohen.1} 
If $\bfT$ is a theory in a countable language and $ \fM$ is a class of its models
then each of $\bbPTS$ and 
 $\bbPTSM$ 
 has a countable dense subset, 
and is therefore equivalent to the standard forcing for adding a Cohen real. 
\end{lemma} 

\begin{proof} By  the separability of $L$, for every $n$
 there exists a countable set of 
formulas $\bfD_n$ which is $d_{\bfT}$-dense in $\bbF_n(\bfT)$. 
We shall  produce a countable dense set $E$ of conditions  
in $ \bbPTS$ by using  a standard continuous functional calculus trick (\S\ref{S.CFC}).  

 For each $m\in \omega$ let $f_m(t)=\max(t-1/m, 0)$. 
Consider the set $\bfC$  of all 
conditions  in $\bbPTS$ of the form $f_m(\varphi(\bar d))<1/n$, for $m,n$ in $\omega$
and $\varphi\in \bfD$. If $\varphi(\bar d)<\e$ is a condition then there exist $M\models\bfT$ 
and a tuple $\bar a$ in $M$ such that $M\models \varphi(\bar a)=r<\e$. 
Therefore  $f_n(\varphi(\bar d))<\e$ is in $\bbPTS$ for all $n>1/(\e-r)$. 

We claim that the set $\bfC$ of all conditions of this sort  is dense in $\bbPTS$. Since this set 
is countable this will conclude the proof. 
Take a condition $\psi(\bar d)<\e$. Fix $n>1/\e$ and let $\varphi(\bar x)\in \bfD$ be such that 
$d_\infty(\varphi,\psi)<1/(2n)$. For a large enough $m>2n$ we have that 
 $f_m(\varphi(\bar d))<1/m$ is a condition in~$\bfC$
 stronger than $\psi(\bar d)<\e$. 

A proof in the case of $\bbPTSM$ is analogous. 
\end{proof}

The \emph{covering number},   $\cov(\meager)$,  
 for the  ideal of meager (i.e. first category) subsets of $\bbR$ is  
the minimal cardinality of  a family of meager sets  required to 
cover the real line (\cite[Definition~III.1.2 and Definition~III.1.6]{Ku:Set}). 
Since every separable, completely metrizable space $X$ with no isolated points 
has a dense $G_\delta$ subset homeomorphic to the 
 Baire space, $\cov(\meager)$ is equal to the minimal number of first category 
 subsets of $X$ required to cover~$X$.
  
\begin{coro} \label{C.Omitting} 
If $\bfT$ is a complete theory in a countable language, $\kappa<\cov(\meager)$, 
 and $\bt_\gamma$ for $\gamma<\kappa$  
is a set of complete non-principal types over~$\bfT$, then 
$\bfT$ has a separable model that omits all $\bt_\gamma$. 
\end{coro} 

\begin{proof} 
By  Theorem~\ref{P3.4} 
to every $\bt_\gamma$ we associate a dense $\calD_\gamma\subseteq \bbPT$ such that 
if $G$ intersects $\calD_\gamma$ and all dense sets in the family 
$\bbF$ from Theorem~\ref{P3}, then 
 $M_G$ is a model of $\bfT$ 
which omits each $\bt_\gamma$. 

 By Lemma~\ref{L.Cohen.1} the space of all 
 filters in $\bbPTS$ is homeomorphic to the Cantor set, hence it is  
 compact, metrizable and with no isolated points.  
 Therefore  $\cov(\meager)$ is equal to the minimal number of 
 dense subsets of $\bbPTS$ not met by a single filter. 
 We can therefore choose $G$ to meet all dense sets in $\bbF\cup \{\calD_\gamma: \gamma<\kappa\}$ and the generic  model $M_G$ is as required. 
\end{proof} 

\subsection{Absoluteness} \label{S.Absoluteness} In \S\ref{S.Atomic} we consider some properties of $M_G$. 
In order to put these properties  in the proper context, 
we include a brief discussion of absoluteness 
for the convenience of the reader. 
A statement $\varphi(\bar x)$ of ZFC is \emph{absolute} if  for any two transitive models
$V\subseteq W$ of a sufficiently large fragment of ZFC and  for all parameters  $\bar a$ in $V$ of the appropriate 
length one has 
\[
\varphi(\bar a)^V\Leftrightarrow \varphi(\bar a)^W
\]  
(see   \cite[\S II.4]{Ku:Set}). 
If $V\subseteq W$ are models of a large enough fragment of ZFC then 
a complete metric structure $N$ in $V$ is identified with its completion $\tilde N$ in the larger model $W$.   
A property $\Theta$ of metric structures  is \emph{absolute} if for any two transitive models
$V\subseteq W$ 
of a sufficiently large fragment of ZFC, 
for every metric structure $N$ in $V$ we have $\Theta(N)^V\Leftrightarrow \Theta(\tilde N)^W$. 
A definition of a subset $P(x,N)$ of a metric structure $N$ (such as the set of types realized in a model), 
is \emph{absolute}  if for every $a\in N^V$ we have $P(a,N)^V$ if and only if $P(a,\tilde N)^W$.

Given a  language $L$ in the logic of metric structures  and a string of characters $\varphi$, 
the assertion `$\varphi$ is an  $L$-formula' 
is absolute. This essentially follows from  
\cite[Lemma~II.4.14, Theorem~II.4.15, Corollary~II.4.17]{Ku:Set} 
and the discussion in between, where the analogous statement for   the first-order logic was proved. 
The semantics of $L$---i.e. the definition of the interpretation of a formula $\varphi$ in a metric structure---is also proved 
to be absolute by a routine induction on the complexity of $\varphi$. 

Suppose $V\subseteq W$ are 
 transitive models of a sufficiently large fragment of ZFC and $L$ is a language in $V$. 
The linear space of $L$-formulas in $V$ (\S\ref{S.Formulas})  
is in general a proper subspace of the space of $L$-formulas in $W$. 
It is however always dense with respect to the metric defined in 
\S\ref{S.preliminaries} 
 (this is a consequence of the Stone--Weierstrass theorem, 
since all  formulas built by  using polynomial functions with rational coefficients belong to $V$), 
and therefore the space of $L$-formulas in $W$ is the  completion of the space of $L$-formulas
in $V$. An $L$-theory is identified with a continuous functional on the space of $L$-formulas. 
Therefore  every  $L$-theory $\bfT$ in $V$  has a unique extension to an $L$-theory in $W$, also denoted $\bfT$.   
By identifying a $k$-type with a functional on the space of $L$-formulas with free variables included in $x_j$, for $j<k$,   
we similarly identify a type over $\bfT$ in $V$ with its unique extension to a type over $\bfT$ in $W$.

\begin{prop} Suppose $L$ is a language of metric structures, $M$ is an $L$-structure,  $\varphi$ is an $L$-formula, 
 $\bar a$ is a tuple in  $M$ of the appropriate length, and $r\in \bbR$. 
The following are absolute between models of a large enough fragment of ZFC. 
\begin{enumerate}
\item \label{P.ABS.Sat} $\varphi(\bar a)^M=r$.   
\item  \label{P.ABS.Type} Type $\tp_{M}(\bar a)$ of  an $n$-tuple $\bar a$ in $M$. 
\item \label{P.ABS.Th} $\Th(M)$. 
\item \label{P.ABS.Types} The set of types realized in $M$. 
\item \label{P.ABS.d} the distance $d(\bt,\bs)$ between complete types as defined in \S\ref{S.Metric}. 
\item  \label{P.ABS.principal} The assertion `$\bt$ is principal.'
\end{enumerate}
\end{prop} 

\begin{proof} Suppose $V\subseteq W$ are models of a large enough fragment of  ZFC and~$M$ is in $V$.  

\eqref{P.ABS.Sat} 
Since~$M$ is dense in $\tilde M$ and the interpretations of $L$-formulas are uniformly continuous, 
the absoluteness of  $\varphi(\bar a)^M=r$ is  proved 
 by  induction on the complexity of~$\varphi$. 
 Clearly \eqref{P.ABS.Sat} implies  
 \eqref{P.ABS.Type}. 
 
\eqref{P.ABS.Th} 
The  theory of $M$ is  
identified with  an affine functional on $\frW_0(L)$,  
has  the unique continuous extension to an affine functional 
on $\tilde{\frW_0(L)}$ 
of the theory of $M$ to 
the completion 
 $\tilde {\frW_0(L)}$ of (the ground-model) $\frW_0(L)$.
 
\eqref{P.ABS.Types} The set of types realized in $M$ is $\bSigma^1_1$ and therefore absolute by Lemma~\ref{L.Sigma11}. 

\eqref{P.ABS.d}  is a consequence of Lemma~\ref{L.d.Borel}. 

\eqref{P.ABS.principal} By \eqref{P.ABS.d}, for a type $\bt$ (complete or not) having 
a metric $\e$-neighbourhood which is nowhere dense in 
the logic topology is absolute.  
\end{proof}

\subsection{Omitting types in  the generic model $M_G$}\label{S.Atomic} 
Recall (\S\ref{S.Forcing.Sigma}) 
that~$\bbPT$ denotes the so-called `infinite' forcing, i.e.~$\bbPTS$ in  the case when~$\Sigma$ is 
the set of all formulas of the language of~$\bfT$. 

\begin{lemma} \label{L.4.5} 
If  a complete theory $\bfT$ in a countable language  
has an atomic model $N$ then $\bbPT$ forces that $M_G$ is atomic, 
and therefore isometric to the completion of $N$.  
 \end{lemma} 

\begin{proof} 
By Theorem~\ref{P3.4} there are countably many dense subsets of 
$\bbPT$ such that if $G$ meets each one of them then 
every nonprincipal type is omitted in $M_G^n$. 
Therefore every complete $n$-type for every $n\geq 1$ 
realized in $M_G$ is principal and $M_G$ is  an 
atomic model of $\bfT$. 
By uniqueness of the atomic model of $\bfT$ (\cite[Corollary~12.9]{BYBHU}), 
  $M_G$ is isometric to the completion of $N$. 
\end{proof}

Lemma~\ref{L.4.5} can be recast as the assertion  that if $\bfT$ has an atomic model then 
 every omissible type is forced to be omitted in the generic model. 
Surprisingly,  the assumption that $\bfT$ has an atomic model cannot be dropped from this assertion (Theorem~\ref{C.generic}).

\subsection{Strong homogeneity of $\bbPTS$}\label{S.strong-homogeneity}
A forcing notion  is \emph{homogeneous} if for any two conditions $p$ and $q$ there exists an automorphism~$\Phi$ such that~$\Phi(p)$ is compatible with $q$. Since the Cohen forcing is homogeneous, 
  $\bbPTS$ is equivalent to a homogeneous forcing by Lemma~\ref{L.Cohen}. This does not imply that $\bbPTS$ is
  homogeneous itself, but   
we shall prove that even more is true in case when $\bfT$ is complete. 
By  $S_\infty$ we denote
 the group of \emph{all} permutations of $\omega$.
To a permutation $h\in S_\infty$ 
we  associate an automorphism $\Phi_h$ of $\bbPTS$ which sends $d_j$ to $d_{h(j)}$ 
for all $j\in \omega$ (writing $h[F]=(h(f_0), \dots, h(f_{n-1}))$ where $n$ is the length of $F$)
 \[
 \Phi_h((\psi,  F, \e))= (\psi, h[F], \e). 
\]

\begin{lemma} \label{L.Homogeneous} 
Assume  $\bfT$ is a complete  $L$-theory and $\Sigma$ is a set of $L$-formulas 
satisfying $(\Sigma 1)$--$(\Sigma 3)$. 
For any two conditions $p_1$ and $p_2$ 
in $\bbPTS$ there is $h\in S_\infty$ such that $p_1$ and $\Phi_h(p_2)$ are compatible. 
\end{lemma} 

\begin{proof} Let $p_j=(\psi_j, F_j, \e_j)$ for $j<2$. We shall write $d(j)$ and $x(j)$ 
for $d_{F_j}$
and $x_{F_j}$, respectively. 
By Lemma~\ref{L.CFC} we 
 may assume $\e_1=\e_2=1$.  
Since $\bfT$ is complete we have  
$\bfT\models \inf_{\bar x(j)}\ \psi_j(\bar x(j))<1$
for $j<2$ and therefore 
\[
\bfT\models \max(\inf_{\bar x(1)} \psi_0(\bar x(0)), \inf_{\bar x(2)}\psi_1(\bar x(1)))<1.
\]
Let   $h$ be such that the tuples $h[F_1]$ and  $F_2$ have no common entries.
Then 
\[
q=(\max(\psi_0(\bar d(0)),\psi_2(\bar d_{h(F_1)}), F_0\cup h(F_1), \e)
\]
is a condition in $\bbPTS$ which extends both $p_0$ and $\Phi_h(p_1)$. 
\end{proof}

In the following $M_G^0$ denotes the countable dense submodel of $M_G$
whose universe consists of interpretations
of  constants $d_j$, for $j\in \omega$, 
  as defined in the proof of Theorem~\ref{P3}. 
We introduce a convenient ad-hoc terminology. 
A statement $\Theta(x)$ of ZFC (possibly with parameters) 
is \emph{symmetric} if for every generic filter $G$ and every 
$h\in S_\infty$ we have 
\[
\Theta(M^0_G)\leftrightarrow \Theta( M_{\Phi_h(G)}^0).
\] 
For example, ``$M_G$ omits type $\bt$'' and ``$M_G\models \bfT$''  
are both symmetric but ``$d(\md_1, \md_2)<1/2$'' is not. 

\begin{coro} \label{C.Homogeneous} 
Assume  $\bfT$ is a complete  $L$-theory and $\Sigma$ is a set of $L$-formulas 
satisfying $(\Sigma 1)$--$(\Sigma 3)$.
 If $\Theta(\bar y)$ is a symmetric  statement of ZFC  with  parameters in the ground model, 
then $\bbPTS$ either  forces $\Theta(M_G^0)$ 
 or it forces $\lnot \Theta(M_G^0)$. 

In particular, for every ground-model 
type $\bt$ forcing notion $\bbPTS$ either forces that $M_G$ realizes $\bt$ or it forces that 
$M_G$ omits $\bt$. 
 \end{coro} 
%%%%!!!!!

\begin{proof} Fix  $p\in \bbPTS$ which decides $\Theta(M_G^0)$. By replacing $\Theta$ with its negation if 
needed, we may assume that $p$ forces $\Theta(M_G^0)$. 
For every $q\in \bbPTS$ by  Lemma~\ref{L.Homogeneous} there exists an $h\in S_\infty$   
such that $\Phi_h(p)$ is compatible with $q$. 
But $\Phi_h$ is an automorphism of  $\bbPTS$ that sends $M_G^0$ to itself and $M_G$ to itself, and since $\Theta$ is symmetric 
$\Phi_h(p)$ forces  $\Theta(M_G^0)$. 
This implies that every condition in $\bbPTS$ decides $\Theta(M_G^0)$ the same way that $p$ does. 
\end{proof}
 
Let $\bbPT$ be a forcing notion of the form $\bbPTS$ or $\bbPTSM$. 
 To every first-order property $\Theta$ of a metric structure one can associate a countable
 family~$\bbF_\Theta$ of dense subsets of $\bbP$  
such that for every $p\in \bbP$  
we have $p\forces\Theta(M_G)$
if and only if for every $\bbF_\Theta$-generic $G\subseteq \bbP$ containing $p$ 
we have $\Theta(M_G)$. 
Therefore it is not necessary to pass to the forcing extension in order to find 
a sufficiently generic model  $M_G$.

If $\bfT$ is a complete theory in a countable language that has an atomic model then 
Lemma~\ref{L.4.5} implies that $\bbPT$ forces $M_G$ is atomic. 
In particular,~$\bbPT$ forces that $M_G$ omits a ground-model type $\bt$ if and only if it
 is non-principal.  The set of all nonprincipal $n$-types is closed in $S_n(\bfT)$ 
because a type is principal if and only if it is an isolated point
 in the logic topology, by \cite[Lecture 3]{Ha:Continuous}.

 If $\bfT$ does not have an atomic model 
 then Lemma~\ref{L.forced} below gives a complexity estimate for the set of types forced 
 to be omitted in $M$. 
  If a theory~$\bfT$ is not complete then the theory of $M_G$, 
as well as the set of types omitted in~$M_G$,  depends on the choice of $G$.

\begin{lemma} \label{L.forced} Suppose $\bfT$ is a  theory in a countable language $L$
and $\Sigma$ is a set of $L$-formulas satisfying $(\Sigma 1)$--$(\Sigma 3)$. 
\begin{enumerate}
\item   For every $q\in \bbPTS$ the set of all types forced by $q$ 
 to be omitted in~$M_G$ is a $\mathbf \bPi^1_1$-set. 
 \item  The set of all types such that some $q\in \bbPTS$ forces   
  to be omitted in~$M_G$ is a $\mathbf \bPi^1_1$-set. 
\item If  $\bfT$ is in addition complete then 
 for every type $\bt$ we have that $\bbPTS$ either forces that 
  $\bt$ is realized in $M_G$ or it 
forces that $\bt$ is omitted in~$M_G$. 
\end{enumerate}
 \end{lemma} 

\begin{proof} 
Let $\bbP_0$ be 
a fixed countable dense subset of $\bbPTS$ as in Lemma~\ref{L.Cohen}. 
A name  for a function $h\colon \omega\to\omega$ 
can be identified with a set of triplets $(p,m,n)$ such that $p\forces \dot h(m)=n$. 
Such set $Z$ of triplets is a name for a function if and only if for every $m\in \omega$ 
the following hold.
\begin{enumerate}
\item 
The set of 
conditions $p$ for which  there exists $n$ satisfying 
$(p,m,n)\in Z$ is dense in $\bbP_0$, and 
 \item If $p$ and $q$ are compatible conditions in $\bbP_0$, $(p,m,n)\in Z$, and $(q,m,k)\in Z$, 
 then $n=k$. 
 \end{enumerate}
The set of such names is clearly  a Borel subset of $\bbP_0\times \omega^2$.  

(1) 
Fix  $q\in \bbPTS$. Lemma~\ref{L1+} implies that
$q$ does not force that $\bt$ is omitted in $M_G$ if and only 
if there exists $p\leq q$ in $\bbP_0$ and a name $\dot h$ for an infinite branch of  
 the tree $T_{D,\bt}$ (with $D=M_G^0$).  Condition $p$ forces that~$\dot h$ is an infinite 
 branch of $T_{D,\bt}$ if and only if for every $n$ the set of $r\in \bbP_0$ such 
 that~$r$ decides the values of $\dot h(j)$, for $j<n$, and that   
($\bt_\omega 1$) and ($\bt_\omega 2$) of \S\ref{S.pomega} hold for $j<n$.
The set of pairs $(p,\dot h)$  where $\dot h$ is a name for a function and these two conditions hold is clearly 
Borel.

(2) This is a consequence of (1) because $\bbPT$ has a countable dense set and a
countable union of $\bPi^1_1$ sets is $\bPi^1_1$. 

(3) 
This is an immediate consequence of Corollary~\ref{C.Homogeneous}. 
\end{proof}

\section{Uniform sequences of types}\label{S.Uniform}
Suppose $\bfT$ is a theory in language $L$. 
For $m\in \omega$ 
a sequence~$\bt_n$, for $n\in \omega$,  of $m$-ary types over $\bfT$ is \emph{uniform} if 
there are $m$-ary formulas $\varphi_i(\bar x)$ for $i\in \omega$ such that 
\begin{enumerate}
\item $ \bt_n(\bar x)=\{\varphi_i(\bar x)\geq 2^{-n}: i\in \omega\}$ for every $n$, and 
\item  all $\varphi_i$, for $i\in \omega$,  have the same modulus of uniform continuity. 
\end{enumerate}
If $\bt_n$ and $\varphi_i$ are as above, 
then
\[
\psi(\bar x)=\inf_{i\in \omega} \varphi_i(\bar x)
\]
is an  $L_{\omega_1\omega}$ formula (see \cite{ben2009model}) whose 
modulus of uniform continuity is equal to the joint modulus of uniform continuity of $\varphi_i$, for $i\in \omega$.
Therefore its interpretation is uniformly continuous in every $L$-structure $A$ and 
 $\sup_{\bar x}\psi(\bar x)^A=0$ if and only if $A$ omits all $\bt_n$ for $n\in \omega$. 

For the simplicity of notation in the following we consider a single-sorted language. 

\begin{lemma} \label{L.uniform.1} 
Suppose  $\bfT$ is a theory  in a 
language $L$  and 
 $\bt_n$, for $n\in \omega$, is a uniform sequence 
of $m$-types over $\bfT$.  If~$M\models \bfT$ then 
\[
Z=\{\bar a\in M^m: \text{for all $n$, $\bt_n$ is not realized by $\bar a$ in $M$}\}
\]
is a closed subset of $M^m$. 
In particular, if  $D$ is a dense  subset of $M$, then 
 all~$\bt_n$ are omitted in $M$ if and only if 
 none of them is realized by any $m$-tuple of elements of $D$. 
 \end{lemma}

\begin{proof} 
Set $Z$ is closed as the zero set of the interpretation of the continuous 
infinitary formula $ \inf_{i\in \omega}  \varphi_i(\bar a)$.   
The last sentence of the lemma follows immediately. 
\end{proof}

The  syntactic characterization of omissible  uniform sequences of types given below is analogous to the syntactic characterization 
 of  complete omissible types given in \cite{BYBHU}.  As Ita\"\i{}  Ben Yaacov and Todor Tsankov pointed out, 
 the set $X$ of complete types extending a type in a uniform sequence of types is metrically open (\S\ref{S.Borel.types}) and therefore 
by a standard argument (see \cite[\S 12]{BYBHU} or \cite[Lecture 4]{Ha:Continuous}) types in $X$ are simultaneously omissible if and only if 
$X$ is meager in the logic topology (\S\ref{S.Borel.types}). 
We spell out details of the proof below since we will need a similar 
argument in the case when the theory $\bfT$ 
is not necessarily complete in Theorem~\ref{T.Uniform.1}. 
As in    \S\ref{S.Principal},  
type $\bt$ (complete or not) is said to be \emph{principal}  
if for every  $\e>0$  the set
\[
B_\e(\bt)=\{\bs\in S_n(\bfT): \dist(\bs,K_{\bt})<\e\}
%B_\e(\bt)=\{\bs\in S_n(\bfT): \inf\{d(\bt',\bs): \bt'\in S_n(\bfT), \bt\subseteq \bt'\}<\e\}
\]
is somewhere dense in the logic topology with respect to $\bfT$.

\begin{thm} \label{T.Uniform} 
Suppose   $\bfT$ is a complete theory in  a countable language $L$.  
If for every $m\in\omega$  we have a uniform sequence of types 
 \[
 \bt^m_n=\{\varphi^m_j(\bar x(m))\geq 2^{-n}: j\in \omega\},\text{ for }n\in \omega,
 \]
    then the following 
are equivalent. 
\begin{enumerate}
\item None of the types  $\bt^m_n$, for $m,n\in \omega$,  is principal. 
\item Theory 
$\bfT$ has a model that omits all $\bt^m_n$, for all $n\in \omega$ and $m\in \omega$.   
\item There are no $\delta>0$, $m\in \omega$,  and condition $\psi(\bar x(m))<\e$ such that $\bfT\models \inf_{\bar x(m)} \psi(\bar x(m))<\e$ and 
$\bfT+\psi(\bar x(m))<\e\models \varphi^m_j(\bar x(m))\geq \delta$ for every $j\in \omega$. 
\end{enumerate}
\end{thm} 

\begin{proof}  (1) and (2) are equivalent by 
Corollary~\ref{C3.4}. 

In order to prove that (3) implies (1), 
 assume (1) fails and let $m$ and $n$ be such that $\bt^m_n$ is principal. 
Since types $\bt^l_n$ for $l\neq m$ do  not appear in this proof, we shall write $\bar x$ 
in place of~$\bar x(m)$ throughout for simplicity.  
 Since all formulas 
 $\varphi^m_j$, for $j<\omega$,  have the same modulus of uniform continuity we can find 
 $\delta>0$  such that for all $\bar x$ and $\bar y$ of the appropriate sort 
 $\max_i d( x_i, y_i)<\delta$ implies  $|\varphi^m_j(\bar x)-\varphi^m_j(\bar y)|\leq 2^{-n-1}$ 
for all $j$. Since $\bt^m_n$ is principal, there is a condition  $\psi(\bar d)<\e$ such that 
in every  $M\models \bfT$ we have
\[
\{\bar a: \psi(\bar a^M)<\e\}\subseteq \{\bar a: \dist(\bar a, \bt^m_n(M))<\delta\}. 
\]
Therefore  $\bfT+\psi(\bar x)<\e\models \varphi^m_j(\bar x) \geq 2^{-n}-2^{-n-1}= 2^{-n-1}$ for all $j$
and~(3) fails. 

In order to prove that (2) implies (3),  assume (3) fails. Fix 
a condition $\psi(\bar x)<\e$ and 
$\delta>0$ such that     $\bfT\models \inf_{\bar x}\psi(\bar x)<\e$ and 
$\bfT+\psi(\bar x)<\e\models \varphi^m_j(\bar x)\geq \delta$ for all $j$.  
 If $2^{-n}<\delta$ then clearly every model of $\bfT$ realizes~$\bt^m_n$, hence (2) fails. 
\end{proof}

Theorem~\ref{T.Uniform} has an analogue in the case when theory $\bfT$ is not necessarily complete. 

\begin{thm} \label{T.Uniform.1} 
Suppose  $\bfT$ is a not necessarily complete theory  in a 
countable language $L$. 
If for every $m\in\omega$  we have a uniform sequence of types 
 \[
\bt^m_n=\{\varphi^m_j(\bar x(m))\geq 2^{-n}: j\in \omega\},\text{ for }n\in \omega,
 \]
    then the following 
are equivalent. 
\begin{enumerate}
\item $\bfT$ has a model omitting all $\bt^m_n$, for all $m$ and $n$ in $ \omega$.  
\item 
There are no $\delta>0$, finite $F\subseteq \omega$,  and   conditions $\psi_m(\bar x(m))<\e$ for $m\in F$ 
such that $\bfT\models \min_{m\in F}\inf_{\bar x(m)}  \psi_m(\bar x(m))<\e$ and 
\[
\bfT+\psi_m(\bar x(m))<\e\models \varphi^m_j(\bar x(m))\geq \delta
\]
 for every   $j\in \omega$
and $m\in F$. 
\end{enumerate}
\end{thm} 

\begin{proof} Recall that by Lemma~\ref{L.CFC} 
condition (2) is equivalent to its special case   in which $\e=1$. 

By Theorem~\ref{T.Uniform} it suffices to show that 
$\bfT$ satisfies (2) if and only if it can be extended to a complete theory that still satisfies~(2). 
Only the direct implication requires a proof, analogous to the corresponding 
 proof in the first order case. Let $\Theta_k$, for $k\in \omega$, 
 enumerate a countable dense set of $L$-sentences such that each sentence occurs infinitely often. 
 %We shall use continuous functional calculus (\S\ref{S.CFC}). 
 By Lemma~\ref{L.CFC} for a closed interval $V\subseteq [0,1]$ and a sentence $\Theta$
the  condition $\Theta\in V$ is equivalent to one of the form $\Theta'=0$ for some $\Theta'$. 
 
 Assume $\bfT$ satisfies (2). 
 We shall find an increasing chain of theories 
 $\bfT=\bfT_0\subseteq \bfT_1\subseteq \bfT_2\subseteq\dots$ 
 and closed intervals $U_{n}\subseteq [0,1]$ 
 of diameter at most $2^{-n}$  
 such that for all $n$ we have   
 \[
 \bfT_n\models \Theta_n\in U_{n}
  \]
 and $\bfT_n$ still satisfies (2). 
 
 Assume that for some $n$ both  $\bfT_n$ and $U_{k}$ for $ k\leq n$ as required have been 
 chosen. 
 Fix  a finite cover $\cV$ of $[0,1]$ by closed intervals of diameter at most~$2^{-n+1}$. 
 We claim that there exists $V\in \cV$ such 
 that  $\bfT_n\cup \{\Theta_n\in V\}$ (identified 
 with a theory by using the closed condition case of Lemma~\ref{L.CFC}) still satisfies~(2). 
 
 Assume otherwise. 
  For every $V\in \cV$ we can find  a finite $F(V)\subseteq\omega$ such that for every $m\in F(V)$  
there exist     an open  condition $\psi_{V,m}(\bar x(m))<1$ 
 and $\delta(V,m)>0$  such that 
the following two conditions hold:  
 \begin{align*} 
 \bfT_n+\Theta_{n+1}\in V&\models \min_{m\in F(V)} \inf_{\bar x(m)} \psi_{V,m}(\bar x(m))<1\\
 \bfT_n+\Theta_{n+1}+\psi_{V,m} (\bar x(m))<1&\models 
 \inf_{j\in \omega} \varphi^m_j(\bar x(m))\geq \delta(V,m). 
 \end{align*} 
 Let
 \[
 \Theta'=\min_{V\in \cV} \dist(\Theta_{n+1}, V) \quad \text{and}\quad
 \delta=\min_{V\in \cV, m\in F(V)} \delta(V,m). 
 \] 
Then 
 $\bfT_n\models \Theta'=0$ and 
 $\bfT_n\models \min_{V\in \cV, m\in F(V)}\inf_{\bar x(m)}  \psi_{V,m}(\bar x(m))<1$ but for each $V\in \cV$ we have
 \[
 \bfT_n   +\psi_{V,m}(\bar x(m))<1\models \inf_{j\in \omega} \varphi^m_j(\bar x(m))\geq \delta, 
 \]
 contradicting (2). 
 This contradiction implies that  we can  find $V\in \cV$ such that adding condition $\Theta_n\in V$ to $\bfT_n$ preserves condition~(2). 

Once all $\bfT_n$ for $n\in \omega$  are constructed, theory 
 $\bfT_\infty=\bigcup_n \bfT_n$ decides the value of each $\Theta_n$
 and is therefore complete. Since the finitary condition (2) is satisfied by every finite fragment 
 of $\bfT_\infty$,  it is satisfied by $\bfT_\infty$ and Theorem~\ref{T.Uniform} can now be applied to $\bfT_\infty$. 
\end{proof}

\subsection{Uniform sequences of types and forcing} \label{S.udt} 
A class $\fM$ of models is \emph{\udt{}} 
if there is a set of sequences of 
uniform types $\langle\bt^m_n:n\in \omega\rangle$, 
for $m\in I$ such that $A$ is in $\fM$ if and only if it omits 
all of these types. This is a special case of being ``definable by uniform families of formulas''
as in \cite[Definition~5.7.1]{Muenster}.

If $\Sigma$ is the set of all formulas of the language of $\bfT$ then we denote $\bbPTSM$ by $\bbPTM$. 

\begin{prop} \label{P4} Assume $\fM$ is a nonempty class of models of a not necessarily complete theory $\bfT$. 
If $\bt_n$, for $n\in \omega$, is  a uniform sequence of types that 
are omitted in every model in $\fM$ then $\bbPTM$ forces that 
$M_G$ omits all~$\bt_n$. 
\end{prop}

\begin{proof} Since 
 $\{\md_j: j \in \omega\}$ is dense 
in~$M_G$, Lemma~\ref{L.uniform.1} implies that $\bt^m_n$ is realized in $M_G$ if and only if it is realized by some $ \md_F$
Suppose that some condition $p$ forces that a tuple $F\subseteq F^p$ realizes 
$\bt^m_n$ for some $m$ and  $n$ in~$ \omega$.  But there is $M\in \fM$ and a tuple $\bar a$ in $M$ of the appropriate sort such that 
$M\models\psi^p(\bar a)<\e^p$. Since $M$ omits $\bt^m_n$, we can extend $p$ to a condition 
that decides that $F$ does not satisfy some condition in $\bt^m_n$, contradicting our assumption on $p$. 
\end{proof} 

Some of the most 
important properties of \cstar-algebras are \udt. 
This includes 
  being AF, UHF (\cite{Mitacs2012}),  $C(X)$ for a chainable continuum (\cite{eagle2015pseudoarc}), 
 TAF,   Popa algebra or quasidiagonal, 
 having nuclear dimension $\leq n$,  decomposition rank $\leq n$, for $n\in \omega$,  
  (\cite[Theorem~5.7.3]{Muenster}). (See however \cite{goldbring2015omitting} and \cite{goldbring2015model}
  for some negative examples.)
  These types are particularly simple and 
 we include a straightforward 
 technical sharpening of Proposition~\ref{P4} with an eye to potential applications.

A uniform sequence of types 
$\bt_n=\{\varphi_j(\bar x)\geq 2^{-n}: j\in \omega\}$, for $n\in \omega$, 
 is \emph{universal} if every $\varphi_j(\bar x)$ is of the form $\inf_{\bar y}\psi(\bar y, \bar x)$ for some quantifier-free formula $\psi(\bar y,\bar x)$. 
 (By Lemma~\ref{L.CFC} a condition of the form $\inf_{\bar y}\psi(\bar y, \bar x)\geq 2^{-n}$ is equivalent to a condition of the form 
 $\sup_{\bar y} \psi'(\bar y, \bar x)=0$, and formulas on the left-hand side of this expression  are commonly recognized as universal given the somewhat controversial convention 
 that $M\models\Theta=0$
 is interpreted as `$\Theta$ is true in $M$'.)  
 The following theorem probably follows from the results from \cite{ben2009model} but we include it for  reader's convenience. 
 
\begin{prop} \label{P4+} Assume that  $\fM$ is a nonempty class of models of a (not necessarily complete) theory~$\bfT$
and  $\Sigma$ is a set of formulas in the language of $\bfT$ that satisfies $(\Sigma 1)$--$(\Sigma 3)$. 
If~$\bt_n$, for $n\in \omega$, is  a uniform sequence of universal types that 
are omitted in every model in $\fM$ then $\bbPTSM$ forces that~$M_G$ omits all $\bt_n$. 

If $\bfT$ is $\forall\exists$-axiomatizable then $\bbPTSM$ forces that 
the generic model satisfies~$\bfT$. 
\end{prop}

\begin{proof} 
Lemma~\ref{L.uniform.1} implies that $M_D$ realizes  $\bt_m$ if and only if  some~$\md_F$ satisfies it.
Suppose  that some condition $p$ forces that a tuple $F\subseteq F^p$ 
realizes~$\bt_n$ for some $n\in \omega$.  There are $M\in \fM$ and a tuple $\bar a$ in $M$ of the appropriate sort such that 
$M\models\psi^p(\bar a)<\e^p$. Since $M$ omits $\bt_n$, 
there exists  a condition $\inf_{\bar y}\psi(\bar y, \bar x)\geq 2^{-n}$ in $\bt_n$ such that 
$\inf_{\bar y} \psi(\bar y, \bar a)^M<2^{-n}-\e$ for some $\e>0$ and~$\psi$ is 
quantifier-free. 
If $F'$ is  of cardinality $|\bar y|$, by the open case of Lemma~\ref{L.CFC} 
we have that  
$\psi(\bar d_{F'}, \bar d_F)<2^{-n}-\e$ is equivalent to a condition in $\bbPTSM$. 
Then $M$ certifies that this condition is  compatible with $p$, and it 
 decides that $d_F$ does not realize $\bar t_n$. 

Since $\bfT$ is $\forall\exists$-axiomatizable, $M_G\models \bfT$ 
by \eqref{I.P3.3} of Theorem~\ref{P3}. 
\end{proof}

If   $\fM$ is a class  of $L$-models then $M\in \fM$ is 
 \emph{existentially closed} (often abbreviated as \emph{e.c.}) for $\fM$
  if whenever $N\in \fM$ is such that $M$ is isomorphic to a submodel of $N$, 
 $\bar a\in M$, and $\varphi(\bar x, \bar y)$ is a quantifier-free $L$-formula
 then 
 \[
 \inf_{\bar y} \varphi(\bar a,  \bar y)^M=\inf_{\bar y}\varphi(\bar a, \bar y)^N.
 \]
 Existentially closed \cstar-algebras and II$_1$ factors were studied in \cite{goldbring2014kirchberg}, \cite[\S 6]{Muenster}, 
 and \cite{FaGoHaSh:Existentially}.

\begin{coro}\label{C.ec} 
Assume $\bfT$ is an $\forall\exists$-axiomatizable
theory in a countable language and $\bt_n$, for $n\in \omega$, is a 
uniform sequence of universal types over~$\bfT$. 
If the class $\fM$ of all models of $\bfT$ that omit all $\bt_n$, for $n\in \omega$, 
is nonempty then it contains a model that is existentially closed for $\fM$. 
\end{coro}

\begin{proof} Let $\Sigma$ be the set of all quantifier-free formulas in the language of $\bfT$. 
Therefore $\bbPTS$-generic models are ``finitely generic'' in the terminology of 
A. Robinson (as used in  \cite{goldbring2014kirchberg}).   
By Proposition~\ref{P4+},  the $\bbPTS$-generic model~$M_G$ is  
forced to omit all $\bt_n$ and to satisfy $\bfT$,  
and is therefore forced to belong to~$\fM$.    

In order to prove that $M_G$ is forced to be 
existentially closed suppose it 
is isomorphic to a submodel of a model $N$ of $\bfT$. 
By continuity it suffices to assure that for 
  a countable dense set of 
 quantifier-free $L^+$-formulas
 $\varphi(\bar d_F, \bar y)$  
we have
 \begin{enumerate}
 \item \label{Eq.ec}
$ \inf_{\bar y} \varphi(\bar d_F,  \bar y)^{M_G}=\inf_{\bar y}\varphi(\bar d_F, \bar y)^N$.
 \pushcounter
 \end{enumerate}
Fix  $\varphi(\bar d_F, \bar y)$ and $r\in \bbQ$.  
Let $k$ be the length of $\bar y$. 
Consider the set~$\bfD_{\varphi(\bar d_F,\bar y),r}$ of conditions
$q\in \bbPT$ such that 
$F^q\supseteq F$ and
one of the following holds. 
\begin{enumerate}
\popcounter
\item $\bfT+q\models \inf_{\bar y}\varphi(\bar d_F, \bar y)\geq r$, or 
\item $\bfT+q\models \inf_{\bar y}\varphi(\bar d_F, \bar y)< r$.  
 \end{enumerate}
We claim that this set is dense in $\bbPT$. Suppose $p\in \bbP$ is such that 
$\bfT+p\not\models  
\inf_{\bar y}\varphi(\bar d_F, \bar y)\geq r$. 
Fix 
$N$ that satisfies $p$ and 
$\inf_{\bar y}\varphi(\bar d_F, \bar y)^N<r$. 
Since $\varphi$ is quantifier-free, we can extend $p$ to decide 
$\bar b$ in $N$ 
such that $\varphi(\bar d_F, \bar b)^N<r$. 
This gives a condition $q\leq p$ in $\bfD_{\varphi(\bar d_F,\bar y),r}$, 
and proves the density. 

Let $\varphi_n(\bar x, \bar y)$, for $n\in \omega$, be an enumeration of a dense 
set of $L^+$ formulas with two distinguished tuples of free variables. 
Choose  filter  $G$ of $\bbPT$ that  meets dense open subsets 
of the form~$\bfD_{\varphi_n(\bar d_F,\bar y),r}$ for all finite $F\subseteq \omega$ of the appropriate size 
and all  $r\in \bbQ_+$. 
Then $M_G$ satisfies 
 \eqref{Eq.ec} 
for each $\varphi_n$ and each~$F$ of the appropriate size. 
By the continuity, $M_G$ is existentially closed. 
 \end{proof}

We conclude this section with  a complexity result related  to Proposition~\ref{P.Borel}  and 
Theorem~\ref{T0} (see  \S\ref{S.C*} for  applications).  
A standard Borel space of $L$-structures $\hM(L)$ was defined in  \S\ref{S.Borel.Models}. 

\begin{prop} \label{P.UDT-Borel}
Assume $\bfT$ is a theory in a countable language $L$. Then every class 
$\fM$  of separable models of $\bfT$
 \udt{} forms a Borel subset of the space of models of $\bfT$. 
\end{prop}

\begin{proof} This is a special case of the continuous variant 
of the standard fact that the set of models of an $L_{\omega_1\omega}$ 
sentence is Borel. 
 Fix a set of sequences of 
uniform types $\langle\bt^m_n:n\in \omega\rangle$, 
for $m\in \omega$, as in \S\ref{S.udt}. 
By Lemma~\ref{L.uniform.1}, model 
$M$ 
omits a uniform sequence of types $\{\bt^m_n: n\in\omega\}$
 if and only if each of the elements of its countable dense
set omits this sequence. Thus for $\gamma\in \hM(L)$ we have 
 $M(\gamma)\in \fM$ if and only if (let 
 $l(m,n)$ denote the arity of $\bt^m_n$,  
 let $[\omega]^{l(m,n)}$  denote the set of subsets of~$\omega$ of 
 cardinality $l(m,n)$, and let $\gamma_F=\langle \gamma_i: i\in F\rangle$)
 \[
 (\forall m)(\forall n)(\exists k\geq 1) (\forall F\in  [\omega]^{l(m,n)})(\exists j\in \omega) 
\varphi^m_n(\gamma_F)\geq 1/k. 
\]
Since all quantifiers range over countable sets, this is a Borel condition. 
\end{proof} 

%For a countable language $L$  
%a compact metrizable topology  on the space of $L$-theories and compact metrizable
% topology on the space of $L$-types were defined in \S\ref{S.Borel.types}. 
%The space of sequences of $L$-types is considered with the product Borel structure. 
%We have an analogue of Proposition~\ref{P.Borel}. 
%
%\begin{coro} For a countable language $L$ and every $n$ the following sets are Borel.  
%\begin{enumerate} 
%\item The set of uniform sequences of $n$-types. 
%\item The set of all pairs $(\bfT,(\bt_j: j\in \omega))$ such that $\bfT$ is a complete theory 
%and $(\bt_n: n\in \omega)$ is a uniform sequence of types omissible in a model of $\bfT$. 
%\item The set of all pairs such that $\bfT$ is a theory and $(\bt_j: j\in \omega)$ is a uniform 
%type realized in some model of $\bfT$. 
%\end{enumerate}
%\end{coro} 
%
%\begin{proof} (1) is clear from the definition. 
%
%(2) and (3) follow from Theorem~\ref{T.Uniform} and the proof of Proposition~\ref{P.Borel}. 
%\end{proof} 
%
%

\section{Simultaneous omission of types}
\label{S.T1}

We prove Theorem~\ref{T1} by constructing 
 an example of a  complete theory~$\bfT_4$ in a countable language $L_*$ 
and  types $\bs_n$, for $n\in \omega$,
such that for every~$k$ there exists a model of $\bfT_4$ that omits
all $\bs_n$ for $n\leq k$ but 
no model of $\bfT_4$ simultaneously omits all~$\bs_n$.

 We shall prove that in addition all types $\bs_n$, for $n>0$, are simultaneously omissible in a single model $M_4$ of~$\bfT_4$. As a matter of fact, we shall first define $M_4$ and then 
 take $\bfT_4$ to be its theory.   
Type  $\bs_0$ has a distinguished role in this construction 
and it  is not to be confused with the type $\bs_0$ as defined in \S\ref{S.q}.

Let $\rank(T)$ denote the rank of a well-founded tree $T$ and let $\rho_T$ denote the 
rank function on $T$. 
  We write $\rank(T)=\infty$ if $T$ is ill-founded. 
The disjoint union of  a sequence $T_i$ for $i\in \omega$ of trees is denoted  
by~$\bigoplus_i T_i$, 
thus $\rank(\bigoplus_i T_i)=\sup_i \rank(T_i)$. 
We identify all the roots of the different  $T_i$'s with a single node which is the 
root of $\bigoplus_i T_i$.

Every tree of height $\leq\omega$ is isomorphic to a tree  of finite sequences from a
large enough set, and all trees used in \S \ref{S.T1} will be of this form. 
This convention has a convenient consequence: Every tree $T$ has 
the empty sequence $\langle\rangle$ as its root and  
 $\rank(T)=\rho_T(\langle\rangle)$. 
Also, the tree $\bigoplus_i T_i$ has a unique root $\langle\rangle$ for every sequence of trees  $T_i$, for $i\in \omega$.

\subsection{The roadmap for the proof of Theorem~\ref{T1}}
In  \S\ref{S.NodesEtc} we verify the definability of some distinguished subsets of a tree, 
such as levels or nodes of a fixed finite rank. 
In \S\ref{S.T1T2} we define the tree  $T_2$ and the concatenation operation 
$S^\frown T$. The language $L_*$ and the  $L_*$-structure $M_4$ are constructed in  
  \S\ref{S.M4}. The theory $\bfT$ as in the statement of Theorem~\ref{T1} 
  is $\bfT_4=\Th(M_4)$. In \S\ref{S.Nl} we define a family of reducts $L_l$, for $l\in \omega$, 
  of $L_*$ and a truncation operation that produces an $L_l$-structure $N\rs l$ 
  from an $L_*$-structure  $N$. These truncated structures 
  will be   used in conjunction with the test for elementary 
  equivalence, 
  Lemma~\ref{L.isomorphic}. 
Types $\bs_m$, for $m\in \omega$, are defined in \S\ref{S.types-sm} and 
 the proof of Theorem~\ref{T1} is completed in  \S\ref{S.M[l]}, after introducing yet another auxiliary family of $L_*$-structures, $M_4(j)$, for $j\in \omega$; 
 $M_4(j)$ is a model of $\bfT_4$ that omits $\bs_n$ for all $n\leq j$.

\subsection{Trees $T_2$,  $S^\frown T$, and $T_{2,m}$}\label{S.T1T2}
The tree $T_2$ is  defined as
(a function $s\colon n\to \omega\times \omega$ is identified with the pair of functions
$s_0\colon n\to \omega$ and $s_1\colon n\to\omega$):  
\[
T_2=\{s\colon n\to \omega\times \omega: n\in \omega, s_0(i)>s_0(i+1)\text{ for }0\leq i<n-1\}. 
\]
This is a well-founded tree of rank $\omega$ 
and $\rho_{T_2}(s)=s_0(|s|-1)$ for all $s$. 
The   straightforward proof of the following lemma  is omitted.     

\begin{lemma} \label{L.T2.characterization} 
The tree $T_{2}$ is a well-founded tree of 
 height $\omega$ and rank $\omega$.    
 Each of its nodes $s$ has infinitely many immediate successors of 
 every rank smaller than the rank of $s$. 
Every countable tree with these three properties is isomorphic to $T_{2}$. \qed
\end{lemma} 

  \begin{lemma}
  %\popcounter
For every node $t$ in $T_2$ and every immediate successor $s$ of $t$ there are infinitely 
many immediate successors $s'$ of $t$ such that there is an automorphism of~$T_2$ swapping $s$ and $s'$. 
\end{lemma}

\begin{proof}  Suppose $s$ is an immediate successor of $t$. 
 Let $s'$ be an immediate successor of $t$  such that $s'_0=s_0$ and  $s'_1\neq s_1$. 
 Let $\Phi$ be an automorphism of~$T_2$ that swaps the cones $\{r\in T_2: s\sqsubseteq r\}$
 and $\{r\in T_2: s'\sqsubseteq r\}$ and leaves all other nodes in $T_2$ fixed. 
 Since there are infinitely many choices for $s'_1(|s|)$, for every $s$ 
 there are infinitely many such pairs $s'$, $\Phi$. 
\end{proof}

The concatenation of finite sequences $s$ and $t$ is denoted 
$s^\frown t$. 
Given trees $S$ and $T$ whose nodes are  finite sequences, 
by   $S^\frown T$ we denote the tree obtained by tagging a
  copy of $T$ to every node of $S$. We define a new tree
\[
\textstyle S^\frown T=\{s^\frown t: s\in S, t\in T\}
\]
with the extension ordering.  
The  subtree $\{s^\frown \langle\rangle: s\in S\}$ of $S^\frown T$
 is  identified with          $S$.
 
 We record two straightforward   
facts about this construction. 
\begin{enumerate}
\item\label{I.T2.0}  $\rho_{S^\frown T}(s)=\rank(T)+\rho_S(s)$ for every $s\in S$. 
\item\label{I.T2.00}  $\rank(S^\frown T)=\rank(T)+\rank(S) $. 
\pushcounter
\end{enumerate}
 
If $s$ is a node in a tree $T$, let  
 \[
 T[s]=\{t\in T: s\sqsubseteq t\},
 \]
  with the ordering inherited from $T$. 
 It is a subset of $T$ that is a tree, but it is a subtree of $T$ only if $s$ is its root. 
 
Let 
 \[
 T_{2,m}=\{s\in T_2: \range(s_0)\subseteq m\}
 \]
The  straightforward proof the following lemma is omitted.     

\begin{lemma} \label{L.T2,m}
The tree $T_{2,m}$ has 
 height $m$  and  each of its nodes $s$ has infinitely many immediate successors of 
 every rank smaller than the rank of $s$. 
Every countable tree with these two properties is isomorphic to $T_{2,m}$, 
and in particular $T_2\rs m\cong T_{2,m}$. 
\qed
\end{lemma} 

Given a tree $T$ and $l<\omega$, we write
\[
T\rs l=\{t\in T: |t|<l\}. 
\]
%This is the subtree of $T$ consisting of its first $l$ levels. 

 \begin{lemma} \label{L.S-T2}
Suppose $T$ is a countable tree and  $s\in T$. 
Then
\begin{enumerate}
\item \label{I.S-T2.1} $(T^\frown T_2)\rs m\cong T_{2,m}$.  
\item \label{I.S-T2.2} 
 $((T^\frown T_2)\rs m)[s]\cong  T_{2,m-k}$, with $k=|s|<m$. 
\end{enumerate}
\end{lemma} 

\begin{proof}  
\eqref{I.S-T2.1} 
Since a copy 
of $T_2$ is attached to every node of the original tree, 
$(T^\frown T_2)\rs m$ satisfies the conditions of Lemma~\ref{L.T2,m} 
  and is therefore isomorphic to $T_{2,m}$. 

\eqref{I.S-T2.2} 
The isomorphism constructed in  \eqref{I.S-T2.1}
sends $((T^\frown T_2)\rs m)[s]$ to $T_{2,m}[t] $ for some $t\in T_{2,m}$ of the same rank 
as $s$. This rank is equal to $m-k$. Hence $t$ has height $k$ and rank $m-k$, and the definition of $T_{2,m}$ implies 
$T_{2,m}[t]\cong T_{2,m-k}$. 
\end{proof}

 The richness of $T_{2,m}$ is exploited in the following lemma,  one of 
 the key tools in the proof of Theorem~\ref{T1}. 
 
 \begin{lemma} \label{L.iso-extension} 
Suppose $S$ and $T$ are trees of height $\omega$, $m<l$, and 
 $S'$ and $T'$ are subtrees of $S$ and $T$, respectively, of height $m$. 
Then any isomorphism    $\Phi\colon S'\to T'$ 
can be extended to an isomorphism $\Psi\colon (S^\frown T_2)\rs l\to  (T^\frown T_2)\rs l$. 
\end{lemma} 
 
 \begin{proof} Fix enumerations of 
 $((S^\frown T_2)\rs l)\setminus S'$   
and $((T^\frown T_2)\rs l)\setminus T'$ in type $\omega$ which extend the corresponding tree orderings. 
An isomorphism $\Psi$ extending $\Phi$ is now constructed by a back-and-forth method along these enumerations. 
Given a node $s\in (S^\frown T_2)\rs l$ such that $\Psi(t)$ is defined for all of its predecessors, let $\Psi(s)$ be any successor of $\Psi(s^-)$ of the same rank as $s$. 
Such a node exists by Lemma~\ref{L.S-T2} and Lemma~\ref{L.T2,m}. 
Similarly, if $t\in (T^\frown T_2)\rs l$ is such that all of its predecessors are in the range of $\Psi$ defined so 
far, let $s_0$ be such that $\Psi(s_0)$ is the immediate predecessor of $t$ and let $\Psi(s)=t$ for an immediate successor $s$ of $s_0$ of the same rank as $t$. 
 This is possible by Lemma~\ref{L.T2.characterization}. After $\omega$ steps one has an isomorphism $\Psi$. 
\end{proof} 

The following proposition will not be used explicitly in any of our other proofs.  
Its proof 
is a  toy version of a part of the proof of Theorem~\ref{T1}. 

\begin{prop} Trees  $S^\frown T_2$ and $T_2$ are (considered as $L_{\cN}^-$-structures)  elementarily equivalent for any tree $S$. 
\end{prop}  
 
 \begin{proof} Since $T_2$ is countable, an easy L\"owenheim--Skolem 
  argument shows that 
it suffices to prove the assertion for a countable tree $S$.  
Lemma~\ref{L.iso-extension} implies that $(S^\frown T_2)\rs m$ 
is isomorphic to $T_2\rs m$ for all $m$. Since $T_2\rs m\cong T_{2,m}$ 
is $1/(m+1)$-dense in $T$ for every $T$,   Lemma~\ref{L.isomorphic} implies  $T_2\equiv S^\frown T_2$. 
\end{proof} 
 
\subsection{The language $L_*$, models $M(k)$ and  $M_4$, and the theory $\bfT_4$} \label{S.M4} 
The most important definition in this subsection is that of $M(k)$, an  expansion of $T_2$ to language $L_*$.  
Language  $L_*$ is an expansion of
  $L_{\cN}^-$  (\S\ref{L.CN-})
with the additional  unary predicate symbols~$P_{i,j}$ for $i>0$ and $j\geq 0$
 in $\omega$ such that each $P_{i,j}$ is $i+1$-Lipshitz. 

With  $T_2$ as defined in \S\ref{S.T1T2} fix  an injection 
\[
k\colon T_2\to \omega
\]
such that $s\sqsubset t$ implies $k(s)<k(t)$ for all $s$ and $t$ in $T_2$. 
An~$L_*$-structure $M(k)$
whose universe is~$T_2$ is defined as follows. 
The function symbols~$f_k$ for $k\in \omega$ are interpreted as in  \S\ref{S.Baire}, by  
$f_k(a)=b$ if $b$ is the unique element of the $k$-th level 
below $a$ if there is such $b$ or $f_k(a)=a$ otherwise. 

In order to interpret the predicates $P_{i,j}$, we first 
 define a colouring $c\colon T_2\to \omega$.  
Let $c(\langle\rangle)=0$. 
Suppose that $|t|\geq 1$, 
 and denote  the immediate predecessor of $t\in T_2$ 
by $t^-$, and   let 
  \[
  c(t)=\max\{k(t^-),\max\{l: (\exists j<|t|) 2^l| t_1(j)\}
  \}. 
\]
We make a note  that $c(t)\geq k(t^-)$ and that $s\sqsubseteq t$ implies $c(s)\leq c(t)$. 

Having defined the colouring $c$, for all $i,j$ let the interpretation of $P_{i,j}$ in $M(k)$ be defined as follows: 
\[
P_{i,j}(t)=
\begin{cases} 
0, \text{ if  } |t|=i \text{ and } c(t)=j,\\ 
1, \text{ otherwise.}
\end{cases} 
\]
Since $d(s,t)<1/(i+1)$ and $s\neq t$ together imply  $\min(|s|,|t|)>i$ and $P_{i,j}(s)=P_{i,j}(t)=1$, 
the predicate  $P_{i,j}$ as defined above  is $i+1$-Lipshitz for all $i\geq 1$ and all $j$. 

\begin{lemma} \label{L.P} For every $t\in T_2$ the following hold
(the condition $P_{i,j}(t)=0$ is interpreted as `the node $t$ is $(i,j)$-coloured'):  
\begin{enumerate}
\item [(P1)] The node $t$ is not $(i,j)$-coloured for any  $i\neq |t|$ and $j\geq 0$. 
\item [(P2)] The node $t$ is $(|t|,j)$-coloured for exactly one $j$.  
\item  [(P3)] If  $t$ is an $(i,j)$-coloured, non-terminal, node of $T_2$
then for every $m<\rho_{T_2}(t)$ and every $j'\geq \max(j,k(t))$
 there are infinitely many immediate $(|t|+1,j')$-coloured successors $s$ of $t$ such that  $\rho_{T_2}(s)=m$. 
%\item  [(P4)] No immediate successor of $t$ is $(i,j)$-coloured for any  $j< k(t)$.   
\item [(P4)] Every immediate successor of an 
  $(i,j)$-coloured node  
is   $(i+1,j')$-coloured for a unique $j'\geq \max(j,k(t))$.   
\end{enumerate}
\end{lemma} 

\begin{proof} Conditions (P1), (P2), and (P4) are clearly satisfied.  Condition (P3) is satisfied because  
  for all $m\in \omega$ and $n\in \omega$, for  every  $s\in T_2$ with $\rho_{T_2}(s)>m$ 
there exists an immediate  successor $t$ of $s$ with $\rho_{T_2}(t)=m$ and $t_1(i-1)=n$. 
In particular for every $l\in \omega$ there are infinitely many immediate 
successors~$t$ of $s$ with $\rho_{T_2}(t)=m$ such that $2^l$ divides $t_1(i-1)$ but $2^{l+1}$ does 
not. 
\end{proof} 

Fix a  family $k^n$, for $n\in \omega$, of functions 
such that 
\begin{enumerate}
\pushcounter
\item \label{I.k0} For all $n$ the function $k^n\colon T_2\to \omega$ is an injection. 
\item \label{I.k1} For all $n$ and $s\sqsubset t$ in $T_2$ we have $k^n(s)<k^n(t)$ ($k^n$ is increasing). 
\item \label{I.k2}  Every increasing injection $k$ from a finite  subtree of $T_2$ into $\omega$ 
is extended by $k^n$ for infinitely many $n$. 
\popcounter
\end{enumerate}
With $M(k^n)$ denoting $M(k)$ for $k=k^n$, let  
\[
M_4=\bigoplus_n M(k^n)^\frown \Too. 
\]
The \emph{bottom part} of $M_4$ is $\bigoplus_n M(k^n)$.   
All other nodes comprise the   \emph{top part} of $M_4$. 
The interpretations of $f_m$ and $P_{i,j}$ in the  bottom part of $M_4$ agree with those in $M(k^n)$. 
Suppose $t$ is in the top part 
of $M_4$. If $i>0,j\geq 0$, and  $t$ belongs to the top part of $M_4$ then 
we let $P_{i,j}(t)=1$; 
i.e. $t$ is not  $(i,j)$-coloured for any pair $(i,j)$. 
Clearly the interpretation of 
each $P_{i,j}$ in $M_4$ is $i+1$-Lipshitz. 
The functions $f_m$ are interpreted in the usual way, by letting $f_m(t)=t$ if 
$|t|<m$ and $f_m(t)=s$ if  $s$ is the unique node on the $m$th level of $M_4$ extended by $s$. 

Since its underlying tree $(\bigoplus_n T_2)^\frown T_2$ is well-founded, $M_4$ has no nontrivial Cauchy sequences  
 and is therefore complete. 
Let 
\[
\bfT_4=\Th(M_4). 
\]
While the definition of $M_4$ depends on the choice of functions $k^n$, it can be proved that 
its theory 
$\bfT_4$ is independent from this choice, as long as these functions satisfy conditions \eqref{I.k0}--\eqref{I.k2} 
 (see Remark~\ref{R.T4}). This fact will not be used in the proof of Theorem~\ref{T1}.  
 
We record some consequences of the definitions  and  properties of rank stated in 
 \eqref{I.T2.0} and~\eqref{I.T2.00}. Since $M_4$ is well-founded, all of its elements are nodes. 
\begin{enumerate}
\popcounter
\item A node $t$ of $M_4$  is coloured if and only if it belongs to the bottom part of $M_4$.
If $t\neq \langle\rangle$ then  
its colour  $(i,j)$ satisfies  $i=|t|$ and $j\geq k^n(t^-)$ for  $n$ such that $t\in M(k^n)$.
If $t^-$ is $(i-1,j')$-coloured, then $j'\leq j$.  
\item\label{I.M4.full}  A node $t$ of $M_4$ satisfies $\rho_{M_4}(t)<\omega$ 
if  $t$ belongs to the top part of~$M_4$  and $\omega\leq\rho_{M_4}(t)<2\omega$ 
 if $t$  belongs to the bottom part of~$M_4$. 
%\item\label{I.M3.4} If $t$ is a node in $M_4$, $m\geq 1$,  and $\rho_{M_4}(t)=m$ then 
%$\{s: t^\frown s\in M_4\}$ is isomorphic to 
%\[
%S_4=\{s
%\]
% Every node $t$ of $M_4$ has infinitely many immediate successors of rank exactly $m$ for every $m\in \omega$
% such that $m<\rho_{M_4}(t)$. 
\pushcounter
\end{enumerate}
%Only  \eqref{I.M3.4} requires a proof. 
%If $t$ belongs to the top part of $M_4$ then $\rho_{M_4}(t)<\omega$ and
%\eqref{I.M3.4} is a consequence of the definitions of $T_2$ and rank. 
%If $t$ belongs to the bottom part of $M_4$ then \eqref{I.M3.4} 
%follows from the fact that the underlying tree is equal to  $(\bigoplus_n T_2)^\frown T_2$ 
%since the first level of $T_2$ has nodes of every rank~$<\omega$. 
 Suppose $N\models \bfT_4$. As in   \S\ref{S.WF}, 
\[
T_N=\{a\in N: (\exists j\in \omega) f_j(a)=a\}
\]
is a tree dense in $N$. Every $b\in N\setminus T_N$ can be identified with the branch $\{f_j(b): j\in \omega\}$, 
and this branch is a Cauchy sequence converging to $b$. 
If $s\in N$ then we write $\rho_N(s)$ for  $\rho_{T_N}(s)$ if $s\in T_N$ and $\rho_N(s)=\infty$ 
otherwise. (Note that $\rho_N(s)=\infty$ is a possibility even if $s\in T_N$.)

If $N\models \bfT_4$ then we 
use the terminology and notation as in \S\ref{L.CN-} %\S\ref{S.NodesEtc}
and refer to the elements of $T_N$ and $N\setminus T_N$  as  
 \emph{nodes} and \emph{branches}, respectively,  of~$N$. 
All  nodes $s$ and $t$ of $N$ are as elements of tree $T_N$ and we 
  write      $|s|$ and $s\sqsubseteq t$, and   
      $\rho_N(s)$ for $\rho_{T_N}(s)$.  

The  \emph{bottom part} of  $N$ is defined to be the set of all of its  coloured nodes. 
We prove a simple fact about it. 
\begin{enumerate}
\popcounter
\item\label{I.bottom.N}  If $N\models \bfT_4$, $t$ is a coloured node of $N$ and $s\sqsubseteq t$, 
then $s$ is coloured. 
In other words, the bottom part of $N$ is an initial segment of $N$. 
%\item \label{I.bottom.N.1} If $N\models \bfT_4$ and $t$ is a node in $N$ of rank at least $m\geq 1$, 
%then  the cone $\{s\in N: t\sqsubseteq s\}$ contains a terminal copy of 
% \[
% T_{2,m}=\{s\in T_2: \range(s)\subseteq m\}
% \]
%analogously to  Lemma~\ref{L.T1m}. 
\pushcounter
\end{enumerate}
If $t$ is an  $(i,j)$-coloured node in $M_4$ then 
 its  immediate predecessor
 $s$ is $(i-1,l)$-coloured for some $l\leq j$ (by (P4)). Since the first $i$ levels of~$T_N$ form a 
 definable set  (Lemma~\ref{L.Def.Nodes}), for a fixed $i$ and $j$ this can be expressed as an  
 $L_*$-sentence and therefore transfers from $M_4$ to~$N$. 
 By  induction, \eqref{I.bottom.N} follows.

%As discussed in \S\ref{S.Definable}, 
% Lemma~\ref{L.Def.Nodes} 
% implies that adding  quantifiers of the form $\sup_{z\in \Succ_k(x)}$
%  does not 
%change the expressive power of the language (see \cite{BYBHU}).  
 
We say that a subtree $S$ of $T$ is \emph{terminal} if every terminal node of $S$ is also a terminal node of $T$.

\begin{lemma}\label{L.T1m}  Suppose $N\models \bfT_4$ and   
$s\in T_N$ satisfies $\rho_{T} (s)\geq n$ for some $n\geq 1$.  
Then $T_N[s]$ contains a terminal copy of $T_{2,m}$. 
\end{lemma} 
 
 \begin{proof} Let $T=T_N$.  
 We need to prove that  there exists $\Phi\colon T_{2,n}\to T$ satisfying the following 
for all $t$ and $t'$ in $T_{2,n}$ 
 (as before,  $|t|$  denotes the level of~$t$): 
\begin{enumerate}
\popcounter
\item $s\sqsubseteq \Phi(t)$,  
\item $t\sqsubseteq t'$ if and only if $\Phi(t)\sqsubseteq \Phi(t')$, 
\item $|\Phi(t)|=|s|+|t|$, 
\item $t$ is a terminal node in $T_{2,n}$ if and only if $\Phi(t)$ is a terminal node in~$T$. 
\pushcounter
\end{enumerate}
Given such $\Phi$, the image of $T_{2,m}$ under $\Phi$ will be a terminal subtree of $T[s]$. 

We start the proof by verifying the case 
  when      $T=T_2$. The case when $s\in T_2$ is the root is clear. 
 Otherwise,  $\rho_{T_2}(s)=s_0(|s|-1)$,  hence  $\rho_{T_2}(s)\geq n$ if and only if 
  $s_0(|s|-1)\geq n$. Therefore every $t\in T_{2,n}$ satisfies  $t_0(0)<n$ and 
  \[
  \Phi(t)=(s_0^\frown t_0, s_1^\frown t_1) 
  \]
  is a function as required. 

A proof that $M_4$ has this property is analogous, since every node in it is equal to 
$r^\frown s$ for some $r$ and $s$ in $T_2$, 
and $\rho_{S^\frown T_2}(r^\frown s)=s_0(|s|-1)$, unless~$s$ is the root
of $T_2$. 

 It remains to prove  that this property  is preserved by  elementary equivalence. 
  Fix $m$ and $k$ in $\omega$. 
  Lemma~\ref{L.counting} 
implies that the assertion `every node  of height $m$  and rank at least $n+1$ 
 has at least~$l$ distinct immediate successors  of rank  $k$' is elementary.
Since in $M_4$ every node of any height and any rank has infinitely many immediate successors of every possible finite rank, this holds for every $N\models \bfT_4$. 
%%%%%%%%%%%%%%

Therefore if $N\models \bfT_4$ and $s\in T_N$ satisfies $\rho(s)=n$, then $s$ has infinitely many immediate successors of 
rank $k$ for every $k<n$. A function $\Phi$ as required  can now be  constructed by a  recursion similar to one in the proof of Lemma~\ref{L.iso-extension}. 
\end{proof}

\subsection{The language $L_l$ and models $N\rs l$} \label{S.Nl}
Fix $l<\omega$ and a separable $L_*$-structure $N$. We shall be interested in the case when 
$N\models \bfT_4$ or when~$N$ is one of the building blocks $M(k^n)$, for $n\in \omega$, 
for $M_4$. 
Let 
\[
N\rs l=\{s\in N: |s|\leq l\}. 
\] 
%(this set is  also equal to $\{f_i(a): i\leq l, a\in N\}$). 
Let  $L_l$ denote the finite reduct of $L_*$ consisting of symbols 
$P_{i,j}$ and $f_i$ for $\max(i,j)<l$. 
We shall consider $N\rs l$ as an $L_l$-structure. 
Hence it is an $L_l$-reduct of the substructure of $N$ with the universe $N\rs l$. 
A node $s\in N\rs l$ is \emph{$L_l$-coloured} if it is $(i,j)$-coloured  for some~$P_{i,j}$ in $L_l$. 
  
\begin{lemma} \label{L.Mk} 
Fix $l<\omega$ and functions $k^k$, for $n\in \omega$, satisfying \eqref{I.k0}--\eqref{I.k2}. 
\begin{enumerate}
\popcounter
\item\label{I.Mk.3}  There are only finitely many 
models 
$\{(M(k^m)^\frown T_2)\rs l: m\in \omega\}$ up to isomorphism. 

\item \label{I.Mk.4} For 
every $L_l$-coloured 
node $t\in M_4\rs l$ there exists an automorphism~$\Phi$ of $M_4\rs l$ such that 
$\Phi(t)$ has a coloured (but not necessarily $L_l$-coloured) successor. 
\end{enumerate}
\pushcounter
\end{lemma} 

\begin{proof} \eqref{I.Mk.3} 
Since $k^m$ is an increasing injection, 
\[
F_m=\{t\in T_2: k^m(t)<l\}
\]
 is a finite subtree of $T_2$ of cardinality $l$. 
Consider it as  a tree  whose nodes~$t$ are 
labeled by pairs $(k^m(t),c(t))$. There are (at most)  $l^2$ such pairs,  
and  
up to the isomorphism, there are only finitely many trees of cardinality $l$ with $l^2$-coloured nodes. 
Let 
\[
S_m=\{t\in M(k^m): \text{$t$ is $L_l$-coloured}\}.
\]
It is a subtree of $M(k^m)$ since $s\sqsubseteq t$ implies $c(s)\leq c(t)$. 
 Clearly 
$F_m$ is a subtree of $S_m$.   Consider $S_m$ as a labelled tree, where each node $t$ is labelled by $c(t)$. 
The isomorphism type of this infinite labelled tree  
is, by  (P3) and (P4),  uniquely determined by the isomorphism type of $F_m$. 
We therefore have only finitely many isomorphism classes of labelled trees $S_m$, for $m\in \omega$. 
 
Fix $m$ and $n$ such that the labelled 
trees $S_m$ and $S_n$ are isomorphic. 
 Lemma~\ref{L.iso-extension} implies that this isomorphism can be extended to an isomorphism between 
trees $((S_m)^\frown T_2)\rs l$ and $((S_n)^\frown T_2)\rs l$. This 
gives the required isomorphism between   
$(M(k^m)^\frown T_2)\rs l$
and 
$(M(k^n)^\frown T_2)\rs l$. 
Therefore the isomorphism type of $(M(k^n)^\frown T_2)\rs l$ depends only on $l$ and the isomorphism type of the labelled tree $F_n$. This completes the proof of \eqref{I.Mk.3}. 

%\popcounter

\eqref{I.Mk.4} 
Let $m$ be such that $t\in M(k^m)\rs l$.  By \eqref{I.k2}  the functions~$k^n$, for $n\in \omega$, are chosen so 
that  every increasing injection from a finite subtree of~$T_2$ into~$\omega$ is extended by some $k^n$. We can therefore find $n$ such that   
 (using the notation from the proof of \eqref{I.Mk.3}) 
 \begin{enumerate}
 \item [(a)] 
 There exists an isomorphism $\Psi\colon F_m\to F_n$ such that $k^m(s)=k^n(\Psi(s))$
 and $c(s)=c(\Psi(s))$ for all $s\in F_m$. 
 \item [(b)] For every $s\in F_m\setminus \{\langle\rangle\}$,   
  the node $r=\Psi(s)$ satisfies $r_0(j)=s_0(j)+1$
 for all  $j<|s|$.  
\end{enumerate}
  As in \eqref{I.Mk.3}, $\Psi$  extends to an isomorphism 
  \[
  \Phi_1\colon  
(M(k^m)^\frown T_2)\rs l\to (M(k^n)^\frown T_2)\rs l.
\] 
By (b), $\rho(\Phi_1(t))=\rho(t)+1$ and therefore $\Phi_1(t)$ is not a terminal node of the bottom part of $M(k^n)^\frown T_2$. 
 The automorphism $\Phi$ as required 
is defined by $\Phi(s)=\Phi_1(s)$, if $s\in M(k^m)^\frown T_2$, $\Phi(s)=\Phi^{-1}_1(s)$, if $s\in M(k^n)^\frown T_2$, 
and $\Phi(s)=s$ otherwise. 
\end{proof} 

While the proof of Lemma~\ref{L.Mk} is fresh, we define an another family of truncations of $M(k^n)$ used in  \S\ref{S.M[l]}. 

\begin{definition}\label{Def.Mk} 
Given functions $k^n$, for $n\in \omega$, satisfying \eqref{I.k0}--\eqref{I.k2},  $j\geq 1$,  and  $n\geq 1$. 
define a substructure of $M(k^n)$ by 
\[
M(k^n, j)=\{s\in M(k^n): s_0( |s|-1)\geq j-|s|\}. 
\]
\end{definition} 
This structure is  obtained by removing every node $s\in M(k^n)$ with  no extension to the $j$th level.
Therefore~$|s|<j$ implies $\rho_{M(k^n,j)}(s)\geq j-|s|$ for $s\in M(k^n,j)$. 
Since the universe of $M(k^n)$ is~$T_2$,  the  elements of $M(k^n,j)$  are pairs $s=(s_0,s_1)$ 
and  $\rho_{M(k^n)}(s)=s(|s|-1)$.

\begin{lemma} 
  \label{I.Mk.5} 
  Fix $l\in \omega$ and functions $k^n$, for $n\in \omega$, satisfying \eqref{I.k0}--\eqref{I.k2}. 
For every $m$ there are infinitely many $n$ such that 
\[
(M(k^m,j)^\frown T_2)\rs l\cong (M(k^n)^\frown T_2)\rs l
\]
and infinitely many $n'$ such that 
\[
(M(k^m)^\frown T_2)\rs l\cong (M(k^{n'},j)^\frown T_2)\rs l.
\]  
\end{lemma}

\begin{proof} Fix $m$. Following the proof of Lemma~\ref{L.Mk}, 
let (suppressing the subscript on $\rho$ when it is clear from the context)
\[
F_m(j)=\{t\in T_2: k^m(t)<l\text{ and }\rho(t)\geq j-|t| \}
\]
and consider it as a tree whose nodes $t$ are coloured by $(k^m(t), c(t))$. 
As in the proof of Lemma~\ref{L.Mk}, the isomorphism type of $(M(k^m, j)^\frown T_2)\rs l$
is uniquely determined by $F_m(j)$. 
 By the choice of the functions $k^n$, for $n\in\omega$ we can find infinitely many $n$ such that the trees $F_m(j)$ and 
 $F_n=\{t\in T_2: k^n(t)<l\}$ are isomorphic as trees with $l^2$-coloured nodes. 
The tree $S_n$ as defined in Lemma~\ref{L.Mk} is isomorphic  to 
\[
S_m(j)=\{t\in M(k^m,j): \text{$t$ is $L_l$-coloured}\}. 
\]
As in Lemma~\ref{L.Mk}, 
this isomorphism extends to an isomorphism between 
$(M(k^m,j)^\frown T_2)\rs l$ and  $(M(k^n)^\frown T_2)\rs l$. 

This proves the first part of the lemma, and the proof of the second part is analogous. 
\end{proof}

\begin{remark} \label{R.T4} The definition of  $M_4$ depends on the choice of  $k^n$, for $n\in \omega$. However, the proof of Lemma~\ref{L.Mk}
can be modified to show 
that  all instances  of  $M_4$ constructed by using
sequences of functions satisfying  \eqref{I.k0}--\eqref{I.k2} 
are elementarily equivalent. In particular $\bfT_4$ does not depend on this choice. 
This fact will not be used in the proof of Theorem~\ref{T1}. 
\end{remark} 

\subsection{Types $\bs_m$, for $m\in \omega$} \label{S.types-sm}
Let $\bs_0(x)$ be the 1-type of an infinite branch,  
\[
\bs_0(x)=\{d(x,f_n(x))=1/(n+1): n\in \omega\}. 
\] 
In order to define $\bs_m(x)$ for $m>0$, 
fix a bijection $m\mapsto (m(0), m(1))$  
between $\{m\in \omega: m>0\}$ and $\{(m(0),m(1))\in \omega^2: m(0)\leq m(1)\}$.  
   As in  \S\ref{S.NodesEtc},  if $a$ is a node in a tree $T$ and $k>m=|a|$ then  
\[
\Succ_k(a)=\{b\in T: f_m(b)=a\text{ and } |b|=k\}
\]
and 
\[
\varphi_m(x)=d(f_m(x), x)+\left|\frac 1{m} -d(f_{m-1}(x),x)\right|.  
\] 
 The type $\bs_m(x)$ consists of  conditions described in (S1)--(S3) below.
  \begin{enumerate}
\item [(S1)] $\varphi_{m(0)}(x)=0$. 
\item [(S2)] $P_{m(0),m(1)}(x)=0$. 
\item [(S3)] \label{I.sm.3} $\sup_{y\in \Succ_{m(0)+1}(x)} (1\dminus P_{m(0)+1, n}(y))=0$
for all  $n\in \omega$. 
\end{enumerate}
The condition in  (S3) is an abbreviation of an  $L_*$-formula because 
  allowing quantification over definable sets results in a conservative extension of the 
language (see~\S\ref{S.Definable}). 

\begin{lemma} \label{L.S1-3}
If   $N\models \bfT_4$ and $a\in N$ then $a$ realizes $\bs_m$  for some $m>0$ 
if and only if $a$ is a terminal node of the bottom part of $N$. 
\end{lemma}  

\begin{proof} Suppose $a$ realizes $\bs_m$ in $N$ and $m>0$.  
Then  (S1) implies $|a|=m(0)$,  (S2) implies $a$ is coloured, 
and (S3) implies that  all immediate successors of $a$ are uncoloured. 
Therefore  $a$ is  a terminal 
node of the bottom part of $N$. Conversely, suppose $a$ is 
a terminal node of the bottom part of $N$. If $m$ is such that $a$ belongs to the $m(0)$th level 
of $N$ and it is $(m(0),m(1))$-coloured, 
then $a$ realizes $\bs_m$. 
\end{proof}

\subsection{Models $M_4(j)$, for $j\in \omega$.} \label{S.M[l]} 
We are almost ready to  complete the proof of Theorem~\ref{T1}. 
For each  $j\geq 1$ we will  construct  $M_4(j)\models\bfT$
  that omits~$\bs_m$ (\S\ref{S.types-sm}) for all $m$  such that $m(0)<j$ and prove that no model of $\bfT_4$ omits~$\bs_m$ for all large enough $m$.  
(Lemma~\ref{L.almostT1}).   

Given functions $k^n$, for $n\in \omega$, satisfying \eqref{I.k0}--\eqref{I.k2},  $j\geq 1$,  and  $n\geq 1$, as in Definition~\ref{Def.Mk}  
let 
\[
M(k^n, j)=\{s\in M(k^n): s_0( |s|-1)\geq j-|s|\}. 
\]
 The $L_*$-structure
\[
M_4(j)=\bigoplus_n M(k^n,j)^\frown \Too 
\]
clearly omits $\bs_0$ and  
$\bs_m$ for all $m$ such that $m(0)<j$. 

\begin{lemma} \label{L.Isomorphism} Fix $l$ and $j$ in $\omega$. 
\begin{enumerate}
\popcounter
\item\label{I.Mk.3+}  There are only finitely many 
models 
$\{(M(k^m,j)^\frown T_2)\rs l: m\in \omega\}$ up to the isomorphism. 
\item \label{I.Mk.4+} For 
every $L_l$-coloured 
node $t\in M_4(j)\rs l$ there exists an automorphism~$\Phi$ of $M_4(j)\rs l$ such that 
$\Phi(t)$ is not a terminal node of the bottom part of $M_4(j)$.
\item \label{I.Isomorphism} For all $j$ and $l$ the 
 models  $M_4(j)\rs l$ and $ M_4\rs l$ 
are isomorphic. 
\item \label{I.iso.M4(j)} For all $j\in\omega$ the
 models $M_4(j)$ and $M_4$ are elementarily equivalent. 
\end{enumerate}
\pushcounter
\end{lemma} 

\begin{proof} 
Lemma~\ref{L.iso-extension} implies 
\[
(M(k^m,j)^\frown T_2)\rs l\cong (M(k^m)^\frown T_2)\rs l.
\]
 Therefore \eqref{I.Mk.3+} and \eqref{I.Mk.4+}   
are direct consequences of  the corresponding assertions in  Lemma~\ref{L.Mk}. 

\eqref{I.Isomorphism} 
By Lemma~\ref{L.Mk} \eqref{I.Mk.5} and the back-and-forth method we can 
find an enumeration $m(i)$, for $i\in \bbZ$,
of $\omega$ and isomorphisms 
\[
\Phi_i\colon (M(k^{m(i)}, j)^\frown T_2)\rs l \to 
 (M(k^{m(i+1)})^\frown T_2)\rs l. 
\]
 The  union of the graphs of $\Phi_i$ is the graph of the required isomorphism~$\Phi$. 

\eqref{I.iso.M4(j)} Fix $j<\omega$.  By
\eqref{I.Isomorphism}, $M_4(j)$ and $M_4$ have $\e$-dense isomorphic substructure for every $\e>0$. Lemma~\ref{L.isomorphic} therefore implies $M_4(j)\equiv M_4$.  
\end{proof}

\begin{lemma} \label{L.almostT1}      No model of $\bfT_4$ omits all types~$\bs_n$, for $n\in \omega$,  simultaneously. 
\end{lemma} 

\begin{proof} 
Assume $N\models\bfT_4$ and it 
   omits~$\bs_m$ for all $m$.  
   It suffices to know that there exists $m'$ such that 
   $N$ omits $\bs_m$ for all $m$ 
   such that  $m(0)\geq m'$. For every $m(0)\geq m'$,  
 the `bottom part' of $N$ (consisting of all coloured nodes) intersects the $m(0)$th level by the elementarity, and  the tree $T_N$  
is  ill-founded.  Therefore~$N$ has an infinite branch and it realizes~$\bs_0$. 
\end{proof} 

\begin{proof}[Proof of Theorem~\ref{T1}]
We shall prove that $\bfT_4$  
and  types $\bs_n$, for $n\in \omega$ as defined above 
are as required. 
Being theory of a model $M_4$, theory $\bfT_4$ is clearly complete.  
For every $j$ model $M_4(j)$ is a model of  $\bfT_4$
by Lemma~\ref{L.Isomorphism} \eqref{I.Isomorphism} and Lemma~\ref{L.isomorphic}. 
By its definition and the discussion following (S1)--(S3), model $M_4(j)$  omits
all $\bs_m$ for $m(0)<j$. Therefore for every $k$ there exists $j$ large 
enough such that model $M_4(j)$ of $\bfT_4$ omits all $\bs_m$ for $m<k$.  

By Lemma~\ref{L.almostT1} no model of $\bfT_4$ simultaneously omits all $\bs_m$ for $m\in \omega$, and 
this concludes the proof. 
\end{proof}

\subsection{The Extension Lemma}  \label{S.Extension}
 %%%%%%%%
Towards the proof of Theorem~\ref{T2}, we prove that the type 
 $\bs_m(x)$ (\S\ref{S.types-sm}) is generically omissible for all $m>0$. 
In the following $\bbPT$ denotes the `infinite' forcing as defined in \S\ref{S.forcing} with $\bfT=\bfT_4$, 
which is $\bbPTS$ in case when $\Sigma$ is the set of all $L$-formulas 
 and $\bar x,y,z$ stand for assorted $d_j$'s.

\begin{lemma}[The Extension Lemma] \label{L.types.2} Assume $\varphi(\bar x,y)<\e$ is a condition in $\bbPT$ which   
forces 
that   $P_{m,n}(y)=0$ for some $m$ and $n$. 
% \item  [(b)] $z$ is an immediate successor of $y$. 
Then for some new variable $z$ 
this condition can be extended to a condition 
$\varphi'(\bar x,y,z)<\e'$ that in addition forces 
  $P_{m+1,k}(z)=0$ for some~$k$ and that $z$ is an immediate successor of $y$. 
\end{lemma} 

\begin{proof} 
Since $\varphi(\bar x,y)<\e$
 is a consistent condition, we can find a tuple $\bar a,b$ in $M_4$ that realizes it. 
Suppose $b$ is not a terminal node of the bottom part of $M_4$. 
 Fix an immediate successor $c$ of $b$ such that $P_{i,j}(c)=0$ for some $i$ and $j$
and extend condition $\varphi(\bar x,y)<\e$ to a condition $\varphi'(\bar x,y,z)<\e'$ that 
decides  $z$ is an immediate successor of $y$ and $P_{i,j}(z)=0$. This condition is as required. 

Now suppose 
 $b$ is a terminal node of the bottom part of $M_4$. 
We may assume that $\varphi$ is in the prenex normal form
since such formulas are uniformly dense in 
the space of all formulas 
by \cite[Proposition~6.9]{BYBHU}. 
By Lemma~\ref{L.CFC} we may also assume that $\e<1$. 
Let $\psi_j(\bar x,y, \bar t)$, for $j<n$, be a list of all atomic subformulas 
of $\varphi(\bar x,y)$. Thus $\varphi(\bar x,y)$ is of the form (variables in $\psi_j$ are 
suppressed and  dummy variables $t_0$ and $t_{l-1}$  are
 added, if necessary, so that the string of quantifiers starts with $\sup$ and ends with $\inf$; this is done 
 only for readability) 
\begin{equation}
\tag{$\dagger$}
\sup_{t(0)} \inf_{t(1)}  \dots \inf_{t(l-1)} 
f(\psi_0, \dots, \psi_{n-1})
\end{equation}
for some $l\in \omega$, variables $t(j)$ for $j<l$, and  continuous function $f$. 
Let 
\[
\e'=(\e-\varphi(\bar a,b)^{M_4})/3. 
 \]
Since the interpretation of $f(\psi_0,\dots, \psi_{m-1})$ is uniformly continuous and its modulus of continuity does not depend on 
the interpretation, we can find   
 $\delta>0$  such that changing the values of all variables occurring in 
any  $\psi_j$ by $< \delta$ affects the 
change of the value of  
$f(\psi_0, \dots, \psi_{n-1})$
 by  $<\e'$. 
Let $d>1/\delta$ be such that all pairs $i,j$ for which predicate $P_{i,j}$  
occurs in some $\psi_j(\bar x,y)$ satisfy $\max(i,j)<d$. By increasing $d$ if needed 
we may also assume 
that all projection functions $f_i$ occurring in some $\psi_j(\bar x,y)$ satisfy $i<d$ and 
 that $\bar a,b$ belong to one of the first $d$ levels of $M_4$. 

Consider $L_d$ and   $M_4\rs d$ as defined in \S\ref{S.Nl}. 
   
   \begin{claim} \label{Claim.*} For any tuple $\bar p, d$ in $M_4\rs d$ 
  of the same  length as $\bar x, y$  we have 
  \begin{equation*} 
  |\varphi(\bar p,d)^{M_4\rs d}-\varphi(\bar p,d)^{M_4}|< \e'.
  \end{equation*}
  \end{claim} 

\begin{proof} 
For every tuple $\bar q$ in $M_4$ such that $\bar p,d,\bar q$ is of the same 
length as $\bar x,y,\bar t$ (where $\bar t$ are the 
variables occurring freely in formulas $\psi_j$ but not in $\varphi$), 
we have $d(q_i, f_d(q_i))<\delta$ for all $i$. 
For every $q\in M_4$ there exists $f_d(q)\in M_4 \rs d$
within $<\delta$ of $q$.  
Since $\varphi$ is as in  ($\dagger$),  
 the claim follows by the choice of~$\delta$. 
\end{proof} 

By \eqref{I.Mk.4} of  Lemma~\ref{L.Mk} there exists an automorphism $\Phi$ of $M_4\rs d$
which sends $b$ to a  node which is not a terminal node of the bottom part of $M_4$. 
This is certainly not an automorphism of $M_4$ but the  
 $\Phi$-image of  $\bar a,b$ still satisfies  $\varphi(\bar x,y)<\e$. 
We can fix an immediate coloured successor $c$ of $\Phi(b)$ 
and extend  $\varphi(\bar x, y)<\e$ to a condition 
 $\psi(\bar x, y,z)<\e'$ which    implies $z$ is a coloured 
 immediate successor of $c$. 
This  completes the treatment of  the case when $b$ is a terminal node 
of the bottom part of $M_4$ and  the proof. 
 \end{proof}

Generic model $M_G$ was defined in Theorem~\ref{P3}. 

\begin{lemma}\label{L.M.generic}
There is a countable family $\bfF$ of  dense  subsets of $\bbPTS$
such that if a filter $G$ is $\bfF$-generic  then $M_G$ is a model of $\bfT_4$
and   every  coloured node of $M_G$ has a coloured immediate successor. 
\end{lemma} 

\begin{proof} By Theorem~\ref{P3} there 
is a countable family~$\bbFTS$ of dense subsets of~$\bbPTS$ such that if a filter  $G$ is $\bbFTS$-generic then 
$M_G\models \bfT_4$.  
  Lemma~\ref{L.types.2}  implies that the set
  \begin{align*} 
\bfD_{i,j,k}=\{&p\in \bbFTS: \bfT_4+p\models \text{$d_k$ is not $(i,j)$-coloured or }\\
& \bfT_4+p\models 
\text{`$d_j$ has an  $(i+1,l)$-coloured successor for some $l$'}\}
\end{align*} 
is dense in $\bbFTS$ for all $i,j$, and $k$. 

Let $\bfF'=\bbFTS\cup \{\bfD_{i,j,k}: i,j,k\in \omega\}$ and suppose 
$G$ is $\bfF'$-generic. 
 Then  
$M_G\models \bfT_4$, 
 hence  node of $M_G$ is an  isolated point 
  and is therefore an interpretation of some constant $d_j$.  
By the  elementarity 
$M_G$  has a coloured node, and every coloured node of $M_G$ has a coloured successor. 
  This concludes the proof. 
 \end{proof}

%\begin{coro} \label{C.almostT1} There exists a complete theory $\bfT$ 
%in a countable language and types $\bs_n$, for $n\in \omega$, such that the following holds. 
%For every $n$ there exists a
% model of $\bfT$ that omits~$\bs_n$,
%but  no model of $\bfT_4$ omits all~$\bs_n$ simultaneously. 
%\end{coro}
%
%\begin{proof} With $M_4,T=\bfT_4$, and $\bs_m$ as above,  
%$M_4$ omits $\bs_0$ by the construction. 
%On the other hand, Corollary~\ref{C.types} implies that the generic model omits all~$\bs_m$ for $m\geq 1$. As the filter $G$ needs to meet only countably many dense sets, we can construct 
%such model by recursion. 
%
%That no model of $\bfT_4$ omits all $\bs_n$  was proved in  Lemma~\ref{L.almostT1}. 
%\end{proof} 

\begin{proof}[Proof of Theorem~\ref{C.generic}] 
We need to prove that $\bfT_4$ is a complete theory in a countable language 
and that  the type $\bs_0$ is  omissible in a model of $\bfT_4$, yet  
 forced by   $\bbP_{\bfT_4}$ to be realized in~$M_G$.
Being defined as  $\Th(M_4)$, 
   $\bfT_4$ is complete. 
The type $\bs_0$ of an infinite branch 
is omitted 
in  $M_4\models\bfT_4$. 
By Lemma~\ref{L.M.generic},  
$\bbPT$ forces that all $\bs_m$ for $m\geq 1$ are omitted in~$M_G$. 
Therefore  the nodes of $M_G$ form an ill-founded tree and $M_G$ realizes~$\bs_0$. 
\end{proof}

\subsection{Proof of Theorem~\ref{T2}}
\label{S.ProofT2}

We shall  define a   complete  theory $\bfT_5$ in a countable language $L_5$ 
and  types $\bs$ and $\bt$ omissible in models of $\bfT_5$  
such that no model of $\bfT_5$ simultaneously omits both of them. 
 
 Let   $T_1$ and $T_2$ be as in \S\ref{S.T1T2}, 
 and let the language $L_*$, the model $M_4$, and the theory $\bfT_4$ be as in 
 \S\ref{S.M4}.  
 A two-sorted language~$L_5$ 
 is the expansion  of~$L_*$ 
 with sorts $X$ and $Y$. The sort $Y$ 
corresponds to $L_*$,  so that  if $N$ is 
an $L_5$-structure then~$Y^N$ is an $L_*$-structure. The sort $X$ is equipped with the following: 
\begin{enumerate}
\popcounter
\item a discrete $\{0,1\}$-metric~$d$,   
\item a unary function $g\colon X\to X$, and 
\item a unary 
function $h\colon X\to Y$. 
\pushcounter
\end{enumerate}
Both $g$ and $h$ are 
 interpreted as  1-Lipshitz functions (necessarily, since $X$ is discrete). 
This describes the language $L_5$ and we proceed to describe an $L_5$-model $M_5$. 
The model $Y^{M_5}$ is isomorphic to $M_4$. 
The set~$X^{M_5}$ is a tree isomorphic to the underlying tree of $M_4$ and 
we interpret~$h$ as the tree isomorphism function from~$X^{M_5}$ onto $Y^{M_5}$. 
Finally, the interpretation of~$g$ sends $\langle\rangle$ to itself and every 
other node of~$X^{M_5}$ to 
its immediate predecessor. 

This describes $M_5$; let 
 $\bfT_5=\Th(M_5)$. 

We proceed to define types $\bs$ and $\bt$. 
The former is 
the type of an infinite branch 
in $Y^N$; i.e. $\bs$  is  
  $\bs_0$ as defined in \S\ref{S.types-sm}. 
The type~$\bt(x)$, for $x$ of sort $X$, subsumes all  types $\bs_n(x)$ for $n\geq 1$ 
as defined in \S\ref{S.types-sm}. It asserts the following ($g^k$ is 
the $k$th iterate of $g$ for~$k\geq 1$): 
\begin{enumerate}
\popcounter
\item\label{I.t.1}   $\inf_{y\in X}  d(x,g^k(y))=0$ for all $k\geq 1$. 
\item \label{I.t.2} $\inf_{y\in X}  \max(d(x,g(y)), P_{m,n}(h(y)))=1 $ for all $m$ and $n$. 
\end{enumerate}
Informally,  \eqref{I.t.1} asserts that $x$ has infinite rank and \eqref{I.t.2} 
asserts that $h(x)$ has no coloured successors. We will prove that this is indeed so, 
after a few preliminaries. 

By Lemma~\ref{L.Def.Nodes}, the following predicates in the sort $Y$:  
\begin{align*}
 \Pgm(s):&\text{ the height of $s$ is $>m$,} \\
\Plm(s):&\text{ the height of $s$ is $\leq m$,}
\end{align*} 
are definable.

% The universe of every other model $N$ of $\bfT_5$ consists of a model $M^N$ of $\bfT$
%and a discrete set $X^N$. It is equipped with functions  $h^N\colon X^N\to M^N$ and  $g^N\colon X^N\to X^N$.  
\begin{lemma} \label{L.M5} 
The following hold for every  $N\models \bfT_5$ and $a$ and $b$ 
in~$X^N$. 
\begin{enumerate}
\popcounter
\item \label{I.T2.1} If $a\neq b$ and $h(a)$ and $h(b)$ are nodes of $Y^N$ 
then $h(a)\neq h(b)$. 
\item \label{I.T2.2} If $h(b)$ is a node of height $m+1$ then $h(g(b))=f_m(h(b))$. 
\item \label{I.T2.3} Every node of $Y^N$ is in the range of $h$. 
\item \label{I.T2.4} The set of immediate successors of $h(a)$ is equal to 
\[
\{h(c): c\in X^N\text { and }  g(c)=a\}.
\]  
\pushcounter
\end{enumerate}
\end{lemma} 

\begin{proof} 
We  prove each of  \eqref{I.T2.1}--\eqref{I.T2.4} for nodes of height $\leq m$. 

\eqref{I.T2.1} 
The following sentence is in $\Th(M_5)$ for every $m\geq 1$: 
\[
\sup_{x\in X,y\in X} \min(\Pgm(h(x)), \Pgm(h(y)), d(x,y), \frac 1m\dminus d(h(x), h(y))). 
\]
Thus in every model $N$ of $\bfT_5$, for all $m\geq 1$ and $a$ and $b$ in $X^N$ such that $\max(|h(a)|, |h(b)|)\leq m$
we have $h(a)= h(b)$ or  $d(h(a), h(b))\geq 1/m$.  This implies \eqref{I.T2.1}. 

\eqref{I.T2.2} For all $m\in \omega$ and $x\in X^N$ one of   
(i) $|h(x)|>m+1$,  (ii) $|h(x)|\leq m$,  or (iii) $h(g(x))=f_m(h(x))$ applies. Therefore  $\Th(M_5)$ contains the sentence 
\[
\sup_{x\in X} \min(\Pgmm(h(x)), \Plm(h(x)), d(h(g(x)), f_m(h(x)))
\]
and if   $N\models \bfT_5$ and  $b\in X^N$ satisfies $|h(b)|=m+1$ 
then  $ h(g(b))=f_m(h(b)$. 

\eqref{I.T2.3} For every $m\geq 1$ the following sentence is in $\Th(M_5)$: 
\[
\sup_{x\in Y} \min(\Pgm(x), \inf_{y\in X} d(h(y), x)). 
\]
Since $\Pgm(x)$ is definable, 
\eqref{I.T2.3} holds
in every model of~$\bfT_5$.

\eqref{I.T2.4}  For every $m\geq 1$ the following sentence is in $\Th(M_5)$: 
\begin{multline*}
\sup_{x\in X, y\in Y} \min(\Pgm(h(x)),\Plmm(h(x)), \Pgmm(y),\Plm(y), \\
 \frac 1m\dminus d(f_m(y), h(x)),
 \inf_{z\in X} \max(d(h(z),y), d(g(z),x))). 
\end{multline*} 
Again, 
 \eqref{I.T2.4} holds in every model of $\bfT_5$. 
\end{proof}

\begin{lemma} \label{L.T5st} Suppose that   
$N\models \bfT_5$ and let $N'=Y^N$.   
 If $a\in N$ realizes~$\bt$ then  $b=h(a)$ 
satisfies $\rho_{N'}(b)\geq \omega$ and $b$ has no coloured successors in $N'$. 
\end{lemma} 

\begin{proof}  
Suppose $a$ realizes $\bt$ and let $b=h(a)$. 
By \eqref{I.t.1} for  every $k\geq 1$ there exists $y\in X^N$ such that $g^k(y)=a$
and therefore 
 $\rho_N(a)\geq \omega$. 
 By  Lemma~\ref{L.M5} we have $\rho_{Y^N}(b)\geq \omega$. 
 On the other hand,    \eqref{I.t.2} implies that    
for all $y$ such that $g(y)=a$, the  node 
$h(y)$  is not coloured.  
 By  Lemma~\ref{L.M5}, $b$ has no coloured successor. 
%
%\eqref{I.T5st.2} Suppose $N\models \bfT_5$ and not all of its nodes are coloured. 
% by any $a\in X^N$ such that one of the following 
%holds. 
%\begin{enumerate}
%\item $h(a)$ belongs to the top part of $Y^N$ 
%and $\rho_N(a)\geq \omega$,  or 
%\item $h(a)$ is a terminal node of the bottom part of $Y^N$.    
%\end{enumerate} 
\end{proof} 
%However, this is of no concern to us as if  $N$ omits $\bs$ then $M^N$ has no branches. 

\begin{lemma} 
Each of the types $\bs$ and $\bt$ is omissible in a model of $\bfT_5$.
\end{lemma} 

\begin{proof}  
Type $\bs$ is clearly omitted in $M_5$.  Type $\bt$ is 
omitted in a $\bbP_{\bfT_5}$-generic model. 
A proof of the latter fact follows  the proofs of the Extension Lemma (Lemma~\ref{L.types.2}), 
Lemma~\ref{L.M.generic} and Theorem~\ref{C.generic} almost verbatim. Alternatively, one could
 prove that the $L_*$-part 
of the $\bbP_{\bfT_5}$-generic model is
$\bbPT$-generic and use these facts directly. 
\end{proof} 

It remains to check that $\bs$ and $\bt$ are not simultaneously omissible in a model  
of $\bfT_5$. If $N\models \bfT_5$ 
 omits $\bs$ then $X^N$ omits~$\bs_0$. 
 Therefore all elements of $X^N$ are nodes.
 In particular the bottom part of $X^N$ is well-founded, and  any of its terminal nodes
   realizes $\bt$ in~$N$. 

This completes the proof of Theorem~\ref{T2}. 

\section{Concluding remarks} \label{S.CR}

We don't know whether \eqref{T0.2} of Theorem~\ref{T0} (i.e.  Theorem~\ref{P2}) is sharp. 

\begin{question} \label{Q.Stability} 
What are the possible complexities of the set of types 
 omissible in a model of a
 complete theory~$\bfT$ in a countable language? 
In particular, 
can this set be $\bSigma^1_2$-complete? 
\end{question}

According to  \cite{caicedo2014omitting} and  
\cite[Definition~4.12]{eagle2013omitting}, a type $\bt(\bar x)$ is 
metrically principal over a theory $\bfT$ (we consider the case when $L$ is the 
fragment consisting of all finitary sentences) if and only if for every $\delta>0$ 
the type $\bt^\delta(\bar x)$, asserting that every finite subset of $\bt$ 
is realizable by an $n$-tuple within $\delta$ of $\bar x$, is principal.

For example, type $\bt$ defined in the proof of Proposition~\ref{P1} 
is metrically principal over $\bfT^S$ if the tree $S$ has height $\omega$. 
This is because $\bt_{1/n}$ is realized by any node of $S$ that is not an end-node. 
This gives an example of an omissible  metrically principal type. 
A simple argument shows that a  \emph{complete} metrically principal type cannot be omissible.

\subsection{\cstar-algebras} \label{S.C*} 

The original motivation for this study came from the model-theoretic study of \cstar-algebras. 
Many important properties of \cstar-algebras are axiomatizable (see \cite[Theorem~2.5.1]{Muenster}), and numerous non-axiomatizable properties 
are \udt{} (see \cite[Definition~5.7.1]{Muenster} and the discussion afterwards). 
The answers to some of the most prominent open problems in the  theory of \cstar-algebras
depend on  whether  \cstar-algebras with these properties (in particular, 
nuclear---also known as amenable---\cstar-algebras) can be constructed in a novel way. 
 Therefore Theorem~\ref{T.Uniform}, 
Theorem~\ref{T.Uniform.1}, Proposition~\ref{P4},  Proposition~\ref{P4+} and Corolary~\ref{C.ec} open 
 possibilities for constructing \cstar-algebras with prescribed first-order properties in these classes (see \cite{goldbring2017enforceable}, \cite{goldbring2017robinson}). 
Also, 
 some of the deepest recent results on classification  of \cstar-algebras have equivalent formulation in the language of (metric) first-order logic (see \cite{EllTo:Regularity},
the introduction to \cite{sato2014nuclear}, and \cite{Fa:Logic} or \cite{Muenster}).

We have   a machine for construction  of   \cstar-algebras with properties prescribed 
by a given  theory. 
These are generic algebras obtained by Henkin  construction as described in  \S\ref{S.forcing}
and \S\ref{S.Uniform}. Although they are assembled from finite pieces corresponding to conditions of $\bbPT$ or one of its variations, 
they are not obviously obtained from matrix algebras and abelian algebras by applying basic operations 
of taking inductive limits, crossed products by $\bbZ$, stable isomorphisms, quotients,  
extensions, hereditary subalgebras, or KK-equivalence
(cf. the bootstrap class problem, \cite[IV.3.1.16 and V.1.5.4]{Black:Operator}). 
At present no method for assuring that the \cstar-algebras obtained by using the Henkin construction 
do not belong to the (large or small) bootstrap class is known. Results of \cite{Muenster} (combined with 
 \S\ref{S.forcing} and \S\ref{S.Uniform}) reduce several  prominent open problems on classification of 
  \cstar-algebras to problems about the existence of theories with certain properties.

In \cite{Kec:C*} Kechris  defined a Borel space of \cstar-algebras and proved that the nuclear \cstar-algebras form 
a Borel subset. We give a generalization of this result. Recall that a Borel structure on the space
of models was defined in \S\ref {S.Borel.Models}. Although this space is different from one used by Kechris, 
these representations of space of separable \cstar-algebras are equivalent (\cite{FaToTo:Descriptive}). 

\begin{coro} \label{C.C*} The following sets  of \cstar-algebras are Borel subsets of the standard Borel space of 
\cstar-algebras: UHF, AF, nuclear, nuclear dimension $\leq n$ for $n\leq \aleph_0$, decomposition rank $\leq n$ for 
$n\leq \aleph_0$, tracially AF, simple. 
\end{coro} 

\begin{proof} 
Since each of the sets of \cstar-algebras listed above is \udt{} by \cite{Mitacs2012} and  \cite{Muenster}, the conclusion follows
by Proposition~\ref{P.UDT-Borel}. 
\end{proof}

%\subsection{Descriptive stability?}
The class of all cardinals $\kappa$ for which the set of types over any set
of cardinality $\kappa$ in a model of theory $\bfT$ has cardinality at most $\kappa$
is an important invariant of a   first-order theory (see~\cite{Sh:Classification}; 
see also  \cite[\S 14]{BYBHU}, \cite[\S5]{FaHaSh:Model2} for the metric version). 
To  a complete theory $\bfT$ in a countable language~$L$ in the logic of metric
structures one can 
associate descriptive complexities of 
distinguished sets of (not necessarily complete)  types. 
Our results suggest that these invariants  provide  
nontrivial information about~$\bfT$. 
 In particular, it is plausible 
  that in the case when $\bfT$ is a natural  theory (e.g. the theory of a \cstar-algebra)
   the set of types omissible in a model of~$\bfT$ is Borel.  
 Each theory used in our counterexamples interprets a nontrivial theory of trees (e.g. the Baire space or the tree $T_2$, see \S\ref{S.Baire} and \S\ref{S.T1T2}). 
Is this a necessary condition for pathological behaviour of metric theories? 

\begin{problem} Find a model-theoretic characterization of 
 complete theories~$\bfT$ in a separable 
language  with each the following properties.   
\begin{enumerate}
\item The set of complete types omissible in a model of~$\bfT$ is 
Borel. 
\item If two types are omissible in models of $\bfT$, then 
they are simultaneously omissible in a model of $\bfT$. 
\item If  types $\bt_n$, for $n\in \omega$, are such that 
for any $k\in \omega$ the types $\bt_n$, for $n<k$, are simultaneously 
omissible in a model of $\bfT$, then 
all of these types  are simultaneously omissible in a model of $\bfT$. 
\end{enumerate}
\end{problem}

\bibliographystyle{acm}
\bibliography{ifmainbib}

\end{document}